\documentclass{amsart}

\usepackage[top=1in, bottom=1in, left=1in, right=1in]{geometry}
\usepackage{graphicx,hyperref,multicol,amsthm,amssymb,tikz,mathtools,xcolor,tikz,tikz-3dplot,ytableau}

\newcommand{\Wsp}[0]{\mathcal{W}^{\mathfrak{sp}}_{\infty}}
\newcommand{\Wso}[0]{\mathcal{W}^{\mathfrak{so}_2}_{\infty}}
\newcommand{\vir}[0]{\text{Vir}^c}
\newcommand{\aff}{V^k(\mathfrak{sl}_2)}
\newcommand{\nlcalg}{\cL^{\fr{so}_2}_{\infty}}

\newcommand{\Winf}{\mathcal{W}_{\infty}}
\newcommand{\Wev}{\mathcal{W}^{\text{ev}}_{\infty}}

\numberwithin{equation}{section}

\newtheorem{lemma}{Lemma}[section]
\newtheorem{remark}{Remark}[section]
\newtheorem{proposition}{Proposition}[section]
\newtheorem{theorem}{Theorem}[section]
\newtheorem{example}{Example}[section]
\newtheorem{corollary}{Corollary}[section]

\newtheorem{conjecture}{Conjecture}[section]

\newcommand{\B}[1]{\textbf{#1}}

\newcommand{\T}[1]{\text{#1}}

\newcommand{\fr}[1]{\mathfrak{#1}}

\def\cO{{\cal O}}

\def\cA{{\cal A}}

\def\ra{\rightarrow}

\def\cA{{\mathcal A}}

\def\cC{{\mathcal C}}
\def\cD{{\mathcal D}}
\def\cE{{\mathcal E}}
\def\cF{{\mathcal F}}

\def\cH{{\mathcal H}}
\def\cI{{\mathcal I}}

\def\cL{{\mathcal L}}

\def\cN{{\mathcal N}}
\def\cO{{\mathcal O}}

\def\cS{{\mathcal S}}

\def\cV{{\mathcal V}}
\def\cW{{\mathcal W}}

\def\ga{{\mathfrak a}}
\def\gb{{\mathfrak b}}

\def\gg{{\mathfrak g}}

\def\gl{{\mathfrak l}}

\def\go{{\mathfrak o}}
\def\gp{{\mathfrak p}}

\def\gs{{\mathfrak s}}
\def\gt{{\mathfrak t}}

\title{New universal vertex algebras as glueings of the basic ones}

\author{Thomas Creutzig}
\address{Department Mathematik, FAU Erlangen}
\email{creutzigt@math.fau.de}
\thanks{T. Creutzig is supported by DFG project Projektnummer 551865932.}

\author{Vladimir Kovalchuk}
\address{Department Mathematik, FAU Erlangen}
\email{v.dexter97@gmail.com}
\thanks{V. Kovalchuk is supported by the Alexander von Humboldt Foundation.}

\author{Andrew R. Linshaw}
\address{Department of Mathematics, University of Denver}
\email{andrew.linshaw@du.edu}
\thanks{A. Linshaw is supported by NSF Grant DMS-2401382 and Simons Foundation Grant MPS-TSM-00007694.}

\begin{document}
	\maketitle

	\pagestyle{plain}
	
	\noindent {ABSTRACT. There are three universal $2$-parameter vertex algebras $\cW_{\infty}$, $\cW^{\text{ev}}_{\infty}$, and $\cW^{\gs\gp}_{\infty}$ which are freely generated of types $\cW(2,3,4,\dots)$, $\cW(2,4,6,\dots)$, and $\cW(1^3, 2, 3^3, 4,\dots)$, respectively. They serve as classifying objects for vertex algebras with these generating types satisfying mild hypotheses. Their $1$-parameter quotients are expected to be the building blocks of all $\cW$-algebras of classical Lie types. Furthermore, such $\cW$-algebras are expected to be organized into families that are governed by new universal $2$-parameter vertex algebras, which are themselves glueings, i.e., extensions of tensor products, of copies of $\cW_{\infty}$ in type $A$ (together with a Heisenberg algebra), and copies of $\cW^{\text{ev}}_{\infty}$ and $\cW^{\gs\gp}_{\infty}$ in types $B$, $C$, and $D$. We denote these universal objects by $\cW^{X,S,M}_{\infty}$, where $X$ denotes the Lie type (either $A$, $C$, or $BD$ since types $B$ and $D$ can be treated uniformly), and $S$, $M$ are sets of positive integers that determine certain families of partitions. More precisely, for a partition $P = (n_0^{m_0}, n_1^{m_1},\dots, n_{t}^{m_t})$ of $N = \sum_{i=0}^t n_i m_i$ consisting of $m_i$ parts of size $n_i$, where $n_0> n_1 > \cdots > n_t \geq 2$, $M = \{m_0,\dots, m_t\}$ is the set of {\it multiplicities}, and $S = \{d_1,\dots, d_t\}$ is the set of {\it height differences} $d_{i+1} = n_i - n_{i+1}$. In all cases, $\cW^{X,S,M}_{\infty}$ should admit several families of $1$-parameter quotients indexed by $\mathbb{N} \times \mathbb{N}$ which arise as $\cW$-algebras or (orbifolds of) affine cosets of $\cW$-algebras, as well as some families of diagonal cosets indexed by $\mathbb{N}$, which we call GKO cosets. After introducing this general conjectural picture, we will construct the first nontrivial example $\cW^{\fr{so}_2}_{\infty}:=\cW^{BD, \emptyset, \{2\}}_{\infty}$, which is a glueing of two copies of $\cW^{\text{ev}}_{\infty}$. It is freely generated of type $\cW(1,2^3,3,4^3,\dots)$ and its $1$-parameter quotients include $\cW^{\ell}(\gs\gp_{4n}, f_{2n, 2n})$ and $\cW^{\ell}(\gs\go_{2(2n+1)}, f_{2n+1, 2n+1})$, which are of rectangular type. This result has several applications including the existence of embeddings of hook-type $\cW$-algebras inside more complicated $\cW$-algebras, and the strong rationality of $\cW_{\frac{4m+2n-1}{4m}-h^{\vee}}(\go\gs\gp_{1|4m}, f_{2m,2m})$ for $m\geq 2n-1$ and $m,2n-1$ coprime. Most importantly, it is a proof of concept that such glueings of the basic universal objects can be constructed.}
	
	\section{Introduction} $\cW$-algebras are an important class of vertex algebras that have been studied in both the mathematics and physics literature for nearly 40 years. Associated to a Lie (super)algebra $\mathfrak{g}$ and a nilpotent element $f$ in the even part of $\mathfrak{g}$, is the $\cW$-algebra $\cW^k(\mathfrak{g}, f)$ at level $k\in \mathbb{C}$. They are defined via the generalized Drinfeld-Sokolov reduction \cite{KRW}, and are a common generalization of affine vertex algebras and the Virasoro algebra. When $f$ is a principal nilpotent, $\cW^k(\mathfrak{g},f)$ is called a principal $\cW$-algebra and is denoted by $\cW^k(\mathfrak{g})$. They appear in many settings including integrable systems \cite{B89,BM,GD,DS}, conformal field theory to higher spin gravity duality \cite{GG}, Alday-Gaiotto-Tachikawa correspondence \cite{AGT,SV,BFN}, and the quantum geometric Langlands program \cite{Fre07,Gai16,AFO,CG,FG,AF}.

Principal $\cW$-algebras satisfy Feigin-Frenkel duality, which is the isomorphism $\cW^k(\gg) \cong \cW^{k'}(^L \gg)$. Here $^L\gg$ is the Langlands dual Lie algebra, $h^{\vee}$,  $^L h^{\vee}$ are the dual Coxeter numbers of $\gg$, $^L\gg$, and $(k+h^{\vee})(k' + ^Lh^{\vee}) = r$ where $r$ is the lacing number of $\gg$ \cite{FF}. For $\gg$ simply-laced, there is another duality called the {\it coset realization} which was proven \cite{ACL}. For generic $\ell$, we have a vertex algebra isomorphism 
	\begin{equation} \label{eq:cosetrealization} \cW^{\ell}(\gg)\cong  \text{Com}(V^{k+1}(\gg),V^k(\gg)\otimes L_1(\gg)),\qquad \ell +h^{\vee}=\frac{k+h^{\vee}}{(k+1)+ h^{\vee}},\end{equation} which descends to an isomorphism of simple vertex algebras $\cW_{\ell}(\gg)\cong \text{Com}(L_{k+1}(\gg), L_k(\gg)\otimes L_1(\gg))$ for all admissible levels $k$. This generalizes the Goddard-Kent-Olive (GKO) construction of the Virasoro algebra \cite{GKO} (the case $\gg = \gs\gl_2$), so we call the cosets on the right hand side of \eqref{eq:cosetrealization} {\it GKO cosets}. When $k$ is an admissible level for $\widehat{\gg}$, $\ell$ is a nondegenerate admissible level, so $\cW_{\ell}(\gg)$ is strongly rational (i.e., lisse and rational) \cite{Ar1, Ar2}. In type $A$, note that the rank $n$ $bc$-system $\cE(n)$ is a conformal extension of $\cH \otimes L_1(\gs\gl_n)$ where $\cH$ is a Heisenberg algebra, and that $\text{Com}(V^{k+1}(\gs\gl_n),V^k(\gs\gl_n)\otimes L_1(\gs\gl_n))$ can be replaced with
\begin{equation} \label{GKOalternate} \text{Com}(V^{k+1}(\gg\gl_n),V^k(\gs\gl_n)\otimes \cE(n)).\end{equation} Therefore we often consider diagonal cosets of the form \eqref{GKOalternate} where one of the factors on the right is a free fermion algebra, $\beta\gamma$-system, or a tensor product of these algebras, to be appropriate analogues of GKO cosets.

The {\it Gaiotto-Rap\v{c}\'ak triality conjecture} is a common generalization of both Feigin-Frenkel duality and the coset realization of principal $\cW$-algebras in type $A$. In \cite{GR}, Gaiotto and Rap\v{c}\'ak introduced a family of vertex algebras $Y_{N_1, N_2, N_3}[\psi]$ called $Y$-algebras, which are indexed by three integers $N_1, N_2, N_3\geq 0$ and a complex parameter $\psi$. They are associated to interfaces of twisted $\cN=4$ supersymmetric gauge theories. These interfaces satisfy a permutation symmetry which suggests a triality of isomorphisms of $Y$-algebras. In the case when one of the labels is zero, these $Y$-algebras arise as affine cosets of certain $\cW$-(super)algebras known as {\it hook-type}. The conjecture when one label is zero was proved by the first and third authors in \cite{CL4} by explicitly realizing the $Y$-algebras as $1$-parameter quotients of the universal $2$-parameter vertex algebra $\cW_{\infty}$ of type $\cW(2,3,\dots)$. This is a classifying object for vertex algebras of type $\cW(2,3,\dots, N)$ for some $N$ satisfying some mild hypothesis which was known to physicists since the early 1990s, and was constructed rigorously by the third author in \cite{Lin}.

In \cite{GR}, Gaiotto and Rap\v{c}\'ak also introduced the orthosymplectic $Y$-algebras and conjectured a similar triality of isomorphisms which includes as special cases Feigin-Frenkel duality, as well as the coset realization \eqref{eq:cosetrealization} for type $D$, and a different coset realization for principal $\cW$-algebras of types $B$ and $C$ as well as $\go\gs\gp_{1|2n}$. This conjecture was proven by the first and third authors in \cite{CL4} by realizing these algebras explicitly as simple $1$-parameter quotients of the universal $2$-parameter vertex algebra $\cW^{\text{ev}}_{\infty}$ of type $\cW(2,4,6,\dots)$. This algebra was also known to physicists \cite{CGKV}, and was constructed by Kanade and the third author in \cite{KL}.

Recently, it has been conjectured that the $Y$-algebras are the building blocks for {\it all} $\cW$-algebras in type $A$ in the sense that any such $\cW$-algebra is an extension of a tensor product of finitely many $Y$-algebras; see \cite[Conjecture B]{CFLN}. This is based on \cite[Conjecture A]{CFLN}, which says that the quantum Drinfeld-Solokov reduction can be carried out in stages; see \cite{GJ} for a similar conjecture in the setting of finite $\cW$-algebras. In \cite{CFLN}, this conjecture was proven at the level of graded characters and was also verified by computer for all $\cW$-algebras in type $A$ of rank at most $4$. In \cite{FFFN}, a more complete picture connecting all the $\cW$-algebras of $\gs\gl_4$ via partial and inverse reduction was given, and it was shown that the total and iterated reduction functors are isomorphic on the Kazhdan-Lusztig category.

In \cite{CKoL2}, it was observed that the orthosymplectic $Y$-algebras are some, but not all, of the building blocks needed for $\cW$-algebras of types $B$, $C$, and $D$. The nilpotents $f_{2n+1,2n+1} \in \gs\gp_{2(2n+1)}$ and $f_{2n, 2n} \in \gs\go_{4n}$ are {\it indecomposable} in a certain sense, and the corresponding $\cW$-algebras, which have generating type $\cW(1^3, 2, 3^2,\dots)$ are not naturally extensions of two orthosymplectic $Y$-algebras. However, there does exist a new universal $2$-parameter vertex algebra $\cW^{\gs\gp}_{\infty}$ of type $\cW(1^3, 2, 3^3, 4,\dots)$ whose $1$-parameter quotients, together with the $1$-parameter quotients of $\cW^{\text{ev}}_{\infty}$, should be all the building blocks for $\cW$-algebras of types $B,C,D$. A precise conjecture was given as \cite[Conjecture 3.2]{FKN}.

\subsection{New universal vertex algebras as glueings}
In order to motivate our results, we will present a conjectural picture of how $\cW$-algebras of classical Lie types can be organized into families that are governed by new universal $2$-parameter vertex algebras that are themselves glueings of finitely many copies of the basic ones, namely, $\cW_{\infty}$ in type $A$ (together with a Heisenberg algebra), and $\cW^{\text{ev}}_{\infty}$ and $\cW^{\gs\gp}_{\infty}$ in types $B$, $C$, and $D$. The examples corresponding to rectangular nilpotents in type $A$ were constructed recently by Gaiotto, Rap\v{c}\'ak and Zhou in \cite{GRZ}, although their uniqueness as $2$-parameter vertex algebras, which is essential for our applications, was not proven. In general, the $1$-parameter quotients of these $2$-parameter vertex algebras include $\cW$-algebras, affine cosets of certain $\cW$-(super)algebras, and certain diagonal cosets that are analogues of the GKO cosets \eqref{GKOalternate}. In type $A$, this picture was partially described \cite{CFLN}, but the GKO cosets were not given. One can also conjecture similar universal $2$-parameter vertex superalgebras that govern $\cW$-superalgebras of classical types, but for brevity we omit them in this paper.

For $t \geq 0$, consider a set of positive integers $n_0,n_1,\dots, n_t$ with $n_0 > n_1 >  \cdots  > n_t \geq 2$, and consider the partition $P = (n_0^{m_0}, n_1^{m_1},\dots, n_{t}^{m_t})$ of $N = \sum_{i=0}^t n_i m_i$ consisting of $m_i$ parts of size $n_i$. Let $f_P \in \gs\gl_N$ be the nilpotent corresponding to $P$. The strong generating type of $\cW^k(\gs\gl_N, f_P)$ is given by \cite[Proposition 2.1]{CFLN}, and it depends only on the height differences $d_{i+1} = n_i - n_{i+1}$ for $i=0,1,\dots, t-1$ in the following sense. If we fix the smallest height $n_t$, all others $n_0,\dots, n_{t-1}$ are determined from the height differences, and the generating type of  $\cW^k(\gs\gl_N, f_P)$ up to any fixed conformal weight is independent of $n_t$ for $n_t$ sufficiently large.

Note that the affine subalgebra of $\cW^k(\gs\gl_N, f_P)$ is $\cH(t) \otimes \big(\bigotimes_{i=0}^t V^{k_i}(\gs\gl_{m_i})\big)$ where $\cH(t)$ is the rank $t$ Heisenberg algebra, and the levels $k_i$ are all related linearly to $k$. The generating fields are also organized into modules for the affine subalgebra of $\cW^k(\gs\gl_N, f_P)$, and this module structure also depends only on the height differences for $n_t$ sufficiently large. Accordingly, one expects that there is a universal $2$-parameter vertex algebra admitting all the above $\cW$-algebras as $1$-parameter quotients; see Conjecture \ref{universal:typeA}. It depends only on the sets $S = \{d_1, \dots, d_t\}$ and $M = \{m_0,\dots, m_t\}$ of height differences and multiplicities, respectively, so we denote it by $\cW^{A,S,M}_{\infty}$. In the case $S = \emptyset$ and $M = \{1\}$, $\cW^{A,S,M}_{\infty}$ is exactly the algebra $\cW_{\infty}$ constructed in \cite{Lin}. In fact, there are $2$ families of $1$-parameter vertex algebras, each indexed by $\mathbb{N} \times \mathbb{N}$, which we also expect to arise as $1$-parameter quotients of $\cW^{A,S,M}_{\infty}$. These generalize the family $\cW^k(\gs\gl_N, f_P)$ and arise as affine cosets of $\cW$-(super)algebras, and are given by \eqref{Asecond} and \eqref{Athird}. All these vertex algebras should be obtained from $\cW^{A,S,M}_{\infty}$ by taking the quotient by an ideal whose weight zero component is an ideal in the ring of parameters. These ideals correspond to curves in the parameter space called {\it truncation curves} which were described explicitly in the case of $\cW_{\infty}$ in \cite{CL3}. In addition, there are two families of GKO cosets, each indexed by $\mathbb{N}$, which we define to be diagonal affine cosets of $\cW$-algebras for smaller partitions tensored with free field algebras. These are given by \eqref{AGKO}. They have the property that the affine subalgebra $V^{k_t}(\gs\gl_{m_t})$ corresponding to the smallest rectangular block $n_t^{m_t}$ in $P$, has constant level $n$ which is either a positive or negative integer. In particular, the truncation curve corresponds to the ideal $(k_t -n)$ in these examples.

Our picture of unifying $2$-parameter vertex algebras for $\cW$-algebras of type $A$ is closely connected the work of Brundan and Kleshchev \cite{BK} in which all finite $\cW$-algebras of type $A$ are realized as truncations of shifted Yangians \cite{BK}. It was first observed by Ragoucy and Sorba \cite{RS} that the finite $\cW$-algebra $\cW_{\text{fin}}(\gs\gl_{mn}, f_{n^m})$ corresponding to the rectangular partition with $m$ parts of size $n$, is isomorphic to the level $n$ truncation of the Yangian $Y_m:=Y(\gg\gl_m)$. In \cite{BK}, Brundan and Kleshchev introduced the shifted Yangians $Y_m(\sigma)$ and showed that for a partition $P$ of $N$ with $m$ parts whose height differences are determined by the shift matrix $\sigma$, the finite $\cW$-algebra $\cW_{\text{fin}}(\gs\gl_N, f_P)$ is a truncation of $Y_m(\sigma)$. 

In the rectangular case $f_P$ for $P = (n^m)$, a generalization to the affine setting was proven by Kodera and Ueda in \cite{KU}: they constructed a map (after a suitable completion) from the affine Yangian $Y^{\text{aff}}_m: = Y(\hat{\gg\gl}_m)$ to the mode algebra of the corresponding rectangular $\cW$-algebra $\cW^k(\gs\gl_{nm}, f_{n^m})$, for all $n$. In \cite{GRZ}, Gaiotto, Rap\v{c}\'ak and Zhou related the mode algebra of the $2$-parameter vertex algebra $\cW^{(K)}_{\infty}$, which in our notation is $\cW^{A,\emptyset, \{m\}}_{\infty}$ with $m = K$, with the affine Yangian $Y(\hat{\gg\gl}_K)$. They also related the Zhu algebra $A(\cW^{(K)}_{\infty})$ to the finite Yangian $Y_K$.


We propose the following generalization.
\begin{conjecture} \label{universal:typeAYangian}  Let $P = (n_0^{m_0}, n_1^{m_1},\dots, n_{t}^{m_t})$ be the partition of $N = \sum_{i=0}^t n_i m_i$ as above, where $n_0 > n_1 > \cdots > n_t \geq 2$, $d_{i+1} = n_i - n_{i+1}$ for $i = 0,\dots, t-1$. Let $S = \{d_1,\dots, d_t\}$ and $M = \{m_0,\dots, m_t\}$ as above, and let $m = \sum_{i=1}^t m_i$ be the number of parts of $P$, and let $\sigma$ be the corresponding shift matrix. There is a suitable notion of shifted affine Yangian $Y^{\text{aff}}_m(\sigma)$ with the following properties:
\begin{enumerate}
\item The mode algebra $U(\cW^{A,S,M}_{\infty})$ after suitable completion and localization, is isomorphic to $Y^{\text{aff}}_m(\sigma)$.
\item The Zhu algebra $A(\cW^{A,S,M}_{\infty})$ is isomorphic to $Y_{m}(\sigma)$.
\end{enumerate}
\end{conjecture}

Together with Conjecture \ref{universal:typeA}, this would imply that there are surjective maps (after completion) from $Y^{\text{aff}}_m(\sigma)$ to the mode algebras of the vertex algebras \eqref{Afirst}, \eqref{Asecond}, \eqref{Athird}, and \eqref{AGKO}. In particular, the GKO coset $\text{Com}(V^{k'+m_t}(\gg\gl_s), \cW^k(\gs\gl_{N_{-2,s}}, f_{S, M,-2,s}) \otimes \cE(m_t s))$ given by \eqref{AGKO} contains the simple affine subVOA $L_s(\gs\gl_{m_t})$, so $Y^{\text{aff}}_m(\sigma)$ would then admit a nontrivial quotient containing the (completed) mode algebra of $L_s(\gs\gl_{m_t})$. Similarly, $Y_m(\sigma)$ would then admit an interesting quotient containing the (finite-dimensional) Zhu algebra $A(L_s(\gs\gl_{m_t}))$.

In orthosymplectic types, we conjecture the existence of similar $2$-parameter universal vertex algebras $\cW^{C,S,M}_{\infty}$ and $\cW^{BD,S,M}_{\infty}$, which also depend on sets $S$ and $M$ with some restrictions on the multiplicities $m_i \in M$ to account for the fact that nilpotents in type $C$ correspond to partitions whose odd parts have even multiplicities, and partitions in types $B$ and $D$ correspond to partitions whose even parts have even multiplicities. These are described in Conjectures \ref{universal:typeC} and Conjecture \ref{universal:typeBD}. Each of these will admit $8$ $\mathbb{N}\times \mathbb{N}$ families of $1$-parameter quotients which are either $\cW$-algebras or (orbifolds of) affine cosets of $\cW$-algebras; these are given by \eqref{Cfirst}, \eqref{Csecond} and \eqref{Cthird} for $\cW^{C,S,M}_{\infty}$, and by \eqref{BDfirst}, \eqref{BDsecond} and \eqref{BDthird} for $\cW^{BD,S,M}_{\infty}$. Similarly, there are $4$ families of GKO cosets for both $\cW^{C,S,M}_{\infty}$ and $\cW^{BD,S,M}_{\infty}$ which are each indexed by $\mathbb{N}$; see \eqref{CGKO} and \eqref{BDGKO}. As in type $A$, the GKO cosets have the property that the level of the affine subVOA corresponding to the lowest rectangular block is a constant in $\frac{1}{2} \mathbb{Z}$ for $\cW^{C,S,M}_{\infty}$, and in $\mathbb{Z}$ for $\cW^{BD,S,M}_{\infty}$. 

Generalizations of the work of Brundan and Kleschchev \cite{BK} to types $B$, $C$, and $D$ have been studied in special cases in \cite{Brown,DSKV}, and the general picture appeared in a recent work of Lu, Peng, Tappeiner, Topley and Wang \cite{LPTTW}. In this paper, the shifted twisted Yangians were introduced, and for all even nilpotents in types $B$ and $C$ (and conjecturally in type $D$), the corresponding finite $\cW$-algebras were shown to be truncations of shifted twisted Yangians. For rectangular nilpotents of type $D$, there is an affine analogue which was proven by Ueda \cite{U21}: he constructed a map (after completion) from the twisted affine Yangian $Y(\hat{\gs\gp}_{2m})$ to the mode algebra of $\cW^k(\gs\go_{4nm}, f_{(2n)^{2m}})$. These $\cW$-algebras are expected to be $1$-parameter quotients of $\cW^{C,\emptyset, \{2m\}}_{\infty}$, and by analogy with \cite[Theorem 17]{GRZ}, one expects that after completion and localization, $Y(\hat{\gs\gp}_{2m})$ is isomorphic to the mode algebra of $\cW^{C,\emptyset, \{2m\}}_{\infty}$. Similar results for twisted affine Yangians in type $D$ corresponding to rectangular nilpotents also appear in \cite{HU25,UtypeD25}. More generally, one expects that after completion and localization, the mode algebras of $\cW^{C,S,M}_{\infty}$ and $\cW^{BD,S,M}_{\infty}$ will admit isomorphisms of suitable shifted twisted affine Yangians in orthosymplectic types, and each $1$-parameter quotient of $\cW^{C,S,M}_{\infty}$ and $\cW^{BD,S,M}_{\infty}$ will give rise to an ideal in the shifted twisted affine Yangian. As in type $A$, we expect the shifted twisted Yangians in \cite{LPTTW} to arise as Zhu algebras of the vertex algebras $\cW^{C,S,M}_{\infty}$ and $\cW^{BD,S,M}_{\infty}$. Finally, $\cW^{C,S,M}_{\infty}$ and $\cW^{BD,S,M}_{\infty}$ are expected to admit the GKO cosets \begin{equation*} \begin{split} & \text{Com}(V^{k'+m_t}(\gs\gp_{2s}), \cW^k(\go\gs\gp_{N_{-1,0}|2s}, f^{2BD}_{S,M,-1,0}) \otimes \cF(2m_t s)),
\\ & \text{Com}(V^{k' + m_t}(\gs\go_s), \cW^k(\go\gs\gp_{s|N_{-1,0}}, f^{2C}_{S,M,-1,0}) \otimes \cF(m_t s))^{\mathbb{Z}_2},\end{split} \end{equation*} respectively, as $1$-parameter quotients; see \eqref{CGKO} and \eqref{BDGKO}. These contain simple affine subVOAs $L_s(\gs\gp_{2m_t})$, and $L_s(\gs\go_{m_t})$, respectively. This suggests that the corresponding shifted twisted Yangians admit interesting quotients containing the finite-dimensional Zhu algebras $A(L_s(\gs\gp_{2m_t}))$ and $A(L_s(\gs\go_{m_t}))$, respectively.


In addition to providing insight into the structure of $\cW$-algebras, constructing the universal objects $\cW^{X,S,M}_{\infty}$ for $X = A,C,BD$, has at least three applications to vertex algebras which have been discussed for the basic examples $\cW_{\infty}$, $\cW^{\text{ev}}_{\infty}$, and $\cW^{\gs\gp}_{\infty}$ in our previous work.
\begin{enumerate}
\item Under some conditions on $S$ and $M$ (and especially for the super analogues of $\cW^{X,S,M}_{\infty}$ which are unifying algebra for $\cW$-superalgebras), we expect that there will be isomorphisms between certain $1$-parameter quotients of these objects, which are analogues of the Gaiotto-Rap\v{c}\'ak trialities. In type $A$, some conjectural examples are given by \cite[Conjectures 3.6 and 3.8]{CFLN}. 
\item The property that $\cW^{X,S,M}_{\infty}$ is a glueing of the basic $2$-parameter vertex algebras $\cW_{\infty}$, $\cW^{\text{ev}}_{\infty}$, and $\cW^{\gs\gp}_{\infty}$ yields many $1$-parameter embeddings of $\cW$-algebras inside larger $\cW$-algebras, which are expected to induce conformal embeddings of simple quotients at special levels. The problem of discovering and classifying such conformal embeddings and determining when the larger $\cW$-algebra collapses to the smaller one, vastly generalizes the program of classifying conformal embeddings of affine vertex algebra inside $\cW$-algebras; see for example \cite{AKMPP17,AKMPP18,AMP23,AEM24,AACLMPP25}. \item Distinct $1$-parameter quotients of a universal $2$-parameter vertex algebra are in bijection with their truncation curves, and we expect that at intersection points of these curves, we have nontrivial isomorphisms between the simple quotients. The GKO cosets play a special role because they often have many quotients which are strongly rational; see Examples \ref{rectA}, \ref{rectC}, and \ref{rectBD} for the rectagular cases, as well as Example \ref{Asecondexample} for a non-rectangular case in type $A$. Pointwise coincidences between GKO cosets and other algebras sometimes allow rationality results to be discovered and proven, both for $\cW$-algebras at non-admissible levels and for $\cW$-superalgebras; see for example \cite{CKoL2,CFKLN}. For this application, it is essential that $\cW^{X,S,M}_{\infty}$ has exactly $2$ free parameters, but even when this has been proven, there are still substantial challenges to obtaining the desired isomorphisms from intersection points; see the discussion in Subsection \ref{subsection:voaring}.
\end{enumerate}

\subsection{Main result} In this paper, we will construct the first nontrivial example of such a universal $2$-parameter vertex algebra which is a glueing of basic ones. This will be $\cW^{BD, \emptyset, \{2\}}_{\infty}$, which is a companion to the algebra $\cW^{C, \emptyset, \{2\}}_{\infty} = \cW^{\gs\gp}_{\infty}$ constructed in \cite{CKoL2}. Since its affine subalgebra is of type $\gs\go_2$, we will denote $\cW^{BD, \emptyset, \{2\}}_{\infty}$ by $\cW^{\gs\go_2}_{\infty}$ throughout the paper. Our first main result, which is a paraphrasing of Theorems \ref{thm:induction} and \ref{one-parameter quotients theorem}, and Corollary \ref{Wso freely generated}, is the following.
\begin{theorem} There exists a unique $2$-parameter vertex algebra $\cW^{\gs\go_2}_{\infty}$ with the following features:
\begin{enumerate}
\item It is defined over a finite localization of the polynomial ring $\mathbb{C}[c,k]$ and is freely generated of type
\begin{equation} \label{so2starting.intro} \cW(1, 2^3, 3, 4^3, 5, 6^3,\dots),\end{equation} and is weakly generated by the fields in weights $1$ and $2$.
\item The field $H$ in weight $1$ generates a Heisenberg algebra of level $k$, i.e. $H(z) H(w) \sim k (z-w)^{-2}$.
\item The three fields in each even weight $2,4,6,\dots$ have Heisenberg charges $0,\pm 2$.
\item The field in weight $2$ with Heisenberg charge $0$ is a conformal vector with central charge $c$.
\item The fields in each odd weight $3,5,\dots$ have Heisenberg charge $0$.
\end{enumerate}

Moreover, $\cW^{\gs\go_2}_{\infty}$ serves as a classifying object for vertex algebras with these properties; any vertex algebra with a strong generating set of type \eqref{so2starting.intro} (not necessarily minimal) satisfying the above conditions, is a quotient of $\cW^{\gs\go_2}_{\infty}$.
\end{theorem} 
There are $8$ infinite families of $1$-parameter quotients of $\cW^{\gs\go_2}_{\infty}$ which are either $\cW$-algebras of a type we call {\it $\gs\go_2$-rectangular}, or (orbifolds of) affine cosets of $\cW$-algebras of a type we call {\it $\gs\go_2$-rectangular with a tail}. They are denoted by $\cC^{\psi}_{XY}(n,m)$ where $X = B,C$ and $Y = B,C,D,O$, and $n,m \in \mathbb{N}$, and are just the specializations of \eqref{BDfirst}, \eqref{BDsecond}, and \eqref{BDthird}. The construction of $\cW^{\gs\go_2}_{\infty}$ is similar to the constructions of $\cW_{\infty}$, $\cW^{\text{ev}}_{\infty}$, and $\cW^{\gs\gp}_{\infty}$ in \cite{Lin,KL,CKoL2}. Unlike  $\cW_{\infty}$, $\cW^{\text{ev}}_{\infty}$, the $1$-parameter vertex algebras $\cC^{\psi}_{XY}(n,m)$ are all distinct, and there are no triality isomorphisms among them. In addition to $\cC^{\psi}_{XY}(n,m)$, we also have the GKO cosets obtained by specializing \eqref{BDGKORect} to the case $m=2$. We denote them by $\cC^{\ell}(n)$, where $\ell$ is a complex parameter related to the central charge, $n \in \mathbb{Z}$ is fixed, and $\cC^{\ell}(n)$ contains a Heisenberg algebra with level $n$ normalized so that the fields in even weights have Heisenberg charges $0, \pm 2$.


Our next main result is that $\cW^{\fr{so}_2}_{\infty}$ is a glueing of two copies of $\cW^{\text{ev}}_{\infty}$; see Theorem \ref{completion:twocopies}. The parameters of one copy are completely determined by the parameters of the other copy. Each of the $1$-parameter quotients $\cC^{\psi}_{XY}(n,m)$ and $\cC^{\ell}(n)$ is a glueing of two orthosymplectic $Y$-algebras which we write down explicitly; see Theorem \ref{pairsofcurves}. In fact, this yields $1$-parameter embeddings of hook-type $\cW$-algebras inside larger $\cW$-algebras; see Corollary \ref{pairsofcurvesrefined}.

Next, we show that for $n \in \mathbb{N}$ and $\ell-2$ an admissible level for $\gs\go_{2n}$, the simple quotient $\cC_{\ell}(2n)$ of 
$\cC^{\ell}(2n):= \T{Com}(V^{\ell}(\fr{so}_{2n}), V^{\ell-2}(\fr{so}_{2n})\otimes \cF(4n))^{\mathbb{Z}_2}$, is strongly rational; see Lemma \ref{rationalfamily}. It follows that $\cW^{\fr{so}_2}_{\infty}$ has many quotients which are strongly rational; see Corollary \ref{cor:rationalityatell}. Using intersections of truncation curves, we will use this to prove the rationality of the $\cW$-superalgebras $\cW_{\frac{4m+2n-1}{4m}-h^{\vee}}(\go\gs\gp_{1|4m}, f_{2m,2m})$ for $m\geq 2n-1$ and $m,2n-1$ coprime. At the moment, such results are otherwise inaccessible, and this is a prototype for one of the applications of universal families that we have in mind.


\subsection{Organization} In Section \ref{sec:VOA} we briefly recall free field algebras, affine vertex algebras, and $\cW$-algebras following the notation in the papers \cite{CL3,CL4} of the first and third authors. In Section \ref{sect:generalgluings}, we present the conjectural description of the universal objects $\cW^{X,S,M}_{\infty}$ for $X = A, C, BD$, as well as the explicit description of their $1$-parameter quotients which are the analogues of Gaiotto-Rap\v{c}\'ak $Y$-algebras and GKO cosets. In Section \ref{sect:YtypeSO} we specialize to the case of $\cW^{BD, \emptyset, \{2\}}_{\infty} = \cW^{\gs\go_2}_{\infty}$, and we introduce the $\gs\go_2$-rectangular $\cW$-algebras as well as those with a tail, and define the $8$ families vertex algebras which have strong generating type \eqref{so2starting.intro}. In Section \ref{sect:Diagonal}, we introduce the $4$ families of GKO cosets that also have this strong generating type. In Section \ref{sect:main}, we construct the universal $2$-parameter vertex algebra $\cW^{\gs\go_2}_{\infty}$ of this type, which is our main result. In Section \ref{sect:1paramquot}, we prove that the $Y$-algebras of type $C$ and the diagonal cosets indeed are $1$-parameter quotients of $\cW^{\gs\gp}_{\infty}$. In Section \ref{sect:recon}, we prove a reconstruction theorem (Theorem \ref{thm:reconstruction}) that says that the full OPE algebra of an $\gs\go_2$-rectangular $\cW$-algebra with a tail, which is an extension of an affine vertex algebra and $\cW^{\gs\go_2}_{\infty}$-quotient, is uniquely and constructively determined by the conformal weight and parity of the extension fields and the zero mode action of the Lie algebra on the extension fields. In Section \ref{sect:twocopies}, we show that $\cW^{\gs\go_2}_{\infty}$ is a conformal extension of a tensor product of two copies of $\cW^{\text{ev}}_{\infty}$, and moreover, each of the $1$-parameter quotients of $\cW^{\gs\go_2}_{\infty}$ discussed earlier is itself a conformal extension of two orthosymplectic $Y$-algebras. In Section \ref{sect:rationality}, we will prove our rationality results for $\cW$-superalgebras.

	\section{Vertex algebras} \label{sec:VOA} We will assume that the reader is familiar with vertex algebras, and we use the same notation as the paper \cite{CL3} of the first and third authors. We will make use of the following well-known identities which hold in any vertex algebra $\cA$.
	\begin{equation}\label{conformal identity}
		(\partial a)_{(r)}b=-ra_{(r-1)}b,\quad r\in\mathbb{Z}.
	\end{equation}
	\begin{equation}\label{skew-symmetry}
		a_{(r)}b=(-1)^{|a||b|+r+1}b_{(r)}a + \sum_{i = 1}^{\infty}\frac{(-1)^{|a||b|+r+i+1}}{i!}\partial^i \left(b_{(r+i)}a\right),\quad r\in\mathbb{Z}.
	\end{equation}
	\begin{equation}\label{quasi-associativity}
		:a (:bc:):\ =\ :(:ab:)c: +\sum_{i=0}^{\infty}\frac{1}{(i+1)!}\left(:\!\partial^{i+1}(a) (b_{(i)}c)\!:+(-1)^{|a||b|}:\!\partial^{i+1}(b) (a_{(i)}c)\!:\right).
	\end{equation}
	\begin{equation}\label{quasi-derivation}
		a_{(r)}:bc:\ =\ :(a_{(r)}b)c:+(-1)^{|a||b|}: b(a_{(r)}c):+\sum_{i=1}^r\binom{r}{i}(a_{(r-i)}b)_{(i-1)}c, \quad r\geq 0.
	\end{equation}
	\begin{equation}\label{Jacobi}
		a_{(r)}(b_{(s)}c) = (-1)^{|a||b|}b_{(s)}(a_{(r)}c) + \sum_{i=0}^r \binom{r}{i}(a_{(i)}b) _{(r+s-i)}c, \quad r,s\geq 0.
	\end{equation}
	Identities \eqref{Jacobi} are known as {\it Jacobi identities}, and we often denote them using the shorthand $J_{r,s}(a,b,c)$. We use $J(a,b,c)$ to denote the set of all Jacobi identities $\{J_{r,s}(a,b,c)|r,s\geq 0\}$.

	\subsection{Free field algebras}
	A {\it free field algebra} is a vertex superalgebra $\cV$ with weight grading 
	\[\cV = \bigoplus_{d\in \frac{1}{2}\mathbb{Z}_{\geq 0}}\cV[d],\quad \cV[0] = \mathbb{C}\B{1},\] 
	with strong generators $\{X^i | i\in I\}$ satisfying OPE relations
	\[X^i(z)X^j(w)\sim a_{i,j} \B{1}(z-w)^{-\Delta(X^i)-\Delta(X^j)}, \quad a_{i,j}\in\mathbb{C}, \quad a_{i,j} =0 \T{ if }\Delta(X^i)+\Delta(X^j)\not\in\mathbb{Z}.\]
	Note $\cV$ is not assumed to have a conformal structure. In \cite{CL3}, the first and third authors introduced the following families of free field algebras.
	\begin{enumerate} 
		\item Even algebras of orthogonal type $\cO_{\text{ev}}(n,k)$ for $n\geq 1 $ and even $k \geq 2$. When $k =2$, $\cO_{\text{ev}}(n,2)$ is just the rank $n$ Heisenberg algebra $\cH(n)$.
		\item Odd algebras of orthogonal type $\cO_{\text{odd}}(n,k)$ for $n\geq 1 $ and odd $k \geq 1$. When $k = 1$, $\cO_{\text{odd}}(n,1)$ is just the rank $n$ free fermion algebra $\cF(n)$.
		\item Even algebras of symplectic type $\cS_{\text{ev}}(n,k)$ for $n\geq 1$ and odd $k \geq 1$. When  $k = 1$, $\cS_{\text{odd}}(n,1)$ is just the rank $n$ $\beta\gamma$-system $\cS(n)$.
		\item Odd algebras of symplectic type $\cS_{\text{odd}}(n,k)$ for $n\geq 1$ and even $k \geq 2$. When  $k = 2$, $\cS_{\text{ev}}(n,2)$ is just the rank $n$ symplectic fermion algebra $\cA(n)$.
	\end{enumerate}
	We refer the reader to \cite{CL3} for the construction and key properties of these algebras.

	\subsection{Affine vertex superalgebras}
	Let $\gg$ be a simple, finite-dimensional Lie superalgebra with normalized Killing form $( \cdot | \cdot )$.
	Let $\{q^{\alpha} | \alpha \in S\}$ be a basis of $\gg$ which is homogeneous with respect to parity.
	Define the corresponding structure constants $\{f^{\alpha,\beta}_{\gamma} | \alpha,\beta,\gamma \in S\}$ by 
	\[[q^{\alpha},q^{\beta}] = \sum_{\gamma \in S} f^{\alpha,\beta}_{\gamma} q^{\gamma}.\]
	The affine vertex algebra $V^k(\gg)$ of $\gg$ at level $k$ is strongly generated by the fields $\{X^{\alpha} | \alpha \in S \}$, satisfying 
	\begin{equation} \label{OPE:affine} X^{\alpha}(z)X^{\beta}(w) \sim k (q^{\alpha}|q^{\beta})\B{1}(z-w)^{-2} + \sum_{\gamma \in S}f^{\alpha,\beta}_{\gamma}X^{\gamma}(w)(z-w)^{-1}.\end{equation}
	We define $X_{\alpha}$ to be the field corresponding to $q_{\alpha}$ where $\{q_{\alpha} | \alpha \in S\}$ is the dual basis of $\gg$ with respect to $(\cdot|\cdot)$.
	The Sugawara conformal vector $L^{\gg}$ with central charge $c^{\fr{g}}$ is given by
	\begin{equation}\label{sugwara}
		L^{\gg}=\frac{1}{2(k+h^{\vee})}\sum_{\alpha \in S}(-1)^{|\alpha|} :\!X_{\alpha}X^{\alpha}\!:, \quad c^{\fr{g}}= \frac{k\ \T{sdim}\gg}{k+h^{\vee}}.
	\end{equation}
	Fields $X^{\alpha}(z)$ and $X_{\alpha}(z)$ are primary with respect to $L^{\gg}$ and have conformal weight of 1.

	\subsection{$\cW$-superalgebras}\label{sect:W}
	Let $\gg$ be a simple, finite-dimensional Lie (super)algebra, and let $f$ be a nilpotent element in the even part of $\gg$. The $\cW$-(super)algebra $\cW^k(\gg,f)$ at level $k$ was defined by Kac, Roan, and Wakimoto \cite{KRW}, generalizing the case when $\gg$ is a Lie algebra and $f$ is a principal nilpotent given earlier by Feigin and Frenkel \cite{FF}. We complete $f$ to an $\fr{sl}_2$-triple $ \{f,h,e\}$ satisfying
	\[[h,e] =2 e, \quad [h,f] = -2f, \quad [e,f] =h.\]
	Then $x= \frac{1}{2}h$ induces a $\frac{1}{2}\mathbb{Z}$-grading on $\gg$,
	\begin{equation}\label{decomposition of g over sl2}
		\fr{g} = \bigoplus_{j\in \frac{1}{2}\mathbb{Z}}\fr{g}_j, \quad \gg_j = \{ a \in \gg | [x,a] = j a\}.
	\end{equation}
	Fix a basis $S = \bigcup S_k$ for $\gg$, where $S_k$ is a basis of $\gg_k$, and write $S_+= \bigcup\{ S_k | k <0\}$ and $S_-= \bigcup\{S_k | k>0\}$ for bases of the respective subspaces 
	$$\gg_+ =\bigoplus_{j\in \frac{1}{2}\mathbb{Z}_{>0}}\fr{g}_j ,\quad  \gg_- =\bigoplus_{j\in \frac{1}{2}\mathbb{Z}_{<0}}\fr{g}_j.$$
Let $F(\gg_+)$ be the algebra of charged fermions associated to the vector superspace $\gg_+ \oplus \gg_+^*$. It is strongly generated by fields $\{\varphi_{\alpha}, \varphi^{\alpha} | \alpha \in S_+\}$, where $\varphi_{\alpha}, \varphi^{\alpha}$ have opposite parity to $q^{\alpha}$. They satisfy
	\[\varphi_{\alpha}(z)\varphi^{\beta}(w) \sim \delta_{\alpha,\beta}(z-w)^{-1},\quad\varphi_{\alpha}(z)\varphi_{\beta}(w)\sim 0 \sim \varphi^{\alpha}(z)\varphi^{\beta}(w).\]
	We give $F(\gg_+)$ the conformal vector and associated central charge
		\begin{equation}\label{charge Fermions}
		L^{\T{ch}}=\sum_{\alpha \in S_+} (1-m_{\alpha}):\partial \varphi^{\alpha}\varphi_{\alpha}\!:-m_{\alpha}:\!\varphi^{\alpha}\partial \varphi_{\alpha}:,\quad c^{\T{ch}}= - \sum_{\alpha \in S_+} (-1)^{|\alpha|}(12m_{\alpha}^2-12m_{\alpha}+2).
	\end{equation}
	The fields $\varphi_{\alpha}$ and $\varphi^{\alpha}$ are primary with respect to $L^{\T{ch}}$ and have conformal weights $1-m_{\alpha}$ and $m_{\alpha}$, respectively. Since $f \in \gg_{-1}$, it endows $\gg_{\frac{1}{2}}$ with a skew-symmetric form
	\begin{equation}\label{bilinear form}
		\langle a,b\rangle = (f|[a,b]).
	\end{equation}
	Denote by $F(\gg_{\frac{1}{2}})$ the algebra of neutral fermions associated to $\gg_{\frac{1}{2}}$. It has strong generators $\{\Phi_{\alpha}| \alpha \in S_{\frac{1}{2}}\}$, where $\Phi_{\alpha}$ has the same parity as $q^{\alpha}$, satisfying
	\[\Phi_{\alpha}(z)\Phi_{\beta}(w) \sim \langle q^{\alpha},q^{\beta} \rangle (z-w)^{-1}.\]
	We give $F(\gg_{\frac{1}{2}})$ the conformal vector $L^{\T{ne}}$ with central charge $c^{\T{ne}}$, where
	\[ L^{\T{ne}}=\frac{1}{2}\sum_{\alpha \in S_{\frac{1}{2}}} :\!\partial \Phi^{\alpha}\Phi_{\alpha}\!:,\quad c^{\T{ne}}=-\frac{1}{2}\T{sdim}\gg_{\frac{1}{2}}.\]
	Here $\Phi^{\alpha}$ is dual to $\Phi_{\alpha}$ with respect to the bilinear form (\ref{bilinear form}), and $\Phi^{\alpha}(z), \Phi_{\alpha}(z)$ are primary of conformal weight of $\frac{1}{2}$ with respect to $L^{\T{ne}}$.
	
	As in \cite{KRW}, define $C^k(\gg,f)= V^k(\gg)\otimes F(\gg_+)\otimes F(\gg_{\frac{1}{2}})$. 
	It admits a $\mathbb{Z}$-grading by charge
	\[C^k(\gg,f)=\bigoplus_{j \in\mathbb{Z}} C_j, \]
	by giving $\varphi_{\alpha}$ charge $-1$, $\varphi^{\alpha}$ charge 1, and all others charge $0$.
	There is an odd field $d = d_{\T{st}} + d_{\T{tw}}$ of charge $-1$, where
	\begin{equation}\label{Qfield}
		\begin{split}
			d_{\T{st}} &= \sum_{\alpha\in S_+} (-1)^{|\alpha|}:\!X^{\alpha}\varphi^{\alpha}\!: - \frac{1}{2}\sum_{\alpha,\beta,\gamma \in S_+} (-1)^{|\alpha||\gamma|}f^{\alpha,\beta}_{\gamma}:\!\varphi_{\gamma}\varphi^{\alpha}\varphi^{\beta}\!:,\\
			d_{\T{tw}} &=\sum_{\alpha\in S_+}(f|q^{\alpha})\varphi^{\alpha} + \sum_{\alpha \in S_{\frac{1}{2}}}:\!\varphi^{\alpha}\Phi_{\alpha}\!:.
		\end{split}
	\end{equation}
	It satisfies $d(z)d(w) \sim  0$, so $(C^k(\gg,f), d_{(0)})$ has the structure of $\mathbb{Z}$-graded homology complex, and one defines
	$$\cW^k(\gg,f): =H(C^k(\gg,f),d_{(0)}).$$
	Its conformal vector is represented by $L = L^{\gg}+\partial x + L^{\T{ch}} + L^{\T{ne}}$, which has central charge 
	\begin{equation}\label{central charge W algebra}
		c(\gg,f,k) = \frac{k\ \T{sdim}(\gg)}{k+h^{\vee}}-12k(x|x) - \sum_{\alpha \in S_+} (-1)^{|\alpha|}(12m_{\alpha}^2-12m_{\alpha}+2)-\frac{1}{2}\T{sdim}(\gg_{\frac{1}{2}}).
	\end{equation}
	Here $m_{\alpha} = j$ if $\alpha \in S_j$.

Denote by $\gg^f$ the centralizer of $f$ in $\gg$, and let $\ga = \gg^f \cap \gg_0$. Define fields
	\begin{equation}\label{XtoJ} J^{\alpha}= X^{\alpha} + \sum_{\beta,\gamma\in S_+}(-1)^{|\gamma|}f^{\alpha,\beta}_{\gamma}:\!\varphi_{\gamma}\varphi^{\beta}\!: + \frac{(-1)^{|\alpha}}{2}\sum_{\beta,\gamma \in S_+}f^{\beta,\alpha}_{\gamma}:\!\Phi_{\beta}\Phi^{\gamma}\!:.\end{equation}
	Then $\{J^{\alpha}|q^{\alpha} \in \fr{a}\}$ close under OPE, and generate an affine vertex algebra of type $\fr{a}$. We have
	\begin{theorem}[\cite{KW}, Theorem 2.1]
		\begin{equation}
			J^{\alpha}(z)J^{\beta}(w)\sim (k(q^{\alpha}|q^{\beta})+k^{\Gamma}(q^{\alpha},q^{\beta}))(z-w)^{-2} + f^{\alpha,\beta}_{\gamma} J^{\gamma}(w)(z-w)^{-1},
		\end{equation}
		where
		\[k^{\Gamma}(q^{\alpha},q^{\beta})=\frac{1}{2}\left(\kappa_{\gg}(q^{\alpha},q^{\beta})-\kappa_{\gg_0}(q^{\alpha},q^{\beta})-\kappa_{\frac{1}{2}}(q^{\alpha},q^{\beta})\right),\]
		with $\kappa_{\frac{1}{2}}$ the supertrace of $\gg_0$ on $\gg_{\frac{1}{2}}$.
	\end{theorem}

\begin{theorem} \cite[Theorem 4.1]{KW} \label{thm:kacwakimoto}
Let $\gg$ be a simple finite-dimensional Lie superalgebra with an invariant bilinear
form $( \ \ |  \ \ )$, and let $x, f$ be a pair of even elements of $\gg$ such that ${\rm ad}\ x$ is diagonalizable with
eigenvalues in $\frac{1}{2} \mathbb{Z}$, $[x,f] = -f$, and all eigenvalues of ${\rm ad}\ x$ on $\gg^f$ are non-positive, so
$\gg^f = \bigoplus_{j\leq 0} \gg^f_j$. Then
 \begin{enumerate}
\item For each $q^\alpha \in \gg^f_{-j}$ with $j\geq 0$, there exists a  $d_{(0)}$-closed field $K^\alpha$ of weight
$1 + j$, with respect to $L$.
\item The homology classes of the fields $K^\alpha$, where $\{q^\alpha\}$ runs over a basis of $\gg^f$, freely generate $\cW^k(\gg, f)$.
\item $H_0(C(\gg, f, k), d_0) = \cW^k(\gg, f)$ and $H_j(C(\gg, f, k), d_0) = 0$ if $j \neq 0$.
\end{enumerate}
\end{theorem}
One can also consider the reduction of a $V^k(\gg)$-module $M$. The homology of the complex $$H(M \otimes F(\gg_+) \otimes F(\gg_{\frac{1}{2}}), d_0),$$ is a $\cW^k(\gg, f)$-module that we denote by $H_{f}(M)$.  In this notation, $\cW^k(\gg,f) = H_f(V^k(\gg))$. 

\subsection{Large level limits of $\cW$-algebras}
Suppose that $\gg$ is a simple Lie (super)algebra and $( \  |  \ )$ is nondegenerate. Let  $x, f$, and 
$\gg^f$ be as in Theorem \ref{thm:kacwakimoto}. In \cite{CL3}, a certain large level limit $\cW^{\T{free}}(\gg,f) = \lim_{k\rightarrow \infty} \cW^k(\gg,f)$ was defined. It is a simple vertex algebra with the following properties.

	\begin{theorem}[\cite{CL3}, Theorem 3.5 and Corollary 3.4]\label{thm:largelevellimit} 
		$\cW^{\T{free}}(\gg,f)$ is a free field algebra with strong generators $\{X^{\alpha}| q^{\alpha} \in \gg^f\}$ and OPEs
		\begin{equation}\label{freeField}
			X^{\alpha}(z)X^{\beta}(w) \sim (z-w)^{-2k}\delta_{j,k}B_k(q^{\alpha},q^{\beta})
		\end{equation}
		for $q^{\alpha} \in \gg^{f}_{-k}$ and $q^{\beta} \in \gg^{f}_{-j}$, where
		\[B_k:\gg_{-k}^{f}\times \gg^f_{-k} \to \mathbb{C}, \quad B_k(a,b):=(-1)^{2k}((\T{ad} (f))^{2k}b|a). \]
		Moreover, $\cW^{\infty}(\gg,f)$ decomposes as a tensor product of the standard free field algebras. 
		Let $X_k=\T{Span}\{X^{\alpha} | q^{\alpha} \in \gg^f_{-k}\}$.
		Then
		\begin{itemize}
			\item If $2k$ is even and $(\cdot|\cdot)$ is symmetric, $X_k$ generates an even algebra of orthogonal type.
			\item If $2k$ is odd and $(\cdot|\cdot)$ is symmetric, $X_k$ generates an odd algebra of orthogonal type.
			\item If $2k$ is odd and $(\cdot|\cdot)$ is skew-symmetric, $X_k$ generates an even algebra of symplectic type.
			\item If $2k$ is even and $(\cdot|\cdot)$ is skew-symmetric, $X_k$ generates an odd algebra of symplectic type.
		\end{itemize}
	\end{theorem}
	This theorem is useful for deducing the strong generating type of cosets of $\cW$-algebras by affine subalgebras. In the limit, the coset becomes an orbifold of a free field algebra whose structure can be analyzed using classical invariant theory. A strong generating for this orbifold then gives rise to a strong generating set for the coset at generic levels.

	\subsection{Vertex algebras over commutative rings} \label{subsection:voaring} 
	When working with vertex algebras over a finitely generated $\mathbb{C}$-algebra $R$, we will use the notation and setup of \cite[Section 3]{Lin}. 
	Let $\cV$ be a vertex algebra over $R$ with conformal weight grading $\cV = \bigoplus_{d \in \frac{1}{2} \mathbb{Z}_{\geq 0}} \cV[d]$ where $\cV[0] \cong R$.
All ideals $\cI \subseteq \cV$ will be graded, that is, $\cI = \bigoplus_{d \in \frac{1}{2} \mathbb{Z}_{\geq 0}} \cI[d]$ where $\cI[d] = \cI \cap \cV[d]$. We call $\cV$ {\it simple} if there are no proper graded ideals $\cI$ such that $\cI[0] = \{0\}$. This coincides with the usual notion of simplicity in the case that $R$ is a field. If $I\subseteq R$ is an ideal, $I$ is a subset of $\cV[0] \cong R$, and we denote by $I \cdot \cV$ the vertex algebra ideal generated by $I$. Then $\cV^I = \cV / (I \cdot \cV)$ is a vertex algebra over $R/I$. Even if $\cV$ is simple as a vertex algebra over $R$, $\cV^I$ need not be simple as a vertex algebra over $R/I$.

If $R$ is the coordinate ring of a variety $X$, and $\cV$ is a simple vertex algebra over $R$, suppose that $I\subseteq R$ is an ideal such that $\cV^I$ is {\it not} simple, i.e., $\cV^I$ has a maximal proper graded ideal $\mathcal{I}$ with $\mathcal{I}[0] = \{0\}$. Then the quotient 
$$\cV_{I} = \cV^I / \cI$$ is a simple vertex algebra over $R/I$. Letting $Y\subseteq X$ be the closed subvariety corresponding to $I$, we can regard $\cV_I$ as a simple vertex algebra defined over $Y$. For ideals $I_1, I_2$ with this property, we have the corresponding simple vertex algebras  $\cV_{I_1} = \mathcal{V}^{I_1} / \mathcal{I}_1$ and $\mathcal{V}_{I_2} = \mathcal{V}^{I_2} / \mathcal{I}_2$ over $R/I_1$ and $R / I_2$, respectively. Let $Y_1, Y_2 \subseteq X$ be the closed subvarieties corresponding to $I_1, I_2$, and let $p \in Y_1 \cap Y_2$. Let $\mathcal{V}^p_{I_1}$ and $ \mathcal{V}^p_{I_2}$ be the vertex algebras over $\mathbb{C}$ obtained by evaluating at $p$. As above, $\mathcal{V}^p_{I_1}$ and $ \mathcal{V}^p_{I_2}$ need not be simple, and we denote their simple quotients by $\mathcal{V}_{I_1,p}$ and $\mathcal{V}_{I_2,p}$. Then $p$ corresponds to a maximal ideal $M_p \subseteq R$ containing both $I_1$ and $I_2$, so we have isomorphisms
\begin{equation}\label{tautological}  \mathcal{V}_{I_1,p} \cong \mathcal{V}_M \cong \mathcal{V}_{I_2,p}.\end{equation}

Our main example $\cW^{\fr{so}_2}_{\infty}$ will be defined over a localization $R = D^{-1} \mathbb{C}[c,k]$ of the polynomial ring $\mathbb{C}[c,k]$, where $D$ is finitely generated multiplicatively closed set; see Theorem \ref{thm:induction}. We will identity certain $1$-parameter vertex algebras $\cC^{\psi}_{XY}(n,m)$ for $n,m \in \mathbb{N}$, and $\cC^{\ell}(n)$ for $n \in \mathbb{Z}$, with quotients of $\cW^{\gs\go_2}_{\infty}$ along prime ideals $I \subseteq R$, after a suitable localization. All the vertex algebras $\cC^{\psi}_{XY}(n,m)$ will be defined either as $\cW$-algebras $\cW^{\psi}_{XY}(n,m)$, affine cosets of such $\cW$-algebras, or $\mathbb{Z}_2$-orbifolds of such $\cW$-algebras or their cosets. Similarly, $\cC^{\ell}(n)$ is either an affine coset of an affine vertex algebra tensored with a free field algebra, or the $\mathbb{Z}_2$-orbifold of such a coset. The defining equation for $I$ is given by the formula for the central charge $c$ as a rational function of the level $k$ of the Heisenberg subalgebra.

We always regard $\cC^{\psi}_{XY}(n,m)$ and $\cC^{\ell}(n)$ as $1$-parameter vertex algebras, where $\psi$ and $\ell$ are viewed as {\it formal} parameters. They are defined over the localization of $R/I$ obtained by inverting all denominators of structure constants of $\cW^{\gs\go_2}_{\infty}$ after replacing $c$ and $k$ with the corresponding functions of $\psi$ (respectively $\ell$). At a given point $\psi_0 \in \mathbb{C}$, the specialization $$\cC^{\psi_0}_{XY}(n,m) := \cC^{\psi}_{XY}(n,m)  / (\psi-\psi_0) \cC^{\psi}_{XY}(n,m)$$ makes sense as long as $\psi$ is not in the (finite) set of poles of denominators of structure constants. However, $\cC^{\psi_0}_{XY}(n,m) $ can be a proper subalgebra of the \lq\lq honest" coset (or orbifold) obtained by first specializing $\cW^{\psi}_{XY}(n,m)$ to $\psi = \psi_0$, and then taking its coset (or orbifold), even though for generic values of $\psi$ these agree. The same can happen for $\cC^{\ell_0}(n) := \cC^{\ell}(n) / (\ell - \ell_0) \cC^{\ell}(n)$ of $\cC^{\ell}(n)$ at a particular level $\ell_0 \in \mathbb{C}$.  The issue is that $\cW^{\gs\go_2}_{\infty}$ is weakly generated by the fields in weight at most $2$, so the same holds for $\cC^{\psi_0}_{XY}(n,m)$ or $\cC^{\ell_0}(n)$, and this weak generation property of the honest coset (or orbifold) can fail at special values $\psi_0$ or $\ell_0$. There is no algorithm for determining the set of points where weak generation fails, and these sets need not be finite. So even though we have isomorphisms \eqref{tautological} between simple quotients of any of the above $1$-parameter quotients of $\cW^{\fr{so}_2}_{\infty}$ at intersection points on their truncation curves, it is a difficult problem to show that at such intersection points, both vertex algebras coincide with the honest cosets (or orbifolds).

We will show that in the cases 
$$\cC^{\psi}_{CC}(0,m) = \cW^{\psi-2m-1}(\gs\gp_{4m}, f_{2m, 2m}),\qquad \cC^{\psi}_{BD}(0,m) = \cW^{\psi-4m}(\gs\go_{4m+2}, f_{2m+1, 2m+1}),$$ there are only finitely many values of $\psi$ where the weak generation property can fail, and we describe them explicitly; see Theorem \ref{weakgeneration:Walgebras}. In Section \ref{sect:rationality}, using a different approach that involves fusion categories for rational $\cW$-algebras of type $D$, we will prove that for $n$ a positive integer and $\ell-2$ an admissible level for $\gs\go_{2n}$, the simple quotient $\cC_{\ell}(2n)$ is strongly rational and the weak generation property holds for all but finitely many values of $\ell$. In Section \ref{sect:rationality}, we also determine the set of levels where $\cC^{\psi}_{CO}(0,m)$ has the weak generation property. Combining this with the rationality of $\cC_{\ell}(2n)$, and using the isomorphisms \eqref{tautological} at intersection points of truncation curves, we prove our rationality result, Theorem \ref{Rationality:COcase}.

\section{New universal vertex algebras as glueings} \label{sect:generalgluings}
In this section, we will give a conjectural picture of how $\cW$-algebras of classical Lie types can be organized into families that are governed by new universal $2$-parameter vertex algebras that are themselves glueings of finitely many copies of the basic ones, namely, $\cW_{\infty}$ in type $A$, and $\cW^{\text{ev}}_{\infty}$ and $\cW^{\gs\gp}_{\infty}$ in types $B$, $C$, and $D$. The $1$-parameter quotients of these $2$-parameter vertex algebras include $\cW$-algebras, affine cosets of certain $\cW$-(super)algebras, and certain diagonal cosets that are analogues of the GKO cosets \eqref{GKOalternate}.

\subsection{$\cW$-algebras in type $A$} For sets of positive integers $S = \{d_1, \dots, d_t\}$ and $M = \{m_0,\dots, m_t\}$ of sizes $t$ and $t+1$ respectively, we obtain a partition $P_{S,M} = (n_0^{m_0}, n_1^{m_1},\dots, n_{t}^{m_t})$ of $N = \sum_{i=0}^t n_i m_i$ consisting of $m_i$ parts of size $n_i$, where $n_t = 2$ and $n_i - n_{i+1} = d_{i+1}$ for $i=0,1,\dots, t-1$. This notation is also meaningful when $t = 0$ and $S = \emptyset$; in this case, $M = \{m\}$ for some $m >0$, and $P = (2^m)$.

For $r,s \geq 0$, we denote by $P_{S,M,r,s}$ the partition of $N_{r,s} = N + s + r \sum_{i=0}^t m_i$ given by $$((n_0+r)^{m_0}, (n_1+r)^{m_1},\dots, (n_{t}+r)^{m_t}, 1^s),$$ which has $m_i$ parts of size $n_i + r$, and $s$ parts of size $1$, so that $P_{S,M,0,0} = P_{S,M}$. Let $f_{S,M,r,s}\in \gs\gl_{N_{r,s}}$ be the corresponding nilpotent element. The strong generating type of 
\begin{equation} \label{Afirst} \cW^k(\gs\gl_{N_{r,0}}, f_{P_{S,M,r,0}})\end{equation} is given by \cite[Proposition 2.1]{CFLN}; up to any fixed conformal weight, it is independent of $r$ for $r$ sufficiently large. In addition, it is easy to see that for all $r$, the Lie algebra $\gs\gl_{N_{r,0}}^{\natural} \cong \gt \oplus (\bigoplus_{i=0}^t \gs\gl_{m_i})$, where $\gt$ is an abelian Lie algebra of dimension $t$. Moreover, the action of $\gs\gl_{N_{r,0}}^{\natural}$ on the generating fields of higher weight, is independent of $r$ for $r$ sufficiently large. This motivates the following conjecture:

\begin{conjecture} \label{universal:typeA} Let $S = \{d_1,\dots, d_t\}$, $M = \{m_0,\dots, m_t\}$, and $n_i$ satisfying $n_i - n_{i+1} = d_{i+1}$ and $n_t = 2$, be as above.There exists a unique $2$-parameter vertex algebra $\cW^{A,S,M}_{\infty}$ with the following properties:
\begin{enumerate}
\item $\cW^{A,S,M}_{\infty}$ is defined over a localization of the polynomial ring in two variables obtained by inverting finitely many polynomials.
\item $\cW^{A,S,M}_{\infty}$ is freely generated of the appropriate type determined by $S$ and $M$.
\item $\cW^{A,S,M}_{\infty}$ admits all the $\cW$-algebras \eqref{Afirst} as $1$-parameter quotients.
\item $\cW^{A,S,M}_{\infty}$ is weakly generated by the fields in weight at most $3$.
\item $\cW^{A,S,M}_{\infty}$ is an extension of $\cH(|M| - 1)$ tensored with $|M|$ commuting copies of $\cW_{\infty}$, where $|M| = \sum_{i=0}^t m_i$.
\item $\cW^{A,S,M}_{\infty}$ serves as a classifying object for vertex algebras satisfying (1)-(4); any vertex algebra with a strong generating set of this type (not necessarily minimal) satisfying (1)-(4), is a quotient of $\cW^{A,S,M}_{\infty}$. 
\end{enumerate}
\end{conjecture}
 In the case $S = \emptyset$ and $M = \{1\}$, $\cW^{A,S,M}_{\infty}$ is exactly the algebra $\cW_{\infty}$ constructed in \cite{Lin}. The algebras $\cW^k(\gs\gl_{N_{r,0}}, f_{P_{S,M,r,0}})$ are the analogues of the $1$-parameter quotients $\cC^{\psi}(n,0) = \cW^{\psi -n}(\gs\gl_n)$ of $\cW_{\infty}$, defined in \cite{CL3}. We expect that in all cases except $\cW_{\infty}$, $\cW^{A,S,M}_{\infty}$ is in fact generated by the fields in weights $1$ and $2$.

We expect $\cW^{A,S,M}_{\infty}$ to admit several other families of $1$-parameter quotients which we now describe. First, $\cW^k(\gs\gl_{N_{r,s}}, f_{P_{S,M,r,s}})$ has an affine subVOA $V^{k'}(\gg\gl_s)$ for some shifted level $k'$. By passing to the large level limit and using standard results of classical invariant theory from \cite{W} as in \cite{CL3,CL4}, it is easy to check that
\begin{equation} \label{Asecond} \text{Com}(V^{k'}(\gg\gl_s), \cW^k(\gs\gl_{N_{r,s}}, f_{P_{S,M,r,s}})),\end{equation} has the same generating type as $\cW^{A,S,M}_{\infty}$, and should therefore be a $1$-parameter quotient of $\cW^{A,S,M}_{\infty}$. Note that the case $s = 0$ of \eqref{Asecond} is just the family \eqref{Afirst}. For $s >0$, \eqref{Asecond} are analogous to the vertex algebras $\cC^{\psi}(n,m)$ for $n\geq 2$ and $m >0$ given in \cite{CL3}.

Next, consider the nilpotent element $f_{P_{S,M,r|s}} \in \gs\gl_{N_{r,0}|s}$ which is just the sum of $f_{P_{S,M,r,0}} \in \gs\gl_{N_{r,0}} \subseteq \gs\gl_{N_{r,0}|s}$ and the zero nilpotent in $\gs\gl_s \subseteq \gs\gl_{N_{r,0}|s}$. Then 
$\cW^k(\gs\gl_{N_{r,0|s}}, f_{P_{S,M,r|s}})$ contains $V^{k''}(\gg\gl_s)$ for some shifted level $k''$, and the coset 
\begin{equation} \label{Athird} \text{Com}(V^{k''}(\gg\gl_s), \cW^k(\gs\gl_{N_{r,0|s}}, f_{P_{S,M,r|s}})),\end{equation} has the same generating type as $\cW^{A,S,M}_{\infty}$, so it should also be a $1$-parameter quotient of $\cW^{A,S,M}_{\infty}$. As above, for $s >0$, these are the analogues of $\cD^{\psi}(n,m)$ for $n \geq 2$ and $m >0$ defined in \cite{CL3}.

Next, we replace $n_i$ with $n_i -1$ in the partition $P_{S,M}$, which is equivalent to taking $r = -1$ above. In particular, since $n_t=2$, this replaces $n_t$ with $1$. Then
$$\text{Com}(V^{k'}(\gg\gl_s), \cW^k(\gs\gl_{N_{-1,s}}, f_{P_{S,M,-1,s}})),\qquad \text{Com}(V^{k''}(\gg\gl_s), \cW^k(\gs\gl_{N_{-1,0|s}}, f_{P_{S,M,-1|s}}))$$ also have the same generating type as $\cW^{A,S,M}_{\infty}$. These are the analogues of the algebras $\cC^{\psi}(1,n)$ and $\cD^{\psi}(1,n)$ defined in \cite{CL3}, which are affine cosets of affine vertex (super)algebras.

Finally, we replace $n_i$ with $n_i -2$ in $P_{S,M}$, which is equivalent to taking $r = -2$. Since $n_t = 2$, this has the effect of removing $n_t$ since it is now zero, i.e., we have replaced $P_{S,M}$ with the smaller partition 
$((n_0-2)^{m_0}, (n_1-2)^{m_1},\dots, (n_{t-1}-2)^{m_{t-1}})$. We use the notation $f_{S,M,-2,s}$ to denote this nilpotent in $\gs\gl_{N_{-2,s}}$ where $N_{-2,s} = N + s -2 \sum_{i=0}^t m_i$, as above. Consider the tensor products 
$$\cW^k(\gs\gl_{N_{-2,s}}, f_{S, M,-2,s}) \otimes \cS(m_t s),\qquad \cW^k(\gs\gl_{N_{-2,s}}, f_{S, M,-2,s}) \otimes \cE(m_t s),$$ where $\cS(m_t s)$ and $\cE(m_t s)$ are the $\beta\gamma$-system and $bc$-system of rank $m_t s$, respectively. Note that
\begin{enumerate}
\item $\cW^k(\gs\gl_{N_{-2,s}}, f_{S, M,-2,s})$ has an affine subVOA $V^{k'}(\gs\gl_s)$ for some $k'$,
\item $\cS(m_t s)$ has an action of $V^{-m_t}(\gg\gl_s) \otimes V^{-s}(\gs\gl_{m_t})$,
\item $\cE(m_t s)$ has an action of $V^{m_t}(\gg\gl_s) \otimes V^{s}(\gs\gl_{m_t})$.
\end{enumerate}

We therefore have diagonal embeddings
$$V^{k'-m_t}(\gg\gl_s) \hookrightarrow \cW^k(\gs\gl_{N_{-2,s}}, f_{S, M,-2,s}) \otimes \cS(m_t s),\qquad V^{k'+m_t}(\gg\gl_s) \hookrightarrow \cW^k(\gs\gl_{N_{-2,s}}, f_{S, M,-2,s}) \otimes \cE(m_t s).$$
The following affine cosets are easily seen to have the generating type of $\cW^{A,S,M}_{\infty}$:
\begin{equation} \label{AGKO} \begin{split} & 
\text{Com}(V^{k'-m_t}(\gg\gl_s), \cW^k(\gs\gl_{N_{-2,s}}, f_{S, M,-2,s}) \otimes \cS(m_t s)),
\\ & \text{Com}(V^{k'+m_t}(\gg\gl_s), \cW^k(\gs\gl_{N_{-2,s}}, f_{S, M,-2,s}) \otimes \cE(m_t s)).
\end{split}
\end{equation}
These are the analogues of the GKO cosets $\cC^{\psi}(0,m)$ and $\cD^{\psi}(0,m)$ defined in \cite{CL3}. Unlike the case of $\cW_{\infty}$, we do {\it not} generally expect there to exist triality isomorphisms relating the algebras \eqref{AGKO} to the other algebras \eqref{Afirst}, \eqref{Asecond}, \eqref{Athird}.

\begin{example} \label{rectA} Consider $S = \emptyset$ and $M = \{m\}$, so $P_{\emptyset, \{m\}} = (2^m)$ and $N = 2m$. We say that $\cW^{A,\emptyset,\{m\}}_{\infty}$ is of {\it rectangular type}, and it has generating type
$$\cW(1^{m^2-1}, 2^{m^2},3^{m^2}, 4^{m^2},\dots).$$ Here the affine subVOA is of type $\gs\gl_m$, and the $m^2$ fields in each weight $d \geq 2$ transform under $\gs\gl_m$ as the sum of the trivial and adjoint modules.  These were constructed in \cite{GRZ}, although their uniqueness as $2$-parameter vertex algebras was not proven. In this case, the partition obtained by replacing $r$ with $-2$ is empty, so the GKO cosets \eqref{AGKO} in this example are
$$\text{Com}(V^{k-m}(\gg\gl_{s}), V^k(\gs\gl_{s}) \otimes \cS(ms)), \quad \text{Com}(V^{k + m}(\gg\gl_{s}), V^k(\gs\gl_{s}) \otimes \cE(ms)).$$ Note that for all levels $k$, the affine subVOA of the first coset is the image of $V^{-s}(\gs\gl_{m})$ in $\cS(ms)$, which is just $V^{-s}(\gs\gl_m)$ when $s > m$. Similarly the affine subVOA of the second coset is the image of $V^s(\gs\gl_m)$ in $\cE(ms)$, which is isomorphic to $L_s(\gs\gl_m)$. 

When $k$ is admissible for $\gs\gl_s$, we have an embedding of simple vertex algebras $L_{k + m}(\gg\gl_{s}) \hookrightarrow L_k(\gs\gl_{s}) \otimes \cE(ms)$, and the coset 
\begin{equation} \label{typeArectsimple} \text{Com}(L_{k + m}(\gg\gl_{s}), L_k(\gs\gl_{s}) \otimes \cE(ms)),\end{equation} is an extension of the tensor product of the vertex algebras 
$\text{Com}(L_{k + i}(\gg\gl_{s}), L_{k+i-1}(\gs\gl_{s}) \otimes \cE(s))$ for $i = 1,\dots,m$, and a rank $m-1$ lattice vertex algebra, and hence is strongly rational. We expect that \eqref{typeArectsimple} is the simple quotient of $\cW^{A,\emptyset,\{m\}}_{\infty}$ at all these points, so that $\cW^{A,\emptyset,\{m\}}_{\infty}$ has many quotients which are strongly rational. \end{example}

\begin{example} \label{Asecondexample} Consider $S = \{2\}$ and $M = \{1,3\}$, so $P_{\{2\}, \{1,3\}} = (4, 2^3)$ and $N = 10$. The generating type of $\cW^{A,\{2\}, \{1,3\}}_{\infty}$ is 
$$\cW(1^9, 2^{16},3^{16},4^{16},\dots),$$ where the affine subVOA is of type $\gg\gl_3$, and the $16$ fields in each weight $d \geq 2$ transform under $\gs\gl_3$ as the sum of the trivial, standard, dual standard, and adjoint modules. The partition obtained by replacing $r$ with $-2$ is just $(2)$, and the corresponding nilpotent $f_{(2)} \subseteq \gs\gl_{2+s}$ is just the minimal nilpotent $f_{\text{min}}$. Therefore the GKO cosets \eqref{AGKO} in this example are
$$ \text{Com}(V^{k'-3}(\gg\gl_{s}), \cW^k(\gs\gl_{2+s}, f_{\text{min}}) \otimes \cS(3s)), \quad \text{Com}(V^{k'+3}(\gg\gl_{s}), \cW^k(\gs\gl_{2+s}, f_{\text{min}}) \otimes \cE(3s)), \quad k' = k+1.$$
When $k$ is admissible for $\gs\gl_s$, we have an embedding of simple vertex algebras $L_{k+4}(\gs\gl_{s})\hookrightarrow \cW_k(\gs\gl_{2+s}, f_{\text{min}}) \otimes \cE(3s)$, and the coset
\begin{equation} \label{typeAexample2} \text{Com}(L_{k+4}(\gg\gl_{s}), \cW_k(\gs\gl_{2+s}, f_{\text{min}}) \otimes \cE(3s)),\end{equation} is an extension the tensor product of a rank $3$ lattice vertex algebra and the following vertex algebras
\begin{enumerate}
\item $\text{Com}(L_{k+1}(\gg\gl_{s}), \cW_k(\gs\gl_{2+s}, f_{\text{min}}))$,
\item $\text{Com}(L_{k+2}(\gg\gl_{s}),  L_{k+1}(\gs\gl_s) \otimes \cE(s))$,
\item $\text{Com}(L_{k+3}(\gg\gl_{s}), L_{k+2}(\gs\gl_s) \otimes \cE(s))$,
\item $\text{Com}(L_{k+4}(\gg\gl_{s}), L_{k+3}(\gs\gl_s) \otimes \cE(s))$.
\end{enumerate}
Suppose that $m\in \mathbb{N}_{\geq 2}$, $m+s$ is odd, and we fix $k = -(s+2) + \frac{m+s+2}{2}$. The second third, and fourth cosets above strongly rational \cite{ACL}, and the first one corresponds to an intersection point on the truncation curves for $\cC^{\psi}(2,s)$ and $\cC^{\phi}(m,0)$ at the point $\psi = \frac{m+s+2}{2}$ and $\phi = \frac{m+s}{m+s+2}$, in the notation of \cite{CL3}. This suggests that 
$$\text{Com}(L_{k+1}(\gg\gl_{s}), \cW_k(\gs\gl_{2+s}, f_{\text{min}})) \cong \cW_{-m+\frac{m+s}{m+s+2}}(\gs\gl_m),$$ which is strongly rational. This isomorphism will be proven in \cite{CLN} as a consequence of a more general theorem, and it implies that \eqref{typeAexample2} is strongly rational for these values of $k$. Finally, we expect these to be strongly rational quotients of $\cW^{A,\{2\}, \{1,3\}}_{\infty}$.
\end{example}

\subsection{$\cW$-algebras in type $B$, $C$, and $D$} Recall that nilpotents correspond to partitions that satisfy the following conditions:
\begin{enumerate}
\item In type $C$, a partition $(n_0^{m_0},  n_1^{m_1},\dots, n_t^{m_t})$ must satisfy $m_i$ even whenever $n_i$ is odd.
\item In type $B$ and $D$, a partition $(n_0^{m_0},  n_1^{m_1},\dots, n_t^{m_t})$ must satisfy $m_i$ even whenever $n_i$ is even.
\end{enumerate}
(In type $D$, there are two conjugacy classes of nilpotents that correspond to very even partitions, i.e., those where all $n_i$ are even, but we need not distinguish them because the corresponding $\cW$-algebras are isomorphic \cite{FKN}).
We will use the same notation $S$ and $M$ as above. For type $C$, we consider partitions $P_{S,M} = (n_0^{m_0}, n_1^{m_1},\dots, n_{t}^{m_t})$ of $N = \sum_{i=0}^t n_i m_i$ consisting of $m_i$ parts of size $n_i$, where $n_i - n_{i+1} = d_{i+1}$ for $i=0,1,\dots, t-1$, and $m_i$ is even whenever $n_i$ is odd. We will take $n_t$ to be either $2$ or $3$, and if at least one of the $m_i$ is odd, only one of these possibilities is allowed. We will use the notation $P^{2C}_{S,M}$ or $P^{3C}_{S,M}$ in the case respectively that $n_t = 2$ or $3$. If all $m_i$ are even, then both $n_t = 2$ and $n_t = 3$, are allowed, and we still use this notation because we need to distinguish these two partitions.

For $r,s \geq 0$, we denote by 
$P^{2C}_{S,M,r,s}$ and $P^{3C}_{S,M,r,s}$ the partitions of \begin{equation} \label{NRSTypeC} N_{r,s} = N + 2s + 2r \sum_{i=0}^t m_i,\end{equation} given by $((n_0+2r)^{m_0}, (n_1+2r)^{m_1},\dots, (n_{t}+2r)^{m_t}, 1^{2s})$, which has $m_i$ parts of size $n_i + 2r$, and $2s$ parts of size $1$. Note that $P^{2C}_{S,M,r,s}$ and $P^{3C}_{S,M,r,s}$ satisfy our restrictions whenever $P^{2C}_{S,M}$ and $P^{3C}_{S,M}$ do, and $P^{iC}_{S,M,0,0} = P_{S,M}$ for $i = 2,3$. Let $f^{2C}_{S,M,r,s}\in \gs\gp_{N_{r,s}}$ and $f^{3C}_{S,M,r,s}\in \gs\gp_{N_{r,s}}$ denote the corresponding nilpotents.

Similarly, in types $B$ and $D$, we have the partition $P_{S,M} = (n_0^{m_0}, n_1^{m_1},\dots, n_{t}^{m_t})$ of $N = \sum_{i=0}^t n_i m_i$ consisting of $m_i$ parts of size $n_i$, where $n_i - n_{i+1} = d_{i+1}$ for $i=0,1,\dots, t-1$, and $m_i$ is even whenever $n_i$ is even. Again, we will take $n_t$ to be either $2$ or $3$, and we use the notation $P^{2BD}_{S,M}$ and $P^{3BD}_{S,M}$ to distinguish these cases. For $r,s \geq 0$, we denote by 
$P^{2BD}_{S,M,r,s}$ and $P^{3BD}_{S,M,r,s}$ the partitions of \begin{equation} \label{NRSTypeBD} N_{r,s} = N + s + 2r \sum_{i=0}^t m_i,\end{equation} given by $((n_0+2r)^{m_0}, (n_1+2r)^{m_1},\dots, (n_{t}+2r)^{m_t}, 1^s)$, which has $m_i$ parts of size $n_i + 2r$, and $s$ parts of size $1$. Note that we are using the same notation $N_{r,s}$ in both \eqref{NRSTypeC} and \eqref{NRSTypeBD}, but it will always be clear from context which expression we are using. As above, $P^{2BD}_{S,M,r,s}$ and $P^{3BD}_{S,M,r,s}$ satisfy our restrictions whenever $P^{2BD}_{S,M}$ and $P^{3BD}_{S,M}$ do, and $P^{iBD}_{S,M,0,0} = P_{S,M}$ for $i = 2,3$. Let $f^{2BD}_{S,M,r,s}\in \gs\go_{N_{r,s}}$ and $f^{3BD}_{S,M,r,s}\in \gs\go_{N_{r,s}}$ denote the corresponding nilpotents.

Clearly if $f^{3C}_{S,M,r,s}\in \gs\gp_{N_{r,s}}$ is a well defined nilpotent, so is $f^{2BD}_{S,M,r,s} \in \gs\go_{N_{r,s}}$. Similarly, if $f^{3BD}_{S,M,r,s}\in \gs\go_{N_{r,s}}$ is well defined, so is $f^{2C}_{S,M,r,s} \in \gs\gp_{N_{r,s}}$. It is straightforward to check that the generating types of 
\begin{equation} \label{Cfirst} \cW^k(\gs\gp_{N_{r,0}}, f^{3C}_{S,M,r,0}),\qquad \cW^k(\gs\go_{N_{r,0}}, f^{2BD}_{S,M,r,0}),\end{equation} are the same up an arbitrary fixed conformal weight when $r$ is sufficiently large. In addition, for $r$ sufficiently large, the Lie algebras $\gs\gp_{N_{r,0}}^{\natural}$ with respect to $f^{3C}_{S,M,r,0}$ and $\gs\go_{N_{r,0}}^{\natural}$ with respect to $f^{2BD}_{S,M,r,0}$ are the same. We have 
$$\gs\gp_{N_{r,0}}^{\natural} \cong \bigoplus_{i=0}^t \gg_i,$$ where $\gg_i = \gs\gp_{m_i}$ when $n_i$ is odd (so that $m_i$ is even), and $\gg_i \cong \gs\go_{m_i}$ when $n_i$ is even. Likewise,
$$\gs\go_{N_{r,0}}^{\natural} \cong \bigoplus_{i=0}^t \gg_i,$$ where $\gg_i = \gs\gp_{m_i}$ when $n_i$ is even (so that $m_i$ is even), and $\gg_i \cong \gs\go_{m_i}$ when $n_i$ is odd. Moreover, the action of $\gs\gp_{N_{r,0}}^{\natural}$ and $\gs\go_{N_{r,0}}^{\natural}$ on the higher weight generating fields is independent of $r$. This motivates the following conjecture.
\begin{conjecture}  \label{universal:typeC} Let $S = \{d_1,\dots, d_t\}$ and $M = \{m_0,\dots, m_t\}$ be sets of positive integers such that $n_i - n_{i+1} = d_{i+1}$ for $i=0,1,\dots, t-1$, and $m_i$ is even whenever $n_i$ is odd, and $n_t = 3$. Then there exists a unique $2$-parameter vertex algebra $\cW^{C,S,M}_{\infty}$ with the following properties:
\begin{enumerate}
\item $\cW^{C,S,M}_{\infty}$ is defined over a localization of the polynomial ring in two variables obtained by inverting finitely many polynomials.
\item $\cW^{C,S,M}_{\infty}$ is freely generated of the appropriate type determined by $S$ and $M$.
\item $\cW^{C,S,M}_{\infty}$ admits all the vertex algebras \eqref{Cfirst} as $1$-parameter quotients.
\item $\cW^{C,S,M}_{\infty}$ is weakly generated by the fields in weight at most $3$. 
\item $\cW^{C,S,M}_{\infty}$ is an extension of the tensor product of $m_i$ copies of $\cW^{\text{ev}}_{\infty}$ for each even $n_i$, and $\frac{1}{2} m_i$ copies of $\cW^{\gs\gp}_{\infty}$ for each odd $n_i$.
\item $\cW^{C,S,M}_{\infty}$ serves as a classifying object for vertex algebras satisfying (1)-(4); any vertex algebra with a strong generating set of this type (not necessarily minimal) satisfying (1)-(4), is a quotient of $\cW^{C,S,M}_{\infty}$. 
\end{enumerate}
\end{conjecture}

Similarly, the generating types of 
\begin{equation} \label{BDfirst} \cW^k(\gs\gp_{N_{r,0}}, f^{2C}_{S,M,r,0}),\qquad \cW^k(\gs\go_{N_{r,0}}, f^{3BD}_{S,M,r,0}),\end{equation} are the same up an arbitrary fixed conformal weight when $r$ is sufficiently large, but they are different from the generating type of \eqref{Cfirst}. As above, the Lie algebras $\gs\gp_{N_{r,0}}^{\natural}$ and $\gs\go_{N_{r,0}}^{\natural}$ are the same, and the action of $\gs\gp_{N_{r,0}}^{\natural}$ and $\gs\go_{N_{r,0}}^{\natural}$ on the higher weight generating fields is independent of $r$ for $r$ sufficiently large. 
\begin{conjecture}   \label{universal:typeBD}  Let $S = \{d_1,\dots, d_t\}$ and $M = \{m_0,\dots, m_t\}$ be sets of positive integers such that $n_i - n_{i+1} = d_{i+1}$ for $i=0,1,\dots, t-1$, and $m_i$ is even whenever $n_i$ is even, and $n_t = 3$. There exists a unique $2$-parameter vertex algebra $\cW^{BD,S,M}_{\infty}$ with the following properties:
\begin{enumerate}
\item $\cW^{BD,S,M}_{\infty}$ is defined over a localization of the polynomial ring in two variables obtained by inverting finitely many polynomials.
\item $\cW^{BD,S,M}_{\infty}$ is freely generated of the appropriate type determined by $S$ and $M$.
\item $\cW^{BD,S,M}_{\infty}$ admits all the $\cW$-algebras \eqref{BDfirst}, as $1$-parameter quotients.
\item $\cW^{BD,S,M}_{\infty}$ is weakly generated by the fields in weight at most $4$. 
\item $\cW^{BD,S,M}_{\infty}$ is an extension of the tensor product of $m_i$ copies of $\cW^{\text{ev}}_{\infty}$ for each odd $n_i$, and $\frac{1}{2} m_i$ copies of $\cW^{\gs\gp}_{\infty}$ for each even $n_i$.
\item $\cW^{BD,S,M}_{\infty}$ serves as a classifying object for vertex algebras satisfying (1)-(4); any vertex algebra with a strong generating set of this type (not necessarily minimal) satisfying (1)-(4), is a quotient of $\cW^{BD,S,M}_{\infty}$. 
\end{enumerate}
\end{conjecture}
In the case $S = \emptyset$ and $M = \{1\}$, $\cW^{BD,\emptyset,\{1\}}_{\infty}$ is exactly the algebra $\cW^{\text{ev}}_{\infty}$ constructed in \cite{Lin}. In the case $S = \emptyset$ and $M = \{2\}$, $\cW^{C,\emptyset,\{2\}}_{\infty}$ is exactly the algebra $\cW^{\gs\gp}_{\infty}$ constructed in \cite{CKoL2}. We expect that $\cW^{C,S,M}_{\infty}$ and $\cW^{BD,S,M}_{\infty}$ are weakly generated by the fields in weights $1$ and $2$ except in the cases $\cW^{\gs\gp}_{\infty}$ and $\cW^{\text{ev}}_{\infty}$, where we also need weight $3$ and weight $4$ fields, respectively.
Note that \eqref{Cfirst} and \eqref{BDfirst} are the analogue of the $1$-parameter quotients $\cC^{\psi}_{1D}(0,m)$ and $\cC^{\psi}_{2C}(0,m)$ of $\cW^{\text{ev}}_{\infty}$ defined in \cite{CL4}, for $m\geq 1$.

\subsection{Families of VOAs with generating type of $\cW^{C,S,M}_{\infty}$} Here we write down more vertex algebras which we expect to arise as $1$-parameter quotients of $\cW^{C,S,M}_{\infty}$. We use the same notation $ f^{3C}_{S,M,r,s}$ to denote this nilpotent in the subalgebra $\gs\gp_{N_{r,s}} \subseteq \go\gs\gp_{1|N_{r,s}}$. Similarly, for $s = 0$, we use the notation $f^{3C}_{S,M,r,0}$ to denote this nilpotent inside the subalgebra $\gs\gp_{N_{r,0}} \subseteq \go\gs\gp_{s|N_{r,0}}$ for $s >1$. We use the notation $f^{2BD}_{S,M,r,0}$ to denote this nilpotent in $\gs\go_{N_{r,0}} \subseteq \go\gs\gp_{N_{r,0}|2s}$. 
As usual, by passing to the large level limit and using classical invariant theory, it can be checked that following vertex algebras all have the generating type of $\cW^{C,S,M}_{\infty}$, and should therefore arise as $1$-parameter quotients:
\begin{equation} \label{Csecond} \begin{split}
& \text{Com}(V^{k'}(\gs\gp_{2s}), \cW^k(\gs\gp_{N_{r,s}}, f^{3C}_{S,M,r,s})),
\\ & \text{Com}(V^{k'}(\gs\gp_{2s}), \cW^k(\go\gs\gp_{N_{r,0}|2s}, f^{2BD}_{S,M,r,0})),
\\ & \text{Com}(V^{k'}(\gs\go_s), \cW^k(\gs\go_{N_{r,s}}, f^{2BD}_{S,M,r,s}))^{\mathbb{Z}_2},
\\ & \text{Com}(V^{k'}(\gs\go_s), \cW^k(\go\gs\gp_{s|N_{r,0}}, f^{3C}_{S,M,r,0}))^{\mathbb{Z}_2},
\\ & \text{Com}(V^{k'}(\go\gs\gp_{1|2s}), \cW^k(\go\gs\gp_{1|N_{r,s}}, f^{3C}_{S,M,r,s}))^{\mathbb{Z}_2},
\\ & \text{Com}(V^{k'}(\go\gs\gp_{1|2s}), \cW^k(\go\gs\gp_{N_{r,0} + 1|2s}, f^{2BD}_{S,M,r,0}))^{\mathbb{Z}_2}.
\end{split} \end{equation}
Note that we recover \eqref{Cfirst} by taking $s = 0$ in the first family, and $s= 0$ in the second family in \eqref{Csecond}. These are analogues of the $1$-parameter quotients $\cC^{\psi}_{iX}(n,m)$ of $\cW^{\text{ev}}_{\infty}$ defined in \cite{CL4}, for $m\geq 1$ and $n\geq 1$. As in \cite{CL4}, we regard \eqref{Csecond} as a list of $8$ families, since the cases $s$ even or odd in the third and fourth families should be distinguished. 

We also have families obtained by setting $r = -1$ in the case of $ f^{3C}_{S,M,r,s}$. This replaces $n_i$ with $n_i -2$, so since $n_t = 3$, it is replaced with $1$. These are the following:
\begin{equation} \label{Cthird} \begin{split}
& \text{Com}(V^{k'}(\gs\gp_{2s}), \cW^k(\gs\gp_{N_{-1,s}}, f^{3C}_{S,M,-1,s})),
\\ & \text{Com}(V^{k'}(\go\gs\gp_{1|2s}), \cW^k(\go\gs\gp_{1|N_{-1,s}}, f^{3C}_{S,M,-1,s}))^{\mathbb{Z}_2},
\\ & \text{Com}(V^{k'}(\gs\go_s), \cW^k(\go\gs\gp_{s|N_{-1,0}}, f^{3C}_{S,M,-1,0}))^{\mathbb{Z}_2}.
\end{split} \end{equation} These are the analogues of the $1$-parameter quotients $\cC^{\psi}_{1B}(n,0)$, $\cC^{\psi}_{1D}(n,0)$, $\cC^{\psi}_{1C}(n,0)$, and $\cC^{\psi}_{1O}(n,0)$ of $\cW^{\text{ev}}_{\infty}$, which are ($\mathbb{Z}_2$-orbifolds of) affine cosets of affine vertex (super)algebras. We regard \eqref{Cthird} as a list of $4$ families, since the cases $s$ even or odd in the third family should be distinguished. 

Finally, we have families obtained by setting $r = -1$ in the case  $f^{2BD}_{S,M,r,s}$.
This replaces $n_i$ with $n_i -2$, so in particular $n_t$ is replaced with $0$. This has the effect of removing $n_t$, so we have replaced the partition $P^{2BD}_{S,M}$ with the partition 
$((n_0-2)^{m_0}, (n_1-2)^{m_1},\dots, (n_{t-1}-2)^{m_{t-1}})$. Note that $m_t$ must be even, and that $\cW^{C,S,M}_{\infty}$ has an affine subVOA of type $\gs\gp_{m_t}$. We use the notation $f^{2BD}_{S,M,-1,s}$ to denote the corresponding nilpotent in $\gs\go_{N_{-1,s}}$, and we use the notation $ f^{2BD}_{S,M,-1,0}$ to denote this nilpotent in $\gs\go_{N_{-1,0}}$ regarded as a Lie subalgebra of $\go\gs\gp_{N_{-1,0}|2s}$ or $\go\gs\gp_{1+N_{-1,0}|2s}$. Observe that 
\begin{enumerate}
\item $\cW^k(\gs\go_{N_{-1,s}}, f^{2BD}_{S,M,-1,s})$, $\cW^k(\go\gs\gp_{N_{-1,0}|2s}, f^{2BD}_{S,M,-1,0})$, and $\cW^k(\go\gs\gp_{1+N_{-1,0}|2s}, f^{2BD}_{S,M,-1,0})$ have affine subVOAs of type $\gs\go_s$, $\gs\gp_{2s}$, and $\go\gs\gp_{1|2s}$ for some shifted levels $k'$,
\item The $\beta\gamma$-system $\cS(\frac{1}{2} m_t s)$ has an action of $V^{- m_t}(\gs\go_s) \otimes V^{-s/2}(\gs\gp_{m_t})$,
\item The free fermion algebra $\cF(2m_t s)$ has an action of $V^{m_t/2}(\gs\gp_{2s}) \otimes V^{s}(\gs\gp_{m_t})$,
\item The tensor product $\cF(2m_t s) \otimes \cS(m_t)$ has an action of $V^{m_t/2}(\go\gs\gp_{1|2s})\otimes V^{s-1/2}(\gs\gp_{m_t})$. 
\end{enumerate}
The following diagonal cosets are easily seen to have the generating type of $\cW^{C,S,M}_{\infty}$.
\begin{equation} \label{CGKO}
\begin{split}
& \text{Com}(V^{k'-m_t}(\gs\go_s), \cW^k(\gs\go_{N_{-1,s}}, f^{2BD}_{S,M,-1,s})\otimes \cS(\frac{1}{2} m_t s))^{\mathbb{Z}_2},
\\ &  \text{Com}(V^{k'+m_t}(\gs\gp_{2s}), \cW^k(\go\gs\gp_{N_{-1,0}|2s}, f^{2BD}_{S,M,-1,0}) \otimes \cF(2m_t s)),
\\ & \text{Com}(V^{k'+m_t}(\go\gs\gp_{1|2s}), \cW^k(\go\gs\gp_{1+N_{-1,0}|2s}, f^{2BD}_{S,M,-1,0}) \otimes \cF(2m_t s) \otimes \cS(m_t))^{\mathbb{Z}_2}.
\end{split}\end{equation}
Note that for all $k$, the level of the affine subVOA of type $\gs\gp_{2m_t}$ is $-\frac{s}{2}$, $s$, $s - \frac{1}{2}$, respectively, in the above cosets.
These are the analogues of the $1$-parameter quotients $\cC^{\psi}_{2B}(n,0)$, $\cC^{\psi}_{2D}(n,0)$, $\cC^{\psi}_{2C}(n,0)$, and $\cC^{\psi}_{2O}(n,0)$ of $\cW^{\text{ev}}_{\infty}$ defined in \cite{CL4}, which we regard as GKO cosets. We regard \eqref{CGKO} as a list of $4$ families, since the cases $s$ even or odd in the first family should be distinguished. 

\begin{example} \label{rectC} Consider $S = \emptyset$ and $M = \{2m\}$, so $P^{2BD}_{\emptyset, \{2m\}} = (2^{2m})$ and $P^{3C}_{\emptyset, \{2m\}} = (3^{2m})$. We say that $\cW^{C,\emptyset, \{2m\}}_{\infty}$ is of rectangular type, and it has generating type 
$$\cW(1^{2m^2+m}, 2^{2m^2-m},3^{2m^2+m}, 4^{2m^2-m},\dots),$$ where the affine subVOA is of type $\gs\gp_{2m}$. The $2m^2+m$ fields in each odd weight $d \geq 3$ transform under $\gs\gp_{2m}$ as the adjoint module, i.e., $V_{\lambda}$ for $\lambda = 2\omega_1$, and the $2m^2-m$ fields in each even weight $d \geq 2$ transform as the trivial module plus the irreducible module $V_{\lambda}$ for $\lambda = \omega_2$.
The GKO cosets \eqref{CGKO} in this example are
\begin{equation} \label{CGKORect}
\begin{split}
& \text{Com}(V^{k-2m}(\gs\go_s), V^k(\gs\go_{s})\otimes \cS(\frac{1}{2} m s))^{\mathbb{Z}_2},
\\ &  \text{Com}(V^{k+m}(\gs\gp_{2s}), V^k(\gs\gp_{2s}) \otimes \cF(2m s)),
\\ & \text{Com}(V^{k+m}(\go\gs\gp_{1|2s}), V^k(\go\gs\gp_{1|2s}) \otimes \cF(2m s) \otimes \cS(m))^{\mathbb{Z}_2}.
\end{split}\end{equation}
As above, the level of the affine $\gs\gp_{2m}$ is $-\frac{s}{2}$, $s$, $s - \frac{1}{2}$, respectively, in the above cosets.

When $k$ is admissible for $\gs\gp_{2s}$, we have an embedding of simple vertex algebras $L_{k+m}(\gs\gp_{2s}) \hookrightarrow L_k(\gs\gp_{2s}) \otimes \cF(2m s)$, and the coset 
\begin{equation} \label{typeCrectsimple} \text{Com}(L_{k+m}(\gs\gp_{2s}), L_k(\gs\gp_{2s}) \otimes \cF(2m s)),\end{equation} is an extension of the tensor product of the strongly rational vertex algebras 
$\text{Com}(L_{k + i}(\gs\gp_{2s}), L_{k+i-1}(\gs\gp_{2s}) \otimes \cF(4s))$ for $i = 1,\dots,\frac{m}{2}$ \cite{CKoL2}, and hence is strongly rational. We expect that \eqref{typeCrectsimple} is the simple quotient of $\cW^{C,\emptyset,\{2m\}}_{\infty}$ at all these points, so that $\cW^{C,\emptyset,\{2m\}}_{\infty}$ has many quotients which are strongly rational. In the case $m=1$, where $\cW^{C,\emptyset, \{2m\}}_{\infty} \cong \cW^{\gs\gp}_{\infty}$, this was proven in \cite{CKoL2}.

As shown in the proof of \cite[Corollary 7.6]{CL4}, the cosets \eqref{typeCrectsimple} satisfy the following level-rank duality:
$$
\text{Com}(L_s(\gs\gp_{2k}),L_s(\gs\gp_{2(k+m)})) 
= \text{Com}(L_{k+m}(\gs\gp_{2s}),L_{k}(\gs\gp_{2s}) \otimes \cF(2ms)).
$$
These cosets on the left-hand side appear in the higher spin  AdS/CFT duality of orthosymplectic rectangular type \cite{CHU}. 
\end{example}

\begin{example} Consider $S = \{2\}$ and $M = \{4,6\}$, so $P^{2BD}_{\{2\}, \{4,6\}} = (4^4, 2^6)$ and $P^{3C}_{\{2\}, \{4,6\}} = (5^4, 3^6)$. The generating type of $\cW^{C,\{2\}, \{4,6\}}_{\infty}$ is 
$$\cW(1^{31}, 2^{45},3^{55},4^{45},5^{55},\dots),$$ where the affine subVOA is of type $\gs\gp_4 \oplus \gs\gp_6$. The $55$ fields in weight $d$ for each odd $d \geq 3$ transform under $\gs\gp_4 \oplus \gs\gp_6$ as 
$$V_{2 \omega_1} \otimes \mathbb{C} \oplus \mathbb{C} \otimes V_{2\omega_1}  \oplus V_{\omega_1} \otimes V_{\omega_1}.$$ i.e. the sum of the adjoint representations of $\gs\gp_4$ and $\gs\gp_6$, and the product of the standard modules $\mathbb{C}^4 \otimes \mathbb{C}^6$. The $45$ fields in each even weight $d \geq 2$ transform under $\gs\gp_4 \oplus \gs\gp_6$ as $$\mathbb{C} \otimes \mathbb{C} \oplus  \mathbb{C} \otimes \mathbb{C}  \oplus V_{\omega_2} \otimes \mathbb{C} \oplus \mathbb{C} \otimes V_{\omega_2} \oplus V_{\omega_1} \otimes V_{\omega_1}.$$
The partition obtained by replacing $r$ with $-2$ is just $(2^4)$. Therefore the GKO cosets \eqref{CGKO} in this example are
\begin{equation} \label{CGKOExample2}
\begin{split}
& \text{Com}(V^{k'-6}(\gs\go_s), \cW^k(\gs\go_{8+s}, f_{2^4})\otimes \cS(3 s))^{\mathbb{Z}_2}, \qquad k' = k+4,
\\ &  \text{Com}(V^{k'+3}(\gs\gp_{2s}), \cW^k(\go\gs\gp_{8|2s}, f_{2^4}) \otimes \cF(12 s)),  \qquad k' = -\frac{k}{2} -2,
\\ & \text{Com}(V^{k'+3}(\go\gs\gp_{1|2s}), \cW^k(\go\gs\gp_{9|2s}, f_{2^4}) \otimes \cF(12 s) \otimes \cS(6))^{\mathbb{Z}_2},  \qquad k' = -\frac{k}{2} -2.
\end{split}\end{equation}
The level of the affine $\gs\gp_{6}$ is $-\frac{s}{2}$, $s$, $s - \frac{1}{2}$, respectively, in the above cosets.
\end{example}

\subsection{Families of VOAs with generating type of $\cW^{BD,S,M}_{\infty}$} Here we give a similar list of vertex algebras which we expect to arise as $1$-parameter quotients of $\cW^{BD,S,M}_{\infty}$. 
We use the notation $ f^{2C}_{S,M,r,s}$, $ f^{2C}_{S,M,r,0}$, and $f^{3BD}_{S,M,r,0}$ to denote the image of these nilpotent in $\go\gs\gp_{1|N_{r,s}}$, $\go\gs\gp_{s |N_{r,0}}$ and $\go\gs\gp_{N_{r,0}|2s}$ respectively.
The following vertex algebras have this generating type, and are analogues of the algebras $\cC^{\psi}_{iX}(n,m)$ in \cite{CL4}, for $n, m \geq 1$:
\begin{equation}  \label{BDsecond} \begin{split}
& \text{Com}(V^{k'}(\gs\gp_{2s}), \cW^k(\gs\gp_{N_{r,s}}, f^{2C}_{S,M,r,s})),
\\ & \text{Com}(V^{k'}(\gs\gp_{2s}), \cW^k(\go\gs\gp_{N_{r,0}|2s}, f^{3BD}_{S,M,r,0})),
\\ & \text{Com}(V^{k'}(\gs\go_s), \cW^k(\gs\go_{N_{r,s}}, f^{3BD}_{S,M,r,s}))^{\mathbb{Z}_2},
\\ & \text{Com}(V^{k'}(\gs\go_s), \cW^k(\go\gs\gp_{s|N_{r,0}}, f^{2C}_{S,M,r,0}))^{\mathbb{Z}_2},
\\ & \text{Com}(V^{k'}(\go\gs\gp_{1|2s}), \cW^k(\go\gs\gp_{1|N_{r,s}}, f^{2C}_{S,M,r,s}))^{\mathbb{Z}_2},
\\ & \text{Com}(V^{k'}(\go\gs\gp_{1|2s}), \cW^k(\go\gs\gp_{1+N_{r,0} |2s}, f^{3BD}_{S,M,r,s}))^{\mathbb{Z}_2}.
\end{split} \end{equation}
We regard \eqref{BDsecond} as a list of $8$ families, since the cases $s$ even or odd in the third and fourth families should be distinguished.

We also have families obtained by setting $r = -1$ in the case of $f^{3BD}_{S,M,r,s}$:
\begin{equation} \label{BDthird} \begin{split}
\\ & \text{Com}(V^{k'}(\gs\go_s), \cW^k(\gs\go_{N_{r,s}}, f^{3BD}_{S,M,-1,s}))^{\mathbb{Z}_2},
\\ & \text{Com}(V^{k'}(\gs\gp_{2s}), \cW^k(\go\gs\gp_{N_{r,0}|2s}, f^{3BD}_{S,M,-1,0})),
\\ & \text{Com}(V^{k'}(\go\gs\gp_{1|2s}), \cW^k(\go\gs\gp_{1+N_{r,0}|2s}, f^{3BD}_{S,M,-1,0}))^{\mathbb{Z}_2}.
\end{split} \end{equation} They are analogues of the algebras $\cC^{\psi}_{1B}(n,0)$, $\cC^{\psi}_{1D}(n,0)$, $\cC^{\psi}_{1C}(n,0)$, and $\cC^{\psi}_{1O}(n,0)$ in \cite{CL4}. We regard \eqref{BDthird} as a list of $4$ families, since the cases $s$ even or odd in the first family should be distinguished. 

We also have the families which are obtained by setting $r = -1$ in the case of $f^{2C}_{S,M,r,s}$. This replaces $n_i$ with $n_i -2$, and since $n_t = 2$, it is removed and we have replaced the partition $P^{2C}_{S,M}$ with the partition $((n_0-2)^{m_0}, (n_1-2)^{m_1},\dots, (n_{t-1}-2)^{m_{t-1}})$. Note that $\cW^{BD,S,M}_{\infty}$ has an affine subVOA of type $\gs\go_{m_t}$, and $m_t$ can be even or odd. We denote by $f^{2C}_{S,M,-1,s}$ the corresponding nilpotent in $\gs\gp_{N_{-1,s}}\subseteq \go\gs\gp_{1|N_{-1,s}}$, and we denote by $f^{2C}_{S,M,-1,0}$ the corresponding nilpotent in $\gs\gp_{N_{-1,0}}\subseteq \go\gs\gp_{s|N_{-1,0}}$. As above,
\begin{enumerate}
\item $\cW^k(\gs\gp_{N_{-1,s}}, f^{2C}_{S,M,-1,s})$, $\cW^k(\go\gs\gp_{1|N_{-1,s}}, f^{2C}_{S,M,-1,s})$, and $\cW^k(\go\gs\gp_{s|N_{-1,0}}, f^{2C}_{S,M,-1,0})$ have affine subVOAs of type $\gs\gp_{2s}$, $\go\gs\gp_{1|2s}$, and $\gs\go_s$ for some shifted levels $k'$,
\item The $\beta\gamma$-system $\cS(s m_t)$ has an action of $V^{- m_t/2}(\gs\gp_{2s}) \otimes V^{-2s}(\gs\go_{m_t})$,
\item The free fermion algebra $\cF(s m_t)$ has an action of $V^{m_t}(\gs\go_s) \otimes V^s(\gs\go_{m_t})$,
 \item The tensor product $ \cS(s m_t) \otimes \cF(m_t)$ has an action of $V^{- m_t / 2}(\go\gs\gp_{1|2s}) \otimes V^{-2s + 1}(\gs\go_{m_t})$.
 \end{enumerate}
The following diagonal cosets have the generating type of $\cW^{BD,S,M}_{\infty}$:
\begin{equation} \begin{split} \label{BDGKO} 
& \text{Com}(V^{k' - m_t/2}(\gs\gp_{2s}), \cW^k(\gs\gp_{N_{-1,s}}, f^{2C}_{S,M,-1,s}) \otimes \cS(m_t s)),
\\ & \text{Com}(V^{k' - m_t / 2}(\go\gs\gp_{1|2s}), \cW^k(\go\gs\gp_{1|N_{-1,s}}, f^{2C}_{S,M,-1,s}) \otimes \cS(m_t s) \otimes \cF(m_t))^{\mathbb{Z}_2},
\\ & \text{Com}(V^{k' + m_t}(\gs\go_s), \cW^k(\go\gs\gp_{s|N_{-1,0}}, f^{2C}_{S,M,-1,0}) \otimes \cF(m_t s))^{\mathbb{Z}_2}.
\end{split} \end{equation}
Note that for all $k$, the level of the affine subVOA of type $\gs\go_{m_t}$ is $-2s$, $-2s+1$, and $s$, respectively, in the above cosets. These are analogues of the GKO cosets $\cC^{\psi}_{2B}(n,0)$, $\cC^{\psi}_{2D}(n,0)$, $\cC^{\psi}_{2C}(n,0)$, and $\cC^{\psi}_{2O}(n,0)$ in \cite{CL4}, and we regard \eqref{BDGKO} as a list of $4$ families, since the cases $s$ even or odd in the third family should be distinguished.

\begin{example} \label{rectBD} Consider $S = \emptyset$ and $M = \{m\}$, so $P^{2C}_{\emptyset, \{m\}} = (2^m)$ and $P^{3BD}_{\emptyset, \{m\}} = (3^m)$. As above, we say that $\cW^{BD,\emptyset, \{m\}}_{\infty}$ is of rectangular type, and it has generating type 
$$\cW(1^{m(m-1)/2}, 2^{m(m+1)/2},3^{m(m-1)/2}, 4^{m(m+1)/2},\dots),$$ where the affine subVOA is of type $\gs\go_{m}$. In addition, the $m(m-1)/2$ fields in each odd weight $d \geq 3$ transform as the adjoint module, i.e., $V_{\lambda}$ for $\lambda = \omega_2$, and the $m(m+1)/2$ fields in each even weight $d \geq 2$ transform as the trivial module plus the irreducible module $V_{\lambda}$ for $\lambda = 2 \omega_1$.
The GKO cosets \eqref{BDGKO} in this example are
\begin{equation} \begin{split} \label{BDGKORect} 
& \text{Com}(V^{k - m/2}(\gs\gp_{2s}), V^k(\gs\gp_{2s}) \otimes \cS(m s)),
\\ & \text{Com}(V^{k - m / 2}(\go\gs\gp_{1|2s}), V^k(\go\gs\gp_{1|2s}) \otimes \cS(m s) \otimes \cF(m))^{\mathbb{Z}_2},
\\ & \text{Com}(V^{k + m}(\gs\go_s), V^k(\gs\go_{s}) \otimes \cF(m s))^{\mathbb{Z}_2}.
\end{split} \end{equation}
As above, the level of the affine $\gs\go_{m}$ is $-2s$, $-2s+1$, and $s$, respectively, in the above cosets.

When $k$ is admissible for $\gs\go_{s}$, we have an embedding of simple vertex algebras $L_{k + m}(\gs\go_s) \hookrightarrow L_k(\gs\go_{s}) \otimes \cF(s m)$, and the coset 
\begin{equation} \label{typeBDrectsimple} \text{Com}(L_{k + m}(\gs\go_s), L_k(\gs\go_{s}) \otimes \cF(s m)),\end{equation} is an extension of the tensor product of the vertex algebras 
$\text{Com}(L_{k + i}(\gs\go_{s}), L_{k+i-1}(\gs\go_{s}) \otimes \cF(s))$ for $i = 1,\dots,m$. These are strongly rational when $s$ is even \cite{ACL}, and conjecturally when $s$ is odd \cite[Conjecture 7.1]{CL4}. This would imply that \eqref{typeBDrectsimple}, as well as its $\mathbb{Z}_2$-orbifold, is strongly rational. Again, this implies that $\cW^{BD,\emptyset,\{m\}}_{\infty}$ has many quotients which are strongly rational. 
In the case $m = 2$, we will prove this in Section \ref{sect:rationality} for most admissible levels; see Corollary \ref{cor:rationalityatell}.
\end{example}

\begin{example} Consider $S = \{3\}$ and $M = \{2,3\}$, so $P^{2C}_{\{2\}, \{2,3\}} = (5^2, 2^3)$ and $P^{3B}_{\{2\}, \{2,3\}} = (6^2, 3^3)$. The generating type of $\cW^{BD,\{3\}, \{2,3\}}_{\infty}$ is 
$$\cW\bigg(1^{6}, 2^{7}, \bigg(\frac{5}{2}\bigg)^6, 3^{6},\bigg(\frac{7}{2}\bigg)^6, 4^{7}, \bigg(\frac{9}{2}\bigg)^6,5^6,\dots\bigg),$$ 
where the affine subVOA is of type $\gs\gp_2 \oplus \gs\go_3$. The $6$ fields in each odd weight $d \geq 3$ transform under $\gs\gp_2 \oplus \gs\go_3$ as the sum of the adjoint representations of $\gs\gp_2$ and $\gs\go_3$. The $7$ fields in each even weight $d \geq 2$ transform trivially under $\gs\gp_2$ and transform under $\gs\go_3$ as two copies of the trivial module and one copy of the $5$-dimensional module $V_{2\omega_1}$. The $6$ fields in each  weight $\frac{2d+1}{2}$ for $d \geq 2$ transform as the product of standard modules $\mathbb{C}^2 \otimes \mathbb{C}^3$.

The partition obtained by replacing $r$ with $-2$ is just $(3^2)$. Therefore the GKO cosets \eqref{BDGKO} in this example are
\begin{equation} \begin{split} \label{BDGKOExample2} 
& \text{Com}(V^{k' - 3/2}(\gs\gp_{2s}), \cW^k(\gs\gp_{6+2s}, f_{3^2}) \otimes \cS(3s)), \qquad k' = k+2,
\\ & \text{Com}(V^{k' - 3 / 2}(\go\gs\gp_{1|2s}), \cW^k(\go\gs\gp_{1|6+2s}, f_{3^2}) \otimes \cS(3s) \otimes \cF(3))^{\mathbb{Z}_2},  \qquad k' = k+2,
\\ & \text{Com}(V^{k' +3}(\gs\go_s), \cW^k(\go\gs\gp_{s|6}, f_{3^2}) \otimes \cF(3s))^{\mathbb{Z}_2},  \qquad k' = -2k-4.
\end{split} \end{equation}
The level of the affine $\gs\go_{3}$ is $-2s$, $-2s+1$, and $s$, respectively, in the above cosets.
\end{example}

\subsection{Rectangular families} We expect the universal objects of rectangular types given by Examples \ref{rectA}, \ref{rectC}, and \ref{rectBD}, to play a special role in the structure of $\cW^{X,S,M}_{\infty}$ for $X = A,C, BD$, because they should form an alternative set of building blocks. More precisely, we expect that $\cW^{A,S,M}_{\infty}$ is a conformal extension of $\cH(t) \otimes \big(\bigotimes_{i=0}^t \cW^{\emptyset, \{m_i\}}_{\infty}\big)$, where $\cH(t)$ is the Heisenberg algebra of rank $t$. Similarly, recall that $\cW^{C,S,M}_{\infty}$ corresponds to the partition $P^{3C}_{S,M} = (n_0^{m_0}, n_1^{m_1},\dots, n_{t}^{m_t})$ where $n_t = 3$. It should be a gluing of the tensor product of $\cW^{C,\emptyset, \{m_i\}}_{\infty}$ for each odd $n_i$ (so that $m_i$ is even), and $\cW^{BD,\emptyset, \{m_i\}}_{\infty}$ for each even $n_i$ (so that $m_i$ can be even or odd). Likewise, $\cW^{BD,S,M}_{\infty}$, which corresponds to the partition $P^{3BD}_{S,M} = (n_0^{m_0}, n_1^{m_1},\dots, n_{t}^{m_t})$ where $n_t = 3$, should be a gluing of the tensor product of $\cW^{C,\emptyset, \{m_i\}}_{\infty}$ for each even $n_i$ (so that $m_i$ is even), and $\cW^{BD,\emptyset, \{m_i\}}_{\infty}$ for each odd $n_i$ (so that $m_i$ can be even or odd). As mentioned above, the case $\cW^{C,\emptyset, \{2\}}_{\infty}$ is exactly the algebra $\cW^{\gs\gp}_{\infty}$ constructed in \cite{CKoL2}, and it is regarded as a fundamental building block on the same footing as $\cW_{\infty}$ and $\cW^{\text{ev}}_{\infty}$. Our main goal in this paper is to construct the first nontrivial example $\cW^{BD,\emptyset, \{2\}}_{\infty}$ which is {\it not} a fundamental building block, but instead is a glueing of two copies of $\cW^{\text{ev}}_{\infty}$. This has affine subVOA of type $\gs\go_2$, and to emphasize the parallels with $\cW^{\gs\gp}_{\infty}$, we will denote it by $\cW^{\gs\go_2}_{\infty}$ for the rest of the paper. In Section \ref{sect:YtypeSO} we will describe in detail the specialization of the family of $1$-parameter vertex algebras \eqref{BDfirst}, \eqref{BDsecond}, and \eqref{BDthird} in this case. We call them {\it $Y$-algebras of $\fr{so}_2$-rectangular type}, and we use the notation $\cC^{\psi}_{XY}(n,m)$ for $X = B,C$ and $Y = B,D,C,O$ to emphasize the parallels with the $Y$-algebras of $\gs\gp_2$-rectangular type in \cite{CKoL2}. Similarly, in Section \ref{sect:Diagonal}, we will describe the specialization of the GKO cosets \eqref{BDGKO}, and we use the notation $\cC^{\ell}(n)$ for $n \in \mathbb{Z}$ to denote these algebras in a uniform way.

	\section{$Y$-algebras of $\fr{so}_2$-orthogonal type} \label{sect:YtypeSO}
	In this section, we specialize the list of $8$ families given by \eqref{BDsecond} to the case of $\cW^{BD, \emptyset, \{2\}}_{\infty} :=\cW^{\gs\go_2}_{\infty}$.
	First, let $\gg$ be a simple Lie (super)algebra of type $B$, $C$, or $D$; in particular $\gg$ is either $\fr{so}_{2n+1}$, $\fr{sp}_{2n}$, $\fr{so}_{2n}$ or $\fr{osp}_{1|2n}$.
	We further assume that $\gg$ admits a decomposition as a $\fr{b}\oplus\fr{so}_2\oplus \fr{a}$-module, where 
	\begin{equation}\label{g decomposition}
		\begin{split}
			\gg\cong& \fr{so}_{2m+1}\oplus\fr{so}_2\oplus \fr{a}\oplus 	\rho_{\omega_2}\otimes \mathbb{C}^3\otimes \mathbb{C}\oplus \mathbb{C}^{2m+1}\otimes \mathbb{C}^2\otimes \rho_{\fr{a}},\\
			\gg\cong& \fr{sp}_{2m}\oplus\fr{so}_2\oplus \fr{a}\oplus \rho_{2\omega_1}\otimes \mathbb{C}^3\otimes \mathbb{C}\oplus \mathbb{C}^{2m}\otimes \mathbb{C}^2\otimes \rho_{\fr{a}},\\
		\end{split}
	\end{equation}
	with the following properties.
	\begin{enumerate}
		\item $\fr{a}$ is a Lie sub(super)algebras of $\gg$, and is either $\mathfrak{so}_{2n}$, $\mathfrak{so}_{2n+1}$,$ \mathfrak{sp}_{2n}$, or $\mathfrak{osp}_{1|2n}$.
		\item $\rho_{\fr{a}}$ and $\rho_{\fr{b}}$ have the same parity, which can be even or odd.
		\item $\rho_{\fr{a}}$ is transforms as the standard representation of $\fr{a}$.
	\end{enumerate}

	Note that if $\fr{a} = \fr{osp}_{1|2n}$, $\rho_{\fr{a}}$ even means that $\rho_{\fr{a}}\cong \mathbb{C}^{2n|1}$ as a vector superspace, whereas $\rho_{\fr{a}}$ odd means that $\rho_{\fr{a}}\cong \mathbb{C}^{1|2n}$.
	If $\fr{g} = \fr{osp}_{m|2n}$ we use the following convention for its dual Coxeter number $h^{\vee}$.
	\begin{equation}
		h^{\vee} = 
		\begin{cases}
			m-2n-2 & \T{type } B\\
			n+1-\frac{1}{2}m & \T{type } C\\
		\end{cases},
		\quad \T{sdim}(\fr{osp}_{m|2n}) = \frac{(m-2n)(m-2n-1)}{2}.
	\end{equation}
	In this notation, type $B$ (respectively $C$) means that the subalgebra $\fr{b} \subseteq \fr{g}$ is of type $B$ (respectively $C$), and the bilinear form on $\fr{osp}_{m|2n}$ is normalized so that it coincides with the normalized Killing form on $\fr{b}$.
	The cases that we need are recorded in Table (\ref{tab:Walgebras}).

	\begin{table}[h]
		\centering
		\caption{$\cW$-algebras of $\fr{so}_2$-rectangular type.}
		\label{tab:Walgebras}
		\begin{tabular}{|l|l|l|l|l|l|}
			\hline
			Case & $\gg$& $\fr{a}$ & $\fr{b}$ & $\rho_{\fr{a}}\otimes\mathbb{C}^2\otimes\rho_{\fr{b}}$& $a$\\
			\hline
			$BD$ & $\fr{so}_{2(2m+1)+2n}$ & $\fr{so}_{2n}$ & $\fr{so}_{2m+1}$ & Even& $\psi-2n$\\
			$BB$ & $\fr{so}_{2(2m+1)+2n+1}$ & $\fr{so}_{2n+1}$ & $\fr{so}_{2m+1}$ & Even& $\psi-2n-1$\\
			$BC$ & $\fr{osp}_{2(2m+1)|2n}$ & $\fr{sp}_{2n}$ & $\fr{so}_{2m+1}$ & Odd &$-\frac{1}{2}\psi - n$\\
			$BO$ & $\fr{osp}_{2(2m+1)+1|2n}$ & $\fr{osp}_{1|2n}$ & $\fr{so}_{2m+1}$ & Odd&$-\frac{1}{2}\psi-n+\frac{1}{2}$ \\
			$CC$ & $\fr{sp}_{2(2m)+2n}$ & $\fr{sp}_{2n}$ & $\fr{sp}_{2m}$ & Even& $\psi-n$\\
			$CD$ & $\fr{osp}_{2n|2(2m)}$ & $\fr{so}_{2n}$ & $\fr{sp}_{2m}$ & Odd &$-2\psi-2n+4$\\
			$CB$ & $\fr{osp}_{2n+1|2(2m)}$ & $\fr{so}_{2n+1}$ & $\fr{sp}_{2m}$ & Odd & $-2\psi-2n+3$\\
			$CO$ & $\fr{osp}_{1|2(2m)+2n}$ & $\fr{osp}_{1|2n}$ & $\fr{sp}_{2m}$ & Even &$\psi-n-\frac{3}{2}$\\
			\hline
		\end{tabular}
	\end{table}

	Let $f_{\fr{b}}$ be the nilpotent element which is principal in $\gb$ and trivial in $\fr{so}_2\oplus \fr{a}$. In our previous notation, in the case $\gb = \gs\go_{2m+1}$, $f_{\gb}$ corresponds to the partition $(2m+1, 2m+1)$, and in the case $\gb = \gs\gp_{2m}$, $f_{\gb}$ corresponds to the partition $(2m, 2m)$. Let $\rho_d$ denote the $(d+1)$-dimensional representation of the $\fr{sl}_2$-triple $\{f_{\fr{b}},x_{\fr{b}},e_{\fr{b}}\}$, and define two $\fr{sl}_2$-modules 
	\begin{equation}\label{evenodd}
		\begin{split}
			\T{Even}(m)= \bigoplus_{i=1}^m \rho_{4i-2},\quad \T{Odd}(m) =  \bigoplus_{i=1}^m \rho_{4i}.
		\end{split}
	\end{equation}
	Recall the decomposition (\ref{g decomposition}).
	Taking $f$ to be principal in $\fr{b}$ we have the following isomorphisms of $\fr{sl}_2$-modules.
	\begin{equation}
		\rho^{\fr{so}_{2n+1}}_{\omega_2} \cong \rho^{\fr{sp}_{2n}}_{2\omega_1}\cong  \T{Even}(n),\quad \rho^{\fr{so}_{2n+1}}_{2\omega_1} \cong \T{Odd}(n),\quad \rho^{\fr{sp}_{2n}}_{\omega_2} \cong\T{Odd}(n-1).\\
	\end{equation}
	To compute the total central charge we have to evaluate contribution of each term in the expression
	\[c=c_{\fr{g}}+c_{\T{dilaton}} + c_{\T{ghost}}.\]
	In types $B$ and $C$ the dilaton contribution reduces to the sum of odd and even squares, respectively.
	\begin{equation*}
		\begin{split}
			c_{\T{dilaton}} =& -\ell\times\begin{cases}
				(m+2)(2m)(2m+1), \quad &\fr{b} = \fr{so}_{2m+1},\\
				(2m-1)(2m)(2m+1), \quad &\fr{b} = \fr{sp}_{2m}.\\
			\end{cases}\\
		\end{split}
	\end{equation*}
	It remains to evaluate the contributions of $c_{\T{ghost}}$.
	Consider the decomposition $\gg=\bigoplus_{d}\rho_{d}$ intro irreducible $\fr{sl}_2$-modules.
	Then each $\rho_d$ gives rise a field of conformal weight $\frac{d+1}{2}$ in $\cW^{\ell}(\gg,f_{\fr{b}})$.
	\begin{equation}\label{ghosts}
		c_d = -\frac{(d-1)(d^2-2d-1)}{2}.
	\end{equation}
	Examining the decomposition (\ref{g decomposition}), we find that the charged fermion contribution consists of three terms
	\begin{equation}\label{ghost}
		c_{\T{ghost}}=c_{\T{odd}}+3c_{\T{even}}+2\T{sdim}(\rho_{\fr{g}})c_{d_{\fr{b}}}.
	\end{equation}
	To compute the contributions $c_{\T{Even}}$ and $c_{\T{Odd}}$ we use formula (\ref{ghosts}) and decomposition (\ref{evenodd}) evaluate the sums
	\[c_{\T{Even}}=\sum_{i=0}^n c_{4i+3}=6m^2-8m^4, \quad c_{\T{Odd}}=\sum_{i=0}^n c_{4i+1}=-2m(m+1)(4m(m+1)-1).\]
	The third term in (\ref{ghost}) is computed by applying formula (\ref{ghosts}).
	
	It follows from above discussion that $\cW^{\ell}(\gg,f_{\fr{b}})$ is of type 
	\begin{equation}
		\begin{split}
			\cW\left(1^{1+\T{dim}(\fr{a})},2^3,3,4^3,\dots,2m-1,(2m)^3, (m+\frac{1}{2})^{2\T{dim}(\rho_{\fr{a}})}\right),\quad &\fr{b}= \fr{sp}_{2m},\\
			\cW\left(1^{1+\T{dim}(\fr{a})},2^3,3,4^3,\dots,(2m)^3,2m+1, \left(m+1\right)^{2\T{dim}(\rho_{\fr{a}})}\right),\quad &\fr{b}= \fr{so}_{2m+1}.
		\end{split}
	\end{equation}
	The affine subalgebra is $V^k(\fr{so}_2)\otimes V^{a}(\fr{a})$ for some level $a$, recorded in the table (\ref{tab:Walgebras}).
	We will always replace $\ell$ with the {critically shifted level} $\psi = \ell + h^{\vee}_{\fr{g}}$, where $h^{\vee}_{\fr{g}}$ is the dual Coxeter number of $\fr{g}$.
	We now describe the examples we need in greater detail.
	
\subsection{Case BD}
	For $\gg =  \fr{so}_{4m+2n+2}$, with $m\geq 0$ and $n\geq 0$, we have an isomorphism of $\fr{so}_{2m+1}\oplus \fr{so}_2\oplus\fr{so}_{2n}$-modules
	\[\fr{so}_{4m+2n+2}\cong\fr{so}_{2m+1}\oplus\fr{so}_2\oplus\fr{so}_{2n} \oplus \rho_{2\omega_1}\otimes \mathbb{C}^3\otimes\mathbb{C} \oplus \mathbb{C}^{2m+1}\otimes \mathbb{C}^2\otimes \mathbb{C}^{2n},\]
	and the critically shifted level is $\psi = \frac{k + 4 m + 2 n}{1 + 2 m}$.
	We define 
	\[\cW^{\psi}_{BD}(n,m) : = \cW^{ \frac{k + 4 m + 2 n}{1 + 2 m} - h^{\vee}}(\fr{so}_{4m+2n},f_{\fr{so}_{2m+1}}), \qquad h^{\vee} = 4m+2n\]
	which has affine subalgebra $V^{k}(\fr{so}_2) \otimes V^{\psi-2n}(\fr{so}_{2n})$.
	Next, we define
	\[\cC^{\psi}_{BD}(n,m) : = \T{Com}(V^{\psi-2n}(\fr{so}_{2n}),\cW^{\psi}_{BD}(n,m))^{\mathbb{Z}_2}.\]
	The conformal element $L-L^{\fr{so}_2} - L^{\fr{so}_{2n}}$ has central charge
	\begin{equation}\label{eq:cBD}
		c_{BD}=-\frac{4 (k-1) (k m-2 m-n) (k m+k+n-1)}{(k+2 n-2) (k+4 m+2 n)}.
	\end{equation}
	The free field limit of $\cW^{\psi}_{BD}(n,m)$ is 
	\[\cO_{\T{ev}}(2n^2-n,2)\otimes \left( \bigotimes_{i=1}^{m}\cO_{\T{ev}}(3,4i)\bigotimes \bigotimes_{i=1}^{m+1}\cO_{\T{ev}}(1,4i-2) \right) \bigotimes \cO_{\T{ev}}(4n,2m+2).\]
		
	\begin{lemma}\label{lemma:BD}
		For $n+m\geq 1$, $\cC^{\psi}_{BD}(n,m)$ is a simple, $1$-parameter vertex algebra of type 
		$\cW(1,2^3,3,4^3,\dots)$.
	\end{lemma}
	
	\begin{proof} By \cite[Lemma 4.2]{CL3}, $\cC^{\psi}_{BD}(n,m)$ has large level limit
	$$ \left( \bigotimes_{i=1}^{m}\cO_{\T{ev}}(3,4i)\bigotimes \bigotimes_{i=1}^{m+1}\cO_{\T{ev}}(1,4i-2) \right) \bigotimes \cO_{\T{ev}}(4n,2m+2)^{\text{O}_{2n}}.$$
	The first factor $( \bigotimes_{i=1}^{m}\cO_{\T{ev}}(3,4i)\bigotimes \bigotimes_{i=1}^{m+1}\cO_{\T{ev}}(1,4i-2))$ has strong generating type $\cW(1,2^3,3,\dots,(2m)^3, 2m+1)$. Using classical invariant theory, the second factor $\cO_{\T{ev}}(4n,2m+2)^{\text{O}_{2n}}$ can be seen to have an infinite strong generating set of type $\cW((2m+2)^3, 2m+3, (2m+4)^3, 2m+5,\dots)$, of which only finitely many are needed. It follows that $\cC^{\psi}_{BD}(n,m)$ has the desired strong generating type. Finally, the simplicity follows from \cite[Theorem 3.6]{CL3}.
	\end{proof}

	\subsection{Case BB}
	For $\gg =  \fr{so}_{4m+2n+3}$, we have an isomorphism of $\fr{so}_{2m+1}\oplus \fr{so}_2\oplus\fr{so}_{2n+1}$-modules
\[\fr{so}_{4m+2n+3}\cong\fr{so}_{2m+1}\oplus\fr{so}_2\oplus\fr{so}_{2n+1} \oplus \rho_{2\omega_1}\otimes \mathbb{C}^3\otimes\mathbb{C} \oplus \mathbb{C}^{2m+1}\otimes \mathbb{C}^2\otimes \mathbb{C}^{2n+1},\]
and the critically shifted level $\psi =\frac{k+4 m+2 n+1}{2 m+1}$.
We define 
\[\cW^{\psi}_{BB}(n,m) : = \cW^{\frac{k+4 m+2 n+1}{2 m+1}-h^{\vee}}(\fr{so}_{4m+2n+3},f_{\fr{so}_{2m+1}}),\qquad h^{\vee} = 4m+2n+1\]
which has affine subalgebra $V^{k}(\fr{so}_2) \otimes V^{\psi-2n-1}(\fr{so}_{2n+1})$.
Next, we define
\[\cC^{\psi}_{BB}(n,m) : = \T{Com}(V^{\psi-2n-1}(\fr{so}_{2n+1}),\cW^{\psi}_{BB}(n,m))^{\mathbb{Z}_2}.\]
The conformal element $L-L^{\fr{so}_2} - L^{\fr{so}_{2n+1}}$ has central charge
\begin{equation}\label{eq:cBB}
	c_{BB}=-\frac{(k-1) (2 k m-4 m-2 n-1) (2 k m+2 k+2 n-1)}{(k+2 n-1) (k+4 m+2 n+1)}.
\end{equation}
The free field limit of $\cW^{\psi}_{BB}(n,m)$ is the free field algebra $\cW^{\infty}_{BB}(n,m)$
\[\cO_{\T{ev}}(2 n^2+n,2)\otimes \left( \bigotimes_{i=1}^{m}\cO_{\T{ev}}(3,4i)\otimes \bigotimes_{i=1}^{m+1}\cO_{\T{ev}}(1,4i-2) \right) \otimes \cO_{\T{ev}}(4n+2,2m+2).\]

\begin{lemma}\label{lemma:BB}
	For $n+m\geq 1$, $\cC^{\psi}_{BB}(n,m)$ is a simple, $1$-parameter vertex algebra of type $\cW(1,2^3,3,4^3,\dots)$. \end{lemma}

	\subsection{Case BC}
	For $\gg =  \fr{osp}_{4m+2|2n}$, we have an isomorphism of $\fr{so}_{2m+1}\oplus \fr{so}_2\oplus\fr{sp}_{2n}$-modules
\[\fr{osp}_{4m+2|2n} \cong \fr{so}_{2m+1}\oplus\fr{so}_2\oplus\fr{sp}_{2n} \oplus \rho_{2\omega_1}\otimes \mathbb{C}^3\otimes\mathbb{C} \oplus \mathbb{C}^{2m+1}\otimes \mathbb{C}^2\otimes \mathbb{C}^{0|2n},\]
and the critically shifted level $\psi = \frac{k+4 m-2 n}{2 m+1}$.
We define 
\[\cW^{\psi}_{BC}(n,m) : = \cW^{\frac{k+4 m-2 n}{2 m+1}-h^{\vee}}(\fr{osp}_{4m+2|2n},f_{\fr{so}_{2m+1}}),\qquad h^{\vee} = 4 m - 2 n\]
which has affine subalgebra $V^{k}(\fr{so}_2) \otimes V^{-\frac{1}{2}\psi - n}(\fr{sp}_{2n})$.
Next, we define its affine coset
\[\cC^{\psi}_{BC}(n,m) : = \T{Com}(V^{-\frac{1}{2}\psi - n}(\fr{sp}_{2n}),\cW^{\psi}_{BC}(n,m)).\]
The conformal element $L-L^{\fr{so}_2} - L^{\fr{sp}_{2n}}$ has central charge
\begin{equation}\label{eq:cBC}
	c_{BC}=\frac{4 (k-1) (k m+k-n-1) (k m-2 m+n)}{(-k+2 n+2) (k+4 m-2 n)}.
\end{equation}
The free field limit of $\cW^{\psi}_{BC}(n,m)$ is the free field algebra $\cW^{\infty}_{BC}(n,m)$
\[\cO_{\T{ev}}(2n^2+n,2)\otimes \left( \bigotimes_{i=1}^{m}\cO_{\T{ev}}(3,4i)\otimes \bigotimes_{i=1}^{m+1}\cO_{\T{ev}}(1,4i-2) \right) \otimes \cS_{\T{odd}}(2n,2m+2).\]

\begin{lemma}\label{lemma:BC}
	For $n+m\geq 1$, $\cC^{\psi}_{BC}(n,m)$ is a simple, $1$-parameter vertex algebra of type $\cW(1,2^3,3,4^3,\dots)$. \end{lemma}

	\subsection{Case BO}
		For $\gg =  \fr{osp}_{4m+3|2n}$, we have an isomorphism of $\fr{so}_{2m+1}\oplus \fr{so}_2\oplus\fr{osp}_{1|2n}$-modules
	\[\fr{osp}_{4m+3|2n} \cong \fr{so}_{2m+1}\oplus\fr{so}_2\oplus\fr{osp}_{1|2n} \oplus \rho_{2\omega_1}\otimes \mathbb{C}^3\otimes\mathbb{C} \oplus \mathbb{C}^{2m+1}\otimes \mathbb{C}^2\otimes \mathbb{C}^{1|2n},\]
	and the critically shifted level $\psi = \frac{k+4 m+2 n+1}{2 m+1}$.
	We define 
	\[\cW^{\psi}_{BO}(n,m) : = \cW^{\frac{k+4 m+2 n+1}{2 m+1}-h^{\vee}}( \fr{osp}_{4m+3|2n}, f_{\fr{so}_{2m+1}}),\qquad h^{\vee} = 4m-2n+1\]
	which has affine subalgebra $V^{k}(\fr{so}_2) \otimes V^{-\frac{1}{2}\psi-n+\frac{1}{2}}(\fr{osp}_{1|2n})$.
	Next, we define
	\[\cC^{\psi}_{BO}(n,m) : = \T{Com}(V^{-\frac{1}{2}\psi-n+\frac{1}{2}}(\fr{osp}_{1|2n}),\cW^{\psi}_{BO}(n,m))^{\mathbb{Z}_2}.\]
	The conformal element $L-L^{\fr{so}_2} - L^{\fr{osp}_{1|2n}}$ has central charge
	\begin{equation}\label{eq:cBO}
		c_{BO}=\frac{(k-1) (2 k m+2 k-2 n-1) (2 k m-4 m+2 n-1)}{(-k+2 n+1) (k+4 m-2 n+1)}.
	\end{equation}
	The free field limit of $\cW^{\psi}_{BO}(n,m)$ is the free field algebra $\cW^{\infty}_{BO}(n,m)$
	\[\cO_{\T{ev}}(2n^2+n,2)\otimes \cO_{\T{odd}}(4n,2)\otimes \left( \bigotimes_{i=1}^{m}\cO_{\T{ev}}(3,4i)\otimes \bigotimes_{i=1}^{m+1}\cO_{\T{ev}}(1,4i-2) \right) \otimes \cS_{\T{ev}}(1,2m+2)\otimes \cS_{\T{odd}}(2n,2m+2).\]
	\begin{lemma}\label{lemma:BO}
		For $n+m\geq 1$, $\cC^{\psi}_{BO}(n,m)$ is of type 
		$\cW(1,2^3,3,4^3,\dots)$
		as a 1-parameter vertex algebra.
		Equivalently, this holds for generic values of $\psi$.
	\end{lemma}

	\subsection{Case CC}
For $\gg =  \fr{sp}_{4m+2n}$, we have an isomorphism of $\fr{sp}_{2m}\oplus \fr{so}_2\oplus\fr{sp}_{2n}$-modules
\[ \fr{sp}_{4m+2n} \cong \fr{sp}_{2m}\oplus\fr{so}_2\oplus\fr{sp}_{2n} \oplus \rho_{\omega_2}\otimes \mathbb{C}^3\otimes\mathbb{C} \oplus \mathbb{C}^{2m}\otimes \mathbb{C}^2\otimes \mathbb{C}^{2n},\]
and the critically shifted level $\psi =\frac{k-4 m+2 n}{4 m}$.
We define 
\[\cW^{\psi}_{CC}(n,m) : = \cW^{\frac{k-4 m+2 n}{4 m}- h^{\vee}}(\fr{sp}_{4m+2n},f_{\fr{sp}_{2m}}), \qquad h^{\vee} = 2m+n+1,\]
which has affine subalgebra $V^{k}(\fr{so}_2) \otimes V^{\psi-n}(\fr{sp}_{2n})$.
Next, we define
\[\cC^{\psi}_{CC}(n,m) : = \T{Com}(V^{\psi-n}(\fr{sp}_{2n},\cW^{\psi}_{CC}(n,m)).\]
The conformal element $L-L^{\fr{so}_2} - L^{\fr{sp}_{2n}}$ has central charge
\begin{equation}\label{eq:cCC}
	c_{CC}=-\frac{(k-1) (2 k m-k-4 m-2 n) (2 k m+k+2 n)}{(k+2 n) (k+4 m+2 n)}.
\end{equation}
The free field limit of $\cW^{\psi}_{CC}(n,m)$ is the free field algebra $\cW^{\infty}_{CC}(n,m)$
\[\cO_{\T{ev}}(2n^2+n,2)\otimes \left( \bigotimes_{i=1}^{m}\cO_{\T{ev}}(3,4i)\otimes \bigotimes_{i=1}^{m}\cO_{\T{ev}}(1,4i-2) \right) \otimes \cS_{\T{ev}}(2n,2m+1).\]
\begin{lemma}\label{lemma:CC}
	For $n+m\geq 1$, $\cC^{\psi}_{CC}(n,m)$ is of type 
	$\cW(1,2^3,3,4^3,\dots)$
	as a 1-parameter vertex algebra.
	Equivalently, this holds for generic values of $\psi$.
\end{lemma}

	\subsection{Case CB}
For $\gg =  \fr{osp}_{2n+1|4m}$, we have an isomorphism of $\fr{sp}_{2m}\oplus \fr{so}_2\oplus\fr{so}_{2n+1}$-modules
\[ \fr{osp}_{2n+1|4m}\cong \fr{sp}_{2m}\oplus\fr{so}_2\oplus\fr{so}_{2n+1} \oplus \rho_{\omega_2}\otimes \mathbb{C}^3\otimes\mathbb{C} \oplus \mathbb{C}^{2m}\otimes \mathbb{C}^2\otimes \mathbb{C}^{2n+1},\]
and the critically shifted level $\psi =\frac{k+4 m-2 n-1}{4 m}$.
We define 
\[\cW^{\psi}_{CB}(n,m) : = \cW^{\frac{k+4 m-2 n-1}{4 m}-h^{\vee}}(\fr{osp}_{2n+1|4m},f_{\fr{sp}_{2m}}), \qquad h^{\vee} = 2 m-n+\frac{1}{2}\]
which has affine subalgebra $V^{k}(\fr{so}_2) \otimes V^{-2\psi-2n+3}(\fr{so}_{2n+1})$.
Next, we define
\[\cC^{\psi}_{CB}(n,m) : = \T{Com}(V^{-2\psi-2n+3}(\fr{so}_{2n+1}),\cW^{\psi}_{CB}(n,m))^{\mathbb{Z}_2}.\]
The conformal element $L-L^{\fr{so}_2} - L^{\fr{so}_{2n+1}}$ has central charge
\begin{equation}\label{eq:cCB}
	c_{CB}=\frac{(k-1) (2 k m+k-2 n-1) (2 k m-k-4 m+2 n+1)}{(-k+2 n+1) (k+4 m-2 n-1)}.
\end{equation}
The free field limit of $\cW^{\psi}_{CB}(n,m)$ is the free field algebra $\cW^{\infty}_{CB}(n,m)$
\[\cO_{\T{ev}}(2n^2+n,2)\otimes \left( \bigotimes_{i=1}^{m}\cO_{\T{ev}}(3,4i)\otimes \bigotimes_{i=1}^{m}\cO_{\T{ev}}(1,4i-2) \right) \otimes  \cO_{\T{odd}}(4n+2,2m+1).\]
\begin{lemma}\label{lemma:CB}
	For $n+m\geq 1$, $C^{\psi}_{CB}(n,m)$ is of type 
	$\cW(1,2^3,3,4^3,\dots)$
	as a 1-parameter vertex algebra.
\end{lemma}

\subsection{Case CD}
For $\gg =  \fr{osp}_{2n|4m}$, we have an isomorphism of $\fr{sp}_{2m}\oplus \fr{so}_2\oplus\fr{so}_{2n}$-modules
\[ \fr{osp}_{2n|4m}\cong \fr{sp}_{2m}\oplus\fr{so}_2\oplus\fr{so}_{2n+1} \oplus \rho_{\omega_2}\otimes \mathbb{C}^3\otimes\mathbb{C} \oplus \mathbb{C}^{2m}\otimes \mathbb{C}^2\otimes \mathbb{C}^{2n},\]
and the critically shifted level $\psi = \frac{k+4 m-2 n}{4 m}$.
We define 
\[\cW^{\psi}_{CD}(n,m) : = \cW^{\frac{k+4 m-2 n}{4 m}-h^{\vee}}(\fr{osp}_{2n|4m},f_{\fr{sp}_{2m}}), \qquad h^{\vee} = 2m-n+1,\]
which has affine subalgebra $V^{k}(\fr{so}_2) \otimes V^{-2\psi-2n+4}(\fr{so}_{2n})$.
Next, we define
\[\cC^{\psi}_{CD}(n,m) : = \T{Com}(V^{-2\psi-2n+4}(\fr{so}_{2n}),\cW^{\psi}_{CD}(n,m)).\]
The conformal element $L-L^{\fr{so}_2} - L^{\fr{so}_{2n}}$ has central charge
\begin{equation}\label{eq:cCD}
	c_{CD}=\frac{(k-1) (2 k m+k-2 n) (2 k m-k-4 m+2 n)}{(2 n-k) (k+4 m-2 n)}.
\end{equation}
The free field limit of $\cW^{\psi}_{CD}(n,m)$ is the free field algebra $\cW^{\infty}_{CD}(n,m)$
\[\cO_{\T{ev}}(2n^2-n,2)\otimes \cO_{\T{odd}}(4n,2)\otimes \left( \bigotimes_{i=1}^{m}\cO_{\T{ev}}(3,4i)\otimes \bigotimes_{i=1}^{m}\cO_{\T{ev}}(1,4i-2) \right) \otimes \cO_{\T{odd}}(4n,2m+1).\]
\begin{lemma}\label{lemma:CD}
	For $n+m\geq 1$, $\cC^{\psi}_{CD}(n,m)$ is of type 
	$\cW(1,2^3,3,4^3,\dots)$
	as a 1-parameter vertex algebra.
\end{lemma}

\subsection{Case CO}
For $\gg =  \fr{osp}_{1|4m+2n}$, we have an isomorphism of $\fr{sp}_{2m}\oplus \fr{so}_2\oplus\fr{osp}_{1|2n}$-modules
\[ \fr{osp}_{1|4m+2n} \cong \fr{sp}_{2m}\oplus\fr{so}_2\oplus\fr{sp}_{2n} \oplus \rho_{\omega_2}\otimes \mathbb{C}^3\otimes\mathbb{C} \oplus \mathbb{C}^{2m}\otimes \mathbb{C}^2\otimes \mathbb{C}^{2n|1},\]
and the critically shifted level $\psi = \frac{k+4 m+2 n-1}{4 m}$.
We define 
\[\cW^{\psi}_{CO}(n,m) : = \cW^{\frac{k+4 m+2 n-1}{4 m} -h^{\vee}}(\fr{osp}_{1|4m+2n},f_{\fr{sp}_{2m}}), \qquad h^{\vee} = 2 m+n+\frac{1}{2},\]
which has affine subalgebra $V^{k}(\fr{so}_2) \otimes V^{\psi-n-\frac{3}{2}}(\fr{osp}_{1|2n})$.
Next, we define
\[\cC^{\psi}_{CO}(n,m) : = \T{Com}(V^{\psi-n-\frac{3}{2}}(\fr{osp}_{1|2n}),\cW^{\psi}_{CO}(n,m))^{\mathbb{Z}_2}.\]
The conformal element $L-L^{\fr{so}_2} - L^{\fr{osp}_{1|2n}}$ has central charge
\begin{equation}\label{eq:cCO}
	c_{CO}=-\frac{(k-1) (2 k m-k-4 m-2 n+1) (2 k m+k+2 n-1)}{(k+2 n-1) (k+4 m+2 n-1)}.
\end{equation}
The free field limit of $\cW^{\psi}_{CO}(n,m)$ is the free field algebra $\cW^{\infty}_{CO}(n,m)$
\[\cO_{\T{ev}}(2n^2+n,2)\otimes  \cO_{\T{odd}}(4n,2)\otimes \left( \bigotimes_{i=1}^{m}\cO_{\T{ev}}(3,4i)\otimes \bigotimes_{i=1}^{m}\cO_{\T{ev}}(1,4i-2) \right) \otimes \cS_{\T{ev}}(2n,2m+1)\otimes \cO_{\T{odd}}(1,2m+1).\]
\begin{lemma}\label{lemma:CO}
	For $n+m\geq 1$, $\cC^{\psi}_{CO}(n,m)$ is of type 
	$\cW(1,2^3,3,4^3,\dots)$
	as a 1-parameter vertex algebra.
\end{lemma}

	A consequence of Theorem Lemmas \ref{lemma:CD}-\ref{lemma:BO} is the following.
\begin{proposition}\label{extension features}
	Let $X= B$ or $C$ and let $Y = B$, $C$, $D$, or $O$.
	\begin{enumerate}
		\item $\cC^{\psi}_{XY}(n,m)$ is simple for generic values of $\psi$.
		\item For $\cW^{\psi}_{CY}(n,m)$, without loss of generality we may replace the strong generating field in each odd weight $1,3,\dots, 2M-1$, and the three strong generating fields in each even weight $2,4,\dots, 2M$, with elements of the same weight in the coset $\cC^{\psi}_{CY}(n,m)$. Similarly, for $\cW^{\psi}_{BY}(n,m)$, we may replace the strong generating field in each odd weight $1,3,\dots, 2M-1$, and the three strong generating fields in each even weight $2,4,\dots, 2M+2$, with elements of the same weight in the coset $\cC^{\psi}_{BY}(n,m)$. 
		\item Let $U\cong \mathbb{C}^2\otimes \rho_{\fr{a}}$, where $\rho_{\fr{a}}$ is the standard representation of $\fr{a}$. It is spanned by $P^{\frac{1}{2},j},P^{-\frac{1}{2},j}$, where $j$ runs over a basis of $\rho_{\fr{a}}$.
		Then $U$ has a supersymmetric bilinear form 
		\[\langle \ ,\ \rangle:  U \to \mathbb{C}, \quad \begin{cases}
			\langle a,b\rangle = a_{(2m)}b,&\quad X=C,\\
			\langle a,b\rangle = a_{(2m+1)}b & \quad X=B.
		\end{cases}  \]
		This form is nondegenerate and coincides with the standard pairing on $\mathbb{C}^2\otimes \rho_{\fr{a}}$.
		Hence, without loss of generality, we may normalize the fields in $U$ as
		\begin{equation}\label{normalization}
			P^{\mu,i}(z)P^{\nu,j}(w)\sim 
			\begin{cases}
				\delta_{i,j}\delta_{\mu+\nu,0}\B{1}(z-w)^{-2m-1}+\dotsb,&\quad X=C,\\
				\delta_{i,j}\delta_{\mu+\nu,0}\B{1}(z-w)^{-2m-2}+\dotsb,&\quad X=B.\\
			\end{cases}
		\end{equation}
		Here, the remaining terms lie $V^{a}(\fr{a}) \otimes \cC_{XY}(n,m)$.
	\end{enumerate}
\end{proposition}

\section{GKO cosets}\label{sect:Diagonal}
In this section, we describe the specialization of GKO cosets \eqref{BDGKO} to the case of $\cW^{BD, \emptyset, \{2\}}_{\infty} := \cW^{\gs\go_2}_{\infty}$. We will use the notation $\cC^{\ell}(n)$ for $n \in \mathbb{Z}$, which have the property that the Heisenberg field (regarded as a generator for affine $\gs\go_2$) has constant level $n$.


\subsection{Cases B and D} Consider the rank $2n$ free fermion algebra $\cF(2n)$ with basis $\{\phi^{1,i}, \phi^{2,i}|\ i = 1,\dots,n\}$ and OPE relations
$$\phi^{1,i}(z) \phi^{1,j} \sim \delta_{i,j} (z-w)^{-1},\qquad \phi^{2,i}(z) \phi^{2,j} \sim \delta_{i,j} (z-w)^{-1}.$$ We have an embedding $L_2(\gs\go_n) \hookrightarrow \cF(2n)$ given by $\omega_{i,j} \mapsto :\phi^{1,i} \phi^{1,j}: + :\phi^{2,i} \phi^{2,j}:$, for $i\neq j$. The commutant of $L_2(\gs\go_n)$ inside $\cF(2n)$ contains the Heisenberg field $H = \sqrt{-1}\sum_{i=1}^n :\phi^{1,i} \phi^{2,i}:$, which satisfies $H(z) H(w) \sim n(z-w)^{-2}$. The weight $\frac{1}{2}$ space of $\cF(2n)$ transforms as $\mathbb{C}^2 \otimes \mathbb{C}^{n}$ under $\gs\go_2 \oplus \gs\go_{n}$. We have the diagonal embedding $V^{\ell}(\gs\go_n)\hookrightarrow V^{\ell-2}(\gs\go_n) \otimes \cF(2n)$, and we define
\begin{equation}\label{diagonalBD}
	\cC^{\ell}\left(n \right) = \T{Com}(V^{\ell}(\fr{so}_n), V^{\ell-2}(\fr{so}_n)\otimes \cF(2n))^{\mathbb{Z}_2}.
\end{equation}

The central charge of the Heisenberg coset of $\cC^{\ell}(n)$ is
\begin{equation} \label{eq:cn} c_n= \frac{(-2 + \ell) (-1 +  n) (-4 + \ell + 2 n)}{(-4 + \ell +  n) (-2 + \ell +  n)}.\end{equation}

The free field limit of $\cC^{\ell}(n)$ is the invariant algebra $\cF(2n)^{\text{O}_{n}}$. By Weyl's first fundamental theorem for standard representation of $\text{O}_{n}$ \cite{W}, this has a strong generating set consisting of
\begin{equation} 
\begin{split} &  \sum_{j=1}^n \big(:\partial^p \phi^{1,j} \partial^q \phi^{1,j}: \big), \qquad p > q \geq 0,
\\ &  \sum_{j=1}^n \big(:\partial^p \phi^{2,j} \partial^q \phi^{2,j}: \big), \qquad p > q \geq 0,
\\ &  \sum_{j=1}^n \big(:\partial^p \phi^{1,j} \partial^q \phi^{2,j}: \big), \qquad p \geq q \geq 0.
\end{split}\end{equation}
Removing redundancy due to differential relations among the above generators, we see that $\cF(2n)^{\text{SO}_{n}}$ has a strong generating set of type $\cW(1,2^3,3,\dots)$. This holds generically for the coset \eqref{diagonalBD}, so we obtain

\begin{lemma}\label{lemma:B and D}
For $n \in \mathbb{Z}_{\geq 1}$, $\cC^{\ell}(n)$ is a simple, $1$-parameter vertex algebra of type $\cW(1,2^3,3,4^3,\dots)$ containing a standardly normalized Heisenberg field $H$ with level $n$.

\end{lemma}

\subsection{Case C}
Consider the rank $2n$ $\beta\gamma$-system $\cS(2n)$ with generators $\{\beta^{i,j}, \gamma^{i,j}|\ i=1,2,\ j =1,\dots,n\}$ and OPE relations
$$\beta^{1,j}(z) \gamma^{1,k}(w) \sim \delta_{j,k}(z-w)^{-1},\qquad \beta^{2,j}(z) \gamma^{2,k}(w) \sim \delta_{j,k}(z-w)^{-1}.$$
We have a homomorphism $V^{-1}(\fr{sp}_{2n}) \rightarrow \cS(2n)$, with generators
$$:\beta^{1,i} \beta^{1,j}: + :\beta^{2,i} \beta^{2,j}:,\qquad  :\beta^{1,i} \gamma^{1,j}: + :\beta^{2,i} \gamma^{2,j}: ,\qquad :\gamma^{1,i} \gamma^{1,j}:  +: \gamma^{2,i} \gamma^{2,j}:.$$
The commutant of $V^{-1}(\gs\gp_{2n})$ inside $\cS(2n)$ contains the Heisenberg field $$H = \sqrt{-1}\big(\sum_{i=1}^n :\beta^{1,i} \gamma^{1,i}: + :\beta^{2,i} \gamma^{2,i}:\big)$$ satisfying $H(z) H(w) \sim -2n(z-w)^{-2}$. The weight $\frac{1}{2}$ space of $\cS(2n)$ transforms as $\mathbb{C}^2 \otimes \mathbb{C}^{2n}$ under $\gs\go_2 \oplus \gs\gp_{2n}$. We have a diagonal embedding $V^{\ell}(\gs\gp_{2n}) \hookrightarrow V^{\ell+1}(\gs\gp_{2n}) \otimes \cS(2n)$, and we define
\begin{equation}\label{diagonalC}
	\cC^{\ell}(-2n) = \text{Com}(V^{\ell}(\fr{sp}_{2n}), V^{\ell+1}(\fr{sp}_{2n}) \otimes \cS(2n)).
\end{equation}
The central charge of the Heisenberg coset of $\cC^{\ell}(-2n)$ is 
\begin{equation} \label{eq:c-2n} c_{-2n} = -\frac{(1 + \ell) (1 + 2 n) (2 + \ell + 2 n)}{(1 + \ell + n) (2 + \ell + n)}.\end{equation}
The free field limit of $\cC^{\ell}(-2n)$ is the invariant algebra $\cS(2n)^{\text{Sp}_{2n}}$. By Weyl's first fundamental theorem for standard representation of $\text{Sp}_{2n}$ \cite{W}, this has a strong generating set consisting of
\begin{equation} 
\begin{split} & \sum_{j=1}^n \big(:\partial^p \beta^{1,j}  \partial^q \gamma^{1,j}: - :\partial^q\beta^{1,j} \partial^p \gamma^{1,j}: \big), \qquad p > q \geq 0,
\\ & \sum_{j=1}^n \big(:\partial^p \beta^{2,j}  \partial^q \gamma^{2,j}: - :\partial^q\beta^{2,j} \partial^p \gamma^{2,j}:\big), \qquad p > q \geq 0,
\\ & \sum_{j=1}^n \big(:\partial^p \beta^{1,j}  \partial^q \gamma^{2,j}: - :\partial^q\beta^{2,j} \partial^p \gamma^{1,j} :\big), \qquad p \geq q \geq 0.
\end{split}\end{equation}
Removing redundancy due to differential relations among the above generators, we see that $\cS(2n)^{\text{Sp}_{2n}}$ has a strong generating set of type $\cW(1,2^3,3,\dots)$. This holds generically for the coset \eqref{diagonalC}, so we obtain
\begin{lemma}\label{lemma:C} For $n \in \mathbb{Z}_{\geq 1}$, $\cC^{\ell}(-2n)$ is a simple, $1$-parameter vertex algebra of type $\cW(1,2^3,3,4^3,\dots)$ containing a standardly normalized Heisenberg field $H$ with level $-2n$.
\end{lemma}

\subsection{Case O}
Using the above notation, we have a homomorphism $V^{-1}(\go\gs\gp_{1|2n}) \rightarrow \cS(2n) \otimes \cF(2)$, with generators
\begin{equation} \begin{split} 
& :\beta^{1,i} \beta^{1,j}: + :\beta^{2,i} \beta^{2,j}:,\qquad  :\beta^{1,i} \gamma^{1,j}: + :\beta^{2,i} \gamma^{2,j}: ,\qquad :\gamma^{1,i} \gamma^{1,j}:  +: \gamma^{2,i} \gamma^{2,j}:,
\\ & :\phi^1 \phi^2:,\qquad :\beta^{1,j} \phi^1: + :\beta^{2,j} \phi^2:,\qquad  :\gamma^{1,j} \phi^1 : + :\gamma^{2,j} \phi^2:.\end{split} \end{equation}
The commutant of $V^{-1}(\go\gs\gp_{1|2n})$ inside $\cS(2n) \otimes \cF(2)$ contains the Heisenberg field $$H = \sqrt{-1} \big(\sum_{j=1}^n \big(: \beta^{1,j} \gamma^{2,j}: - \beta^{2,j} \gamma^{1,j} \big)- : \phi^1  \phi^2:\big)$$ satisfying $H(z) H(w) \sim (-2n+1)(z-w)^{-2}$. The weight $\frac{1}{2}$ space of $\cS(2n) \otimes \cF(2)$ transforms as $\mathbb{C}^2 \otimes \mathbb{C}^{2n|1}$ under $\gs\go_2 \oplus \go\gs\gp_{1|2n}$.
We have the diagonal embedding  $V^{\ell}(\go\gs\gp_{1|2n}) \hookrightarrow V^{\ell+1}(\go\gs\gp_{1|2n}) \otimes \cS(2n)\otimes \cF(2)$, and we define
\begin{equation}\label{diagonalO}
	\cC^{\ell}(-2n+1) = \text{Com}(V^{\ell}(\fr{osp}_{1|2n}), V^{\ell+1}(\go\gs\gp_{1|2n}) \otimes \cS(2n)\otimes \cF(2))^{\mathbb{Z}_2}.
\end{equation}
The central charge of the Heisenberg coset of $\cC^{\ell}(-2n+1)$ is
\begin{equation} \label{eq:c-2n+1}c_{-2n+1}= -\frac{8 (1 + \ell) n (1 + \ell + 2 n)}{(1 + 2 \ell + 2 n) (3 + 2 \ell + 2 n)}.\end{equation}
The free field limit of $\cC^{\ell}(-2n+1)$ is the invariant algebra $(\cS(2n) \otimes \cF(2))^{\text{Osp}_{1|2n}}$. By Sergeev's first fundamental theorem for the standard representation of $\text{Osp}_{1|2n}$ \cite{SI,SII}, this has a strong generating set consisting of
\begin{equation} 
\begin{split} & \sum_{j=1}^n \big(:\partial^p \beta^{1,j}  \partial^q \gamma^{1,j}: - :\partial^q\beta^{1,j} \partial^p \gamma^{1,j}: \big) - :\partial^p \phi^1 \partial^q \phi^1:, \qquad p > q \geq 0,
\\ & \sum_{j=1}^n \big(:\partial^p \beta^{2,j}  \partial^q \gamma^{2,j}: - :\partial^q\beta^{2,j} \partial^p \gamma^{2,j}: \big) - :\partial^p \phi^2 \partial^q \phi^2:, \qquad p > q \geq 0,
\\ & \sum_{j=1}^n \big(:\partial^p \beta^{1,j}  \partial^q \gamma^{2,j}: - :\partial^q\beta^{2,j} \partial^p \gamma^{1,j}: \big)- :\partial^p \phi^1 \partial^q \phi^2:, \qquad p \geq q \geq 0.
\end{split}\end{equation}
Removing redundancy due to differential relations among the above generators, we see that $(\cS(2n) \otimes \cF(2))^{\text{Osp}_{1|2n}}$ has a strong generating set of type $\cW(1,2^3,3,\dots)$. This holds generically for the coset \eqref{diagonalO}, so we obtain

\begin{lemma}\label{lemma:O}
For $n \in \mathbb{Z}_{\geq 1}$, $\cC^{\ell}(-2n+1)$ is a simple, $1$-parameter vertex algebra of type $\cW(1,2^3,3,4^3,\dots)$ containing a standardly normalized Heisenberg field $H$ with level $-2n+1$.
\end{lemma}

	\section{Universal $2$-parameter vertex algebra $\Wso$} \label{sect:main}

	In this section we will construct the universal $2$-parameter vertex algebra $\Wso$ of type 
	\begin{equation} \label{so2starting} 
		\cW(1, 2^3, 3, 4^3, 5, 6^3,\dots).
	\end{equation}
	The weight one field $H$ generates a copy of the Heisenberg algebra $V^k(\fr{so}_2)$, 
    \begin{equation}\label{eq:heis}
        H(z)H(w) \sim k(z-w)^{-2},
    \end{equation}
	and the weight two field $L$ generates the universal Virasoro algebra of central charge $c+1$,
	\begin{equation}\label{eq:vir}
		\begin{split}
			L(z) L(w) \sim& \frac{c+1}{2}(z-w)^{-4}+2  L(w)(z-w)^{-2}(w)+\partial  L(w)(z-w)^{-1}.
		\end{split}
	\end{equation}
	The remaining strong generators are denoted by $W^{n,\mu}$, with $n$ denoting the conformal weight and $\mu$ the Heisenberg charge.
    These are assumed to be primary for Heisenberg subVOA, transforming as the trivial and standard $\fr{so}_2$-modules,
	\begin{equation}\label{eq:affine primary}
	H(z) W^{n,\mu}(w) \sim \mu W^{n,\mu} (w)(z-w)^{-1},
	\end{equation}
	with $W^{3,0}$, $W^{4,0},W^{4,\pm 2}$ being also primary for Virasoro subVOA
	\begin{equation}\label{eq:lowprimary}
		\begin{split}
			L(z)W^{n,\mu}(w) \sim & n W^{n,\mu}(w)(z-w)^{-2}(w)+\partial  W^{n,\mu}(w)(z-w)^{-1}, \quad (n,\mu)=(3,0),(4,0),(4,\pm 2).
		\end{split}
	\end{equation}
	Note that in the above notation, generator $W^{2,0}$ is the Heisenberg-coset Virasoro field $L-\frac{2}{k}:\!HH\!:$, which has central charge $c$.
	Instead of assuming that all fields $W^{n,\mu}$ are Virasoro primary, we postulate that the algebra $\Wso$ is weakly generated by $V^k(\fr{so}_2)$, $\vir$, and $W^{4,0}$; specifically, $W^4$ satisfies the following weak generation property
	\begin{equation}\label{eq:raise}
		\begin{split}
			W^4_{(1)}W^{n,\mu}=W^{n+2,\mu},\quad (n,\mu)=(2,\pm 2),\T{ and } n\geq 3.
		\end{split}
	\end{equation}
	The remaining OPEs assume the most general form, compatible with conformal and Heisenberg gradations; that is, we make the following ansatz,
	\begin{equation}\label{eq:OPEgeneralG}
	W^{n,\mu}(z)W^{m,\nu}(w)\sim \sum_{r=0}^{n+m-1}\sum_{\Omega\in \text{PBW}_{r}} w^{W^{n,\mu},W^{m,\nu}}_{\Omega}\Omega(w) (z-w)^{-n-m+r},
	\end{equation}
	where $\Omega$ is a PBW normally ordered monomial in the generators and their derivatives, with conformal weight $n+m-r-1$ and Heisenberg charge $\mu+\nu$.

	Existence of $\vir\otimes \aff$ has useful consequences.
	Recall that $V^k(\fr{so}_2)$ has an automorphism $\sigma$ which maps $H\mapsto -H$; by postulating that $W^{3,0}\mapsto -W^{3,0}$ and $W^{2,\pm2}\mapsto W^{2,\mp2}$, thanks to the weak generation hypothesis (\ref{eq:raise}) it extends to $\Wso$ as
	\begin{equation}\label{eq:auto}
		\sigma:W^{n,\mu}\mapsto (-1)^nW^{n,-\mu}.
	\end{equation}
	In particular, the structure constants appearing in (\ref{eq:OPEgeneralG}) satisfy
	\begin{equation}\label{eq:autoSC}
		w^{W^{n,\mu},W^{m,\nu}}_{\Omega}=(-1)^{n+m}w^{W^{n,-\mu},W^{m,-\nu}}_{\sigma(\Omega)}.
	\end{equation}
	Further, we may without loss of generality refine the ansatz (\ref{eq:OPEgeneralG}).
	Let us denote by the round braces $W^{(\Gamma,\gamma)}$ a span of $V^k(\fr{so}_2) \otimes \vir$-primary fields $W^{(\Gamma,\gamma)}$ of conformal weight $\Gamma$ and Heisenberg charge $\gamma$, i.e. certain linear combinations of PBW monomials satisfying OPE relations
	\[L(z)W^{(\Gamma,\gamma)} \sim \Gamma W^{(\Gamma,\gamma)}(w)(z-w)^{-2}+W^{(\Gamma,\gamma)}(w)(z-w)^{-1}, \quad H(z)W^{(\Gamma,\gamma)}(w)\sim \gamma W^{(\Gamma,\gamma)}(w)(z-w)^{-1}.\]
	So we have a decomposition of $\Wso$ as a module
	\begin{equation*}
	\Wso \cong \bigoplus_{\Gamma,\gamma}W^{(\Gamma,\gamma)}= \B{1}\oplus W^{(2,\pm 2)}\oplus W^{(3,0)}\oplus 2W^{(4,0)}\oplus W^{(4,\pm 2)}\oplus W^{(4,\pm4)}\oplus 2W^{(5,0)}\oplus W^{(5,2)}\oplus\dotsb.
	\end{equation*}
	Rather heuristically, we will use round braces around a PBW monomial to denote symbol for a primary vector, which is a correction to it, for example we write
	\[(W^{4,0}) =W^{4,0}, \quad (:\!W^{2,2}W^{2,-2}\!:)=:\!W^{2,2}W^{2,-2}\!:+\dotsb,\]
	where the omitted terms are some PBW monomials involving $H,L,W^{2,\pm2}$, and $W^{3,0}$.
	Heisenberg and conformal symmetry constraints on the OPEs $W^{n,\mu}(z)W^{m,\nu}(w)$, i.e. the imposition of Jacobi identities $J(H, W^{n,\mu}, W^{m,\nu})$ and $J(L,W^{n,\mu},W^{m,\nu})$, give rise to linear constraints among the yet undetermined structure constants in (\ref{eq:OPEgeneralG}), over the field of rational functions $\mathbb{C}(c,k)$
\begin{equation}\label{conformal ansatz}
	\begin{split}
w^{W^{n,\mu},W^{m,\nu}}_{\Lambda(W^{\Gamma,\gamma})}=\beta_{\Gamma,\gamma}^{n,{\mu};m,{\nu}}(\Lambda)w^{W^{n,\mu},W^{m,\nu}}_{W^{\Gamma,\gamma}}.
	\end{split}
\end{equation}
Here, the fields $\Lambda(W^{(\Gamma,\gamma)})$ are $V^k(\fr{so}_2) \otimes \vir$ descendants of the primary field $W^{(\Gamma,\gamma)}$, and  $\beta_{\Gamma,\gamma}^{n,{\mu};m,{\nu}}(\Lambda)$ are rational functions of $c$, $k$, and polynomial in $w^{a,{\alpha};b,{\beta}}_{X,\xi}$ for $a+b \leq n+m-1$, and $X$ a $V^k(\fr{so}_2) \otimes \vir$-primary field; by definition, we have $\beta_{\Gamma,\gamma}^{n,{\mu};m,{\nu}}(\B{1})=1$.

Importantly for us, if any additional parameters in the OPE algebra of $\Wso$ exist, they must arise as structure constants $w^{n,{\mu};m,{\nu}}_{\Gamma,\gamma}$.
Therefore, we extract the coefficients of primary fields arising in OPEs, i.e. terms in equation (\ref{conformal ansatz}) with $\Lambda=\B{1}$, and use the shorthand
\begin{equation}\label{OPEs refined}
	W^{n,\mu} \times  W^{m,\nu} = \sum_{\Gamma,\gamma}w^{W^{n,\mu},W^{m,\nu}}_{W^{\Gamma,\gamma}}W^{(\Gamma,\gamma)}.
\end{equation}

Lastly, we organize the OPEs $W^{n,\mu}(z)W^{m,\nu}(w)$ into three types of interactions.
	\begin{enumerate}
		\item Trivial and trivial. 
		We denote by $D^{n}_{r}(1,1)$ the following set of products
		\begin{equation*}
			\begin{split}
				D^{n}_{r}(1,1) = &\{H_{(r)}W^{n-1,0},W^{2,0}_{(r)}W^{n-2,0},W^{3,0}_{(r)}W^{n-3,0},\dotsb\}.
			\end{split}
		\end{equation*}
		\item Trivial and standard.
		We denote by $D^{n}_{r}(1,S)$ the following set of products
		\begin{equation*}
			\begin{split}
				D^{n}_{r}(1,S) = &\{W^{2,\pm2}_{(r)}W^{n-2,0},W^{4,\pm2}_{(r)}W^{n-4,0},W^{6,\pm2}_{(r)}W^{n-6,0},\dotsb\}.
			\end{split}
		\end{equation*}
		\item Standard and standard. 
		We denote by $D^{2n+2}_{r}(S,S)$ the following set of products
		\[D^{2n}_{r}(S,S) =\{W^{2,\pm2}_{(r)}W^{2n-2,\pm2}, W^{2,\pm2}_{(r)}W^{2n-2,\mp2}, W^{4,\pm2}_{(r)}W^{2n-2,\pm2},W^{4,\pm2}_{(r)}W^{2n-2,\mp2},\dotsb\} \]
	\end{enumerate}
	Let $D^n_r$ denote the set of all $r^{ \T{th} }$ products among generators of total weight $n$; specifically
    \begin{equation}\label{eq:products}
    \begin{split}
        D^{2m}_r = D^{2m}_r(1,1)\cup D^{2m}_r(1,S)\cup D^{2m}_r(S,S),\quad D^{2m+1}_r =D^{2m+1}_r(1,S)\cup D^{2m+1}_r(1,1).
    \end{split}
    \end{equation}
Write $D^n = \bigcup_{r\leq n} D^{n}_r$ and $D_n = \bigcup_{m\leq n}^n D^m$ for the OPE data among fields of total weight $n$ and not exceeding $n$, respectively. Lastly, let $J^m$ denote the set of all Jacobi identities $J_{r,s}(a,b,c)$ among generating fields $a,b,c$ of total weight exactly $m$, and $J_n = \bigcup_{m\leq n}^n J^m$.

	\subsection{Step 1: base computation}
The first OPEs that are not a direct consequence of our assumptions are 
$W^{2,2}(z)W^{2,-2}(w)$, $W^{2,2}(z)W^{3,0}(w)$, and $W^{2,2}(z)W^{4,2}(w)$, with $W^{2,-2}(z)W^{3,0}(w)$ and $W^{2,-2}(z)W^{4,-2}(w)$ determined by automorphism (\ref{eq:auto}).
The Jacobi identities $J(W^{2,2}, W^{2,2},W^{2,-2})$, $J(W^{2,2}, W^{2,2},W^{3,0})$ and $J(W^{2,2}, W^{2,-2},W^{3,0})$ express all the structure constants arising in these OPEs as rational functions in the central charge $c$ and $\fr{so}_2$-level $k$
\begin{equation*}
\begin{split}
    W^{2,2} \times W^{2,-2} =& \omega^{W^{2,2}, W^{2,-2}}_{\B{1}}\B{1}+\omega^{W^{2,2}, W^{2,-2}}_{W^{3,0}}W^{3,0}, \quad  \omega^{W^{2,2}, W^{2,-2}}_{\B{1}}=\frac{3 c k^3 \omega _2 \omega _3}{8 (k-1) \left(c k+4 c+3 k^2-15 k+12\right)},\\
    W^{2,2} \times W^{3,0} =&\omega^{W^{2,2},W^{3,0}}_{W^{2,2}} W^{2,2} +  \omega^{W^{2,2},W^{3,0}}_{W^{4,2}} W^{4,2},\\
W^{2,2}\times W^{4,2} =& \omega^{W^{2,2}, W^{4,2}}_{:\!W^{2,2}W^{2,2}\!:} :\!W^{2,2}W^{2,2}\!:,\quad \omega^{W^{2,2}, W^{4,2}}_{:\!W^{2,2}W^{2,2}\!:}=-\frac{6 k  \left(c k+4 c+8 k^2-12 k+4\right)\omega _2}{(k-1)\left(5 c k^2-4c k+44 k^2-108 k+64\right) \omega _4 },
\end{split}
\end{equation*}
where we chose $\omega^{W^{2,2},W^{3,0}}_{W^{2,2}}=\omega_2,$ $\omega^{W^{2,2}, W^{2,-2}}_{W^{3,0}}=\omega_3$, $\omega^{W^{2,2},W^{3,0}}_{W^{4,2}} =\omega_4$ to be scaling parameters, since conditions (\ref{eq:affine primary}) and (\ref{eq:lowprimary}) leave them undetermined.
	
Next, imposing Jacobi identities $J(W^{2,2},W^{2,2},W^{4,0})$, $J(W^{2,0},W^{2,-2},W^{4,0})$, $J(W^{2,2},W^{2,-2},W^{4,-2})$, and $J(W^{2,2},W^{2,2},W^{4,-2})$ allows to us determine OPEs $W^{2,2}(z)W^{4,0}(w)$, $W^{3,0}(z)W^{3,0}(w)$, $W^{4,0}(z)W^{3,0}(w)$, and $W^{2,2}(z)W^{4,-2}(w)$  uniquely in terms the scaling parameters $\omega_2,\omega_3,\omega_4$ and $c,k$. 
Specifically, we find
\begin{equation*}\label{data}
\begin{split}
     W^{3,0} \times  W^{3,0}  =&  w^{W^{3,0},W^{3,0}}_{\B{1}} \B{1} + w^{W^{3,0},W^{3,0}}_{W^{4,0}} W^{4,0}+w^{W^{3,0},W^{3,0}}_{(:\!W^{2,2}W^{2,-2}\!:)} (:\!W^{2,2}W^{2,-2}\!:),\\
    W^{4,0} \times  W^{3,0}  =&   w^{W^{4,0},W^{3,0}}_{W^{3,0}}  W^{3,0} +W^{5,0}+w^{W^{4,0},W^{3,0}}_{(:\!W^{2,2}W^{4,-2}\!:)} (:\!W^{2,2}W^{4,-2}\!:) ,\\
    W^{2,0} \times  W^{5,0}  =&   w^{W^{2,0} ,W^{5,0}}_{ W^{3,0} }  W^{3,0} +5 (W^{5,0}),\\
    W^{4,0}\times W^{2,2}   =&  W^{4,2}+ w^{W^{2,2} ,W^{4,0}}_{ :\!W^{2,2}W^{3,0}\!:} (:\!W^{2,2}W^{3,0}\!:),\\
     W^{2,2} \times  W^{4,-2}  =&  w^{W^{2,2},W^{4,-2}}_{W^{3,0}} W^{3,0} + w^{W^{2,2},W^{4,-2}}_{:\!W^{2,2}W^{2,-2}\!:}(:\!W^{2,2}W^{2,-2}\!:)+w^{W^{2,2},W^{4,-2}}_{W^{4,0}}  W^{4,0}\\
     &+w^{W^{2,2},W^{4,-2}}_{(W^{5,0})}  (W^{5,0})+w^{W^{2,2},W^{4,-2}}_{(:\!\partial W^{2,2}W^{2,-2}\!:)}  (:\!\partial W^{2,2}W^{2,-2}\!:).
\end{split}
\end{equation*}
For the primary fields
\begin{equation*}
	\begin{split}
		(W^{5,0})=&W^{5,0}+\dotsb,\quad (:\!W^{2,0}W^{3,0}\!:)=\!W^{2,2}W^{3,0}\!:+\dotsb,\\(:\!W^{2,2}W^{4,-2}\!:)=&:\!W^{2,2}W^{4,-2}\!:-:\!W^{2,-2}W^{4,2}\!:+\dotsb,\\
		(:\!\partial W^{2,2}W^{2,-2}\!:)=&:\!\partial W^{2,2}W^{2,-2}\!:-:\! W^{2,2}\partial W^{2,-2}\!:+\dotsb,\\
	\end{split}
\end{equation*}
we find the following structure constants,
\begin{equation*}
	\begin{split}
w^{W^{3,0},W^{3,0}}_{\B{1}}=&\frac{3 c k^3 \omega _2^2}{8 (k-1) \left(c k+4 c+3 k^2-15 k+12\right)},\\
w^{W^{3,0},W^{3,0}}_{W^{4,0}}=&\frac{12 (k-3) k (k+2)(c k+c-2 k+1) \left(5 c k^2-4 c k+44 k^2-108k+64\right) \omega _2 \omega _4 }{(k-2) \left(c k+4 c+3 k^2-15 k+12\right) D(c,k),}\\
w^{W^{3,0},W^{3,0}}_{:\!W^{2,2}W^{2,-2}\!:}=&\frac{6 k^2 \left(5 c^2 k+20 c^2+16 c k^2-54 c k+148 c-200 k+200\right) \omega _2}{(k-1)D(c,k)\omega _3 },\\
w^{W^{4,0},W^{3,0}}_{W^{3,0}}=&\frac{9 k^2(c k-4 c+4 k-4) \left(c k+4 c+8 k^2-12 k+4\right) \omega _2 }{2 (k-1)  \left(c k+4 c+3 k^2-15 k+12\right) \left(5 c k^2-4 c k+44 k^2-108 k+64\right)\omega_4},\\
w^{W^{4,0},W^{3,0}}_{:\!W^{2,2}W^{4,-2}\!:}=&\frac{2 k \left(9 c^2 k^2+14 c^2 k-40 c^2+12 c k^3-65 c k^2+214 c k-240 c-172 k^2+468k-296\right)}{(k+2)(c k+c-2 k+1) \left(5 c k^2-4 c k+44 k^2-108k+64\right) \omega _3 },\\
w^{W^{2,0},W^{5,0}}_{W^{3,0}}=&\frac{36 k^2(c k-4 c+4 k-4) \left(c k+4 c+8 k^2-12 k+4\right) \omega _2 }{(k-1)\left(c k+4 c+3 k^2-15 k+12\right) \left(5 c k^2-4 c k+44 k^2-108 k+64\right) \omega _4},\\
w^{W^{2,2},W^{4,-2}}_{W^{3,0}}=&-\frac{9 k^2(c k-4 c+4 k-4) \left(c k+4 c+8 k^2-12 k+4\right) \omega _2 \omega _3 }{2(k-1) \left(c k+4 c+3 k^2-15 k+12\right) \left(5 c k^2-4 c k+44 k^2-108
	k+64\right) \omega _4},\\
w^{W^{2,2},W^{4,-2}}_{W^{4,0}}=&\frac{12 (k-3) k (k+2) (c k+c-2 k+1) \left(5 c k^2-4 c k+44 k^2-108k+64\right) \omega _2 \omega _3}{(k-2) \left(c k+4 c+3 k^2-15 k+12\right) D(c,k)},\\
w^{W^{2,2},W^{4,-2}}_{:\!W^{2,2}W^{2,-2}\!:}=&\frac{6 (5 c+22) (k-8) k^2(c k-4 c+4 k-4) \left(c k+4 c+8 k^2-12k+4\right) \omega _2 }{(k-1)\left(5 c k^2-4 c k+44 k^2-108 k+64\right) D(c,k)\omega _4 },\\
w^{W^{2,2},W^{4,-2}}_{W^{5,0}}=&-\frac{4 (k-3) (k+2)(c k+c-2 k+1) \left(5 c k^2-4 c k+44 k^2-108k+64\right) \omega _3 }{k Q(c,k)},\\
w^{W^{2,2},W^{4,-2}}_{:\!\partial W^{2,2}W^{2,-2}\!:}=&\frac{6 (k-8) k^2 (c k-4 c+4 k-4) \left(c k+4 c+8 k^2-12 k+4\right) E(c,k)\omega _2 }{(k-1)  \left(5 c k^2-4 c k+44 k^2-108k+64\right)D(c,k)Q(c,k)\omega _4},\\
w^{W^{2,2} ,W^{4,0}}_{ :\!W^{2,2}W^{3,0}\!:}&=-\frac{(k-8) \left(c k+4 c+8 k^2-12 k+4\right) \left(3 c k^2-2 c k-10 k+12\right)}{(k-3)(k+2) (c k+c-2 k+1) \left(5 c k^2-4 c k+44 k^2-108 k+64\right) \omega _4},
	\end{split}
\end{equation*}
and $D(c,k),Q(c,k)$, $E(c,k)$ are the polynomials
\begin{equation}
	\begin{split}\label{eq:loc}
D(c,k)=&5 c^2 k^3+30 c^2 k^2-160
	c^2+46 c k^3-52 c k^2+480 c k-1024 c-88 k^3-264 k^2+1600 k-1248,\\
	Q(c,k)=&25 c^2 k^4+28 c^2 k^3-196 c^2 k^2-144 c^2 k+640 c^2+134 c k^4+88
	c k^3-352 c k^2-2864 c k+4480 c\\
	&-440 k^4+936 k^3+2768 k^2-8640 k+5376,\\
E(c,k)=&25c^3 k^4+60 c^3 k^3+130 c^3 k^2+20 c^3 k-1200 c^3+110 c^2 k^4+454 c^2 k^3-518 c^2k^2+7308 c^2 k\\
&-15680 c^2-968 c k^4+5764 c k^3-20876 c k^2+59168 c k-61040 c+5808k^3-48576 k^2\\
&+106800 k-64032.
	\end{split}
\end{equation}

Next, we proceed to analyze OPEs among higher weight fields.
Unfortunately, we cannot display the structure constants here, since they are too complicated. 
One can proceed similarly, and we obtain the following proposition.
\begin{proposition}\label{prop:base case}
All OPEs in $D_9$ are expressed in terms of the central charge $c$ and $\fr{so}_2$-level $k$, and the scaling parameters of weights 2,3 and 4 fields $\omega_{2},\omega_3,\omega_4$.
Furthermore, all denominators $D$ in this algebra are powers of linear factors $(k-3),(k-2),(k-1),(k),(k+2)$, together with the set of curves $$\{(c k+c-2 k+1),(c k+4 c+3 k^2-15 k+12),(5 c k^2-4 c k+44 k^2-108 k+64),D(c,k),Q(c,k),P(c,k),R(c,k)\},$$
where in addition to the polynomials (\ref{eq:loc}) we have
\begin{align*}
    P(c,k)=&35 c^2 k^4+68 c^2 k^3-316 c^2 k^2-464 c^2 k+1280 c^2+226 c k^4-200 c k^3\\
    &+816 c k^2-6832 c
k+8576 c-616 k^4+760 k^3+7024 k^2-17536 k+10368,\\
R(c,k)=&25 c^2 k^4+28 c^2 k^3-196 c^2 k^2-144 c^2 k+640 c^2+134 c k^4+88 c k^3-352 c k^2\\
&-2864 c
k+4480 c-440 k^4+936 k^3+2768 k^2-8640 k+5376.
\end{align*}
\end{proposition}

	\subsection{Step 2: induction}	
	We begin by introducing modified strong generators $\tilde W^{n,\mu}$, that are adapted for the purposes of induction. Whenever possible, we denote them by
    \begin{equation}\label{eq:notation}
        X^{2n}=\tilde W^{2n,-2},\quad Y^{2n}=\tilde W^{2n,2},\quad H^{2n}=\tilde W^{2n,0},\quad W^{2n+1}=\tilde W^{2n+1,0},
    \end{equation}
        to ease the syntactical burden.
		The Heisenberg field $H$, Virasoro fields $L$, fields $W^{2,\pm2}$, and $W^{3,0}$ remain fixed, i.e. we have        
        \[W^{2,2}=X^2,\quad W^{2,-2}=Y^2, \quad W^{3,0}=W^3,\]
        while field $W^{4,0}$ is modified
			\begin{equation}\label{W4deformed}
			\begin{split}
				{W}^{4} = W^{4,0} +\frac{1}{k}:\!HW^{3}\!:-\frac{1-k\omega_{1,3}}{\omega_3}(:\!X^2Y^2\!:).
			\end{split}
		\end{equation}
		Up to rescaling, $W^4$ is unique correction to $W^{4,0}$, that it has Heisenberg charge $0$, is primary for $L$, and satisfies $W^{4}_{(1)} H = W^{3}$.
		Further, we set the scaling constants as follows.
		\begin{equation}\label{eq:scale1}
			\begin{split}
				\omega_2=&-\frac{c k+c+12 k^2-15 k+3}{36 k^2},\quad \omega_3=-3,\quad 				\omega_{1,3}=\frac{A(c,k)}{3 k (2 k-1)B(c,k)},\\
				\omega_4=&-\frac{6 k (2 k-1)B(c,k)}{(2 k+1) (4
					k-3) (4 c k+c-8 k+1) \left(5 c k^2-c k+44 k^2-27 k+4\right)},\\
				A(c,k)=&60 c^2 k^4-12 c^2 k^3-19 c^2 k^2+11 c^2 k+168 c k^4+376 c k^3-380 c k^2+69 ck+6 c\\
				&-1056 k^4+1112 k^3-372 k^2+16 k+6,\\
				B(c,k)=&10 c^2 k^3+15 c^2
				k^2-5 c^2+92 c k^3-26 c k^2+60 c k-32 c-176 k^3-132 k^2+200 k-39.
			\end{split}
		\end{equation}
        \begin{remark}
            Scaling conditions (\ref{eq:scale1}) for $\omega_2$ and $\omega_{3}$ are somewhat arbitrary, while those for $\omega_4$ and $\omega_{1,3}$ are necessary. 
            These are chosen so that certain structure constants we define soon, and that appear in Proposition (\ref{prop:structure constants}) are either nonzero rational constants, or zero.
        \end{remark}
    Lastly, we postulate that the raising property is respected by fields of higher conformal weights
		\begin{equation}\label{raise}
		{W}^{4}_{(1)}\tilde{W}^{n,\mu}= \tilde{W}^{n+2,\mu}, \quad n\geq 1.
		\end{equation}
		We continue using notation $D^N_r(1,1),D^N_r(1,S),D^N_r(S,S),D^N_r,D_N$ as defined in (\ref{eq:products}), since both sets contain the same data, via change of strong generators.
Now, we name the relevant structure constants for our induction procedure in the following general ansatz
	\begin{equation}\label{eq:gen anchors}
		\begin{split}
			\tilde W^{n,\mu}_{(0)}\tilde W^{m,\nu}=&w^{\tilde W^{n,\mu},\tilde W^{m,\nu}}_{ \tilde W^{n+m-1,\mu+\nu}}\tilde W^{n+m-1,\mu+\nu}+w^{\tilde W^{n,\mu},\tilde W^{m,\nu}}_{ \partial \tilde W^{n+m-2,\mu+\nu}}\partial \tilde W^{n+m-2,\mu+\nu}\\
			+&\sum w^{\tilde W^{n,\mu},\tilde W^{m,\nu}}_{:\!\tilde W^{i,\alpha}\tilde W^{j,\beta}\!:}:\!\tilde W^{i,\alpha}\tilde W^{j,\beta}\!:+W^{n,\mu;n,\nu}_0,\\
	\tilde W^{n,\mu}_{(1)}\tilde W^{m,\nu}=&w^{\tilde W^{n,\mu},\tilde W^{m,\nu}}_{ \tilde W^{n+m-2,\mu+\nu}}\tilde W^{n+m-2,\mu+\nu}+W^{n,\mu;n,\nu}_1.
		\end{split}
	\end{equation}
	We proceed to perform a combinatorial analysis on each of these OPEs. 
	\begin{itemize}
		\item Trivial with trivial. The first order pole of $W^{i}$ and $W^{j}$ transforms as the trivial $\fr{so}_2$-module.
            When parity of conformal weights are equal we have
			\begin{equation}\label{ansatzWW}
			\begin{split}
				W^{i}_{(0)}W^{j}=&w_{\partial W^{i+j-2}}^{W^i,W^j}\partial W^{i+j-2}+W_{0}^{i,j},\\
				W^{i}_{(1)}W^{j}=&w_{W^{i+j-2}}^{W^i,W^j}W^{i+j-2}+W_{1}^{i,j}.
			\end{split}
		\end{equation}
		When the parities of conformal weights are distinct we have
				\begin{equation}\label{ansatzWWeo}
		\begin{split}
			W^{2i}_{(0)}W^{2j-1}=&w_{\partial W^{2i+2j-3}}^{W^{2i},W^{2j-1}}\partial W^{2i+2j-3}+w_{:\!X^2Y^{2i+2j-4}\!:}^{W^{2i},W^{2j-1}}:\!X^2Y^{2i+2j-4}\!:+\dotsb\\
			&+w_{:\!Y^2X^{2i+2j-4}\!:}^{W^{2i},W^{2j-1}}:\!Y^2X^{2i+2j-4}\!:+W_{0}^{2i,2j-1},\\
			W^{2i}_{(1)}W^{2j-1}=&w_{ W^{2i+2j-1}}^{W^{2i},W^{2j-1}}W^{2i+2j-3}+W_{1}^{2i,2j-1}.
		\end{split}
	\end{equation}
		Here, in both of the above cases (\ref{ansatzWW}) and (\ref{ansatzWWeo}), expressions $W_{r}^{i,j}$ are some normally ordered polynomials in the generators $\tilde W^{n,\mu}$ for $n\leq i+j-3$, and their derivatives. They are $\fr{so}_2$-invariant of conformal weight $i+j-r-1$, and have the following dependencies.
		\begin{equation}
		\begin{split}
			W^{2i,2j}_{-} =& W^{2i,2j}_{-} (H,L,X^2,Y^2,\dots,W^{2i+2j-4},X^{2i+2j-4},Y^{2i+2j-4},W^{2i+2j-3}),\\
			W^{2i,2j-1}_{-} =& W^{2i,2j-1}_{-} (H,L,X^2,Y^2,\dots,W^{2i+2j-6},X^{2i+2j-6},Y^{2i+2j-6},W^{2i+2j-5},W^{2i+2j-4}).\\
		\end{split}
		\end{equation}
				
		\item  Trivial with standard. The first order pole of $X^{2i}$ and $W^{2j}$ has odd conformal weight $2i+2j-1$ and transforms as the standard $\fr{so}_2$-module.
		
		\begin{equation}\label{ansatzWH}
			\begin{split}
				W^{2i}_{(0)}X^{2j}=&w_{ \partial X^{2j+2i-2}}^{W^{2i},X^{2j}} \partial X^{2j+2i-2}+w_{:\!W^1X^{2i+2j-2}\!:}^{W^{2i},X^{2j}}:\!W^{1}X^{2i+2j-2}\!:+\dotsb\\
				&+w_{:\!X^{2}W^{2i+2j-3}\!:}^{W^{2i},X^{2j}}:\!X^{2}W^{2i+2j-3}\!:+\dotsb +X_{0}^{2i,2j},\\
				W^{2i}_{(1)}X^{2j}=&w_{ X^{2j+2i-2}}^{W^{2i},X^{2j}}  X^{2j+2i-2}+X_{1}^{2i,2j}.
			\end{split}
		\end{equation}
		If the parities of conformal weights are unequal, then the first order pole of $X^{2i}$ and $W^{2j-1}$ has even conformal weight $2i+2j-2$, again transforming as the standard $\fr{so}_2$-module. 
		\begin{equation}\label{ansatzWHeo}
		\begin{split}
			W^{2i-1}_{(0)}X^{2j}=&x_{0}^{2i-1,2j} X^{2j+2i-2}+X_{0}^{2i-1,2j}.
		\end{split}
		\end{equation}
			Here, in both of the above cases (\ref{ansatzWH}) and (\ref{ansatzWHeo}), expressions $X_{-}^{i,j}$ are some normally ordered polynomial in the generators $\tilde W^{n,\mu}$ for $n\leq i+j-4$, and their derivatives. Specifically, we have the following dependencies.
			\begin{equation}
				\begin{split}
					X^{2i,2j}_{-} =& X^{2i,2j}_{-} (H,L,X^2,Y^2,\dots,W^{2i+2j-5},W^{2i+2j-4},X^{2i+2j-4},Y^{2i+2j-4}),\\
					X^{2i-1,2j}_{-} =& X^{2i,2j-1}_{-} (H,L,X^2,Y^2,\dots,W^{2i+2j-5},X^{2i+2j-4},Y^{2i+2j-4}).\\
				\end{split}
			\end{equation}

		\item Standard with standard.  The first order pole of $X^{2i}$ and $Y^{2j}$ has odd conformal weight $2i+2j-1$ and transforms as the trivial $\fr{so}_2$-module. 
		\begin{equation}\label{ansatzXH}
			\begin{split}
				X^{2i}_{(0)}Y^{2j}=& w^{W^{2i},W^{2j}}_{W^{2i+2j-1}} W^{2i+2j-1}+ w^{W^{2i},W^{2j}}_{\partial W^{2i+2j-2}} \partial W^{2i+2j-2}+ M_{0}^{2i,2j},\\
				X^{2i}_{(1)}Y^{2j}=& w^{W^{2i},W^{2j}}_{W^{2i+2j-2}} W^{2i+2j-2}+M_{1}^{2i,2j},\\
			\end{split}
		\end{equation}
		where expressions $M_{-}^{2i,2j}$ are some normally ordered polynomial in the generators $\tilde W^{n,\mu}$ for $n\leq 2i+2j-3$, and their derivatives. Specifically, we have the following dependence.
		\begin{equation}
			\begin{split}
				M^{2i,2j}_{-} =& M^{2i,2j}_{-} (H,L,X^2,Y^2,\dots,W^{2i+2j-4},X^{2i+2j-4},Y^{2i+2j-4},W^{2i+2j-3}).\\
			\end{split}
		\end{equation}
	\end{itemize}

	Our notation above extends the one used for the structure constants (\ref{eq:autoSC}). 
	Specifically, if we denote by $W^{n,\mu;m,\nu}(r)$ the normally ordered differential polynomial, arising as above in $W^{n,\mu}(z)W^{m,\nu}(w)$, then $\T{O}_2$-symmetry affords a relation
	\[W^{n,\mu;m,\nu}_r = (-1)^{n+m}W^{n,-\mu;m,-\nu}_r ,\quad r\leq n+m.\]
	
	Denote the double factorial by
	$$a!!=\begin{cases} 
		1\cdot 3\dotsb (2n-3) (2n-1) & a=2n-1,\\
		2\cdot 4 \dotsb (2n-2)(2n) & a=2n.
	\end{cases}$$

	\begin{proposition} \label{prop:structure constants} Let the notation be fixed as in (\ref{ansatzWW}-\ref{ansatzXH}). We have the following expressions for the first and second order poles among the strong generators.
		\begin{enumerate}
			\item Trivial and trivial fields. 
			\begin{equation*}
				\begin{split}
					W^{i}_{(0)}W^{j}=&\frac{(i-1)i!!j!!}{8(i+j-2)(i+j-4)!!}\partial W^{i+j-2}+W^{i,j}_{0},\\
					W^{i}_{(1)}W^{j}=&\frac{i!!j!!}{8(i+j-4)!!}W^{i+j-2}+W^{i,j}_{1}, \quad i\geq2.
				\end{split}
			\end{equation*}

			\item Trivial and standard fields. 
			\begin{equation*}
				\begin{split}
					W^{2i-1}_{(0)}X^{2j}&=\frac{(2i-1)!!(2j)!!}{(2i+2j-2)!!} X^{2i+2j-2}+X_{0}^{2i-1,2j},\quad i\geq 2,\\
					W^{2i}_{(0)}X^{2j}&=\frac{(2i-1)i!!(2j)!!}{8(2+2j-2)(2i+2j-4)!!}\partial X^{2i+2j-2}+X_{0}^{2i,2j},\\
					W^{2i}_{(1)}X^{2j}&=\frac{(2i)!!(2j)!!}{8(2i+2j-4)!!} X^{2i+2j-2}+X_{1}^{2i,2j}, \quad i\geq2.
				\end{split}
			\end{equation*}
			\item Standard and standard fields.
			\begin{equation*}
				\begin{split}
					X^{2i}_{(0)}Y^{2j}=&\frac{9(2i)!!(2j)!!}{4(2i+2j-1)!!} W^{2j+2i-1}+\frac{2(2i-1)(2i)!!(2j)!!}{9(2i+2j-4)(2i+2j-4)!!}\partial W^{2i+2j-2}+M_{0}^{2i,2j},\\
					X^{2i}_{(1)}Y^{2j}=&\frac{2(2i)!!(2j)!!}{9(2i+2j-4)!!} W^{2i+2j-2}+M_{1}^{2i,2j}.
				\end{split}
			\end{equation*}
		\end{enumerate}
	\end{proposition}
	
	\begin{proof}
		We will proceed by induction on integer index $N$.
		Our base case can be taken to be the Proposition (\ref{prop:base case}). Specifically, the fact that the structure constants highlighted in (\ref{eq:gen anchors}) appearing in $D^7_0\cup D^7_1\cup D^8_0\cup D^8_1$, in the new generating set, have the desired form.
		Inductively, assume that all structure constants defined in (\ref{ansatzWW}-\ref{ansatzXH}) have the form as in the Proposition (\ref{prop:structure constants}) for products in $D^{2N+1}_0\cup D^{2N+1}_1\cup D^{2N+2}_0\cup D^{2N+2}_1$. In particular, it means that the structure constants of quadratic terms all vanish, that is we have the following.
		\begin{equation}
		\begin{split}
		w_{:\!HX^{2i+2j-2}\!:}^{W^{2i},X^{2j}}=&w_{:\!W^3X^{2i+2j-4}\!:}^{W^{2i},X^{2j}}=\dotsb = w_{:\!X^2W^{2i+2j-3}\!:}^{W^{2i},X^{2j}}=0, \\	w_{:\!X^2Y^{2i+2j-4}\!:}^{W^{2i},W^{2j-1}}=&w_{:\!X^4Y^{2i+2j-6}\!:}^{W^{2i},W^{2j-1}}=\dotsb = w_{:\!Y^2X^{2i+2j-4}\!:}^{W^{2i},W^{2j-1}}=0. \\
	\end{split}	
	\end{equation}
		First, we will show that quadratics $w_{:\!HX^{2i+2j}\!:}^{W^{2i+1},X^{2j}},\dotsc,$ and $w_{:\!Y^2X^{2i+2j-2}\!:}^{W^{2i+2},W^{2j-1}},\dots,$ arising in $D^{2N+3}_0(1,S)$, and $D^{2N+3}_0(1,1)$, respectively, all vanish.
		
		Clearly no quadratics arise in the products $H_{(0)}W^{2N+2}$ and $L_{(0)}X^{2N}$. Now consider the Jacobi identities $J_{0,0}(W^{2l},X^{2i},Y^{2j})$ and  $J_{0,0}(W^{2l},W^{2i-1},X^{2j})$, which have the form
		\begin{equation}\label{eqs}
			\begin{split}
			W^{2l}_{(0)}(X^{2i}_{(0)}Y^{2j}) =&X^{2i}_{(0)}(W^{2l}_{(0)}Y^{2j})+ (W^{2l}_{(0)}X^{2i})_{(0)}Y^{2j}, \quad w^{X^{2i},Y^{2j}}_{W^{2i+2j-1}}\neq 0,\\
			W^{2l}_{(0)}(W^{2i-1}_{(0)}X^{2j}) =&W^{2i-1}_{(0)}(W^{2l}_{(0)}X^{2j})+ (W^{2l}_{(0)}W^{2i-1})_{(0)}X^{2j},\quad w^{W^{2i-1},X^{2j}}_{X^{2i+2j-2}}\neq 0.
			\end{split}
		\end{equation}
		By induction only the left sides can contributes a nonzero quadratic in (\ref{ansatzWHeo}) and (\ref{ansatzXH}), and thanks to (\ref{eqs}) no quadratics arise in the products $W^{2l}_{(0)}W^{2i+2j-1}$ and $W^{2l}_{(0)}X^{2i+2j-2}$.
		
		Now, assuming that no quadratics arise in $D^{2N+3}_0(1,1)$ and in $D^{2N+2}_0(1,S)$, we
		extract the coefficient of fields $\tilde W^{2l+n+m-4,\mu+\nu}$ in identities $J_{2,0}(W^{2l},\tilde W^{n,\mu},\tilde W^{m,\nu})$ to obtain a relation
		\begin{equation}\label{20identity}
			 w^{\tilde W^{n,\mu}, \tilde W^{m,\nu}}_{\partial \tilde W^{n+m-2,\mu+\nu}}w^{W^{2l},\tilde W^{n+m-2,\mu+\nu}}_{\tilde W^{2l+n+m-4,\mu+\nu}}+ \left(w^{W^{2l}, \tilde W^{n,\mu}}_{\partial \tilde W^{2l+n-2,\mu}} - w^{W^{2l}, \tilde W^{n,\mu}}_{\tilde W^{2l+n-2,\mu}}\right) w^{\tilde W^{2l+n-2,\mu}, \tilde W^{m,\nu}}_{\tilde W^{2l+n+m-4,\mu+\nu}}=0.
		\end{equation}
		First, set $l=1$ in(\ref{20identity}), and recall that the Virasoro action implies that $w^{L,\tilde W^{n,\mu}}_{\partial W^{n,\mu}}=1$ and $w^{L,\tilde W^{n,\mu}}_{\tilde W^{n,\mu}}=n$.
		Thus we find the following relation.
		\begin{equation}\label{tempEq}
			w^{\tilde W^{n,\mu}, \tilde W^{m,\nu}}_{\partial\tilde  W^{n+m-2,\mu+\nu}}=\frac{n-1}{n+m-2}w^{\tilde W^{n,\mu}, \tilde W^{m,\nu}}_{\tilde W^{n+m-2,\mu+\nu}}.
		\end{equation}
		Next, set $l=2$ in (\ref{20identity}), and recall the raising property $w^{W^4,\tilde W^{n,\mu}}_{\tilde W^{n+2,\mu}}=1$. Using (\ref{tempEq}) we obtain the following recurrence relation.
		\begin{equation}\label{inductLeft}
			w^{\tilde W^{n+2,\mu}, \tilde W^{m,\nu}}_{\tilde W^{n+m,\mu+\nu}}=-\frac{(n-1)w^{W^4,\tilde W^{n+m-2,\mu}}_{\tilde W^{n+m,\mu}}}{(n+m-2)(\frac{3}{n+2}w^{W^4,\tilde W^{n,\mu}}_{\tilde W^{n+2,\mu}} - w^{W^4,\tilde W^{n,\mu}}_{\tilde W^{n+2,\mu}})}w^{\tilde W^{n,\mu}, \tilde W^{m,\nu}}_{\tilde W^{n+m-2,\mu+\nu}} =\frac{n+2}{n+m-2}w^{\tilde W^{n,\mu}, \tilde W^{m,\nu}}_{\tilde W^{n+m-2,\mu+\nu}}.
		\end{equation}
		Thanks to skew-symmetry we exchange indices $n$ and $m$ to obtain
		\begin{equation*}\label{inductRight}
			w^{\tilde W^{n,\mu},\tilde W^{m+2,\nu}}_{\tilde W^{n+m,\mu+\nu}}=\frac{m+2}{n+m-2}w^{\tilde W^{n,\mu},\tilde W^{m,\nu}}_{\tilde W^{n+m-2,\mu+\nu}}.
		\end{equation*}
		This proves the inductive hypothesis for $w^{\tilde W^{n,\mu},\tilde W^{m,\nu}}_{\tilde W^{2i+2j-2,\mu+\nu}}$ and $w^{\tilde W^{n,\mu},\tilde W^{m,\nu}}_{\partial \tilde W^{2i+2j-2,\mu+\nu}}$. 
		
		Lastly, we evaluate the remaining structure constants appearing in the first order poles (\ref{ansatzWHeo}) and (\ref{ansatzXH}).
		Extracting the coefficient of $W^{2i+2j+1}$ arising in identities $J_{1,0}(W^{4},X^{2i},Y^{2j})$ and $J_{0,1}(W^{4},X^{2i},Y^{2j})$ affords the following recursions
		\begin{equation*}\label{xyw}
w^{X^{2i},Y^{2j}}_{W^{2i+2j+1}}=\frac{2 (i+1) }{2 i+2 j+1}w^{X^{2i},Y^{2j}}_{W^{2i+2j-1}},\quad w^{X^{2i},Y^{2j+2}}_{W^{2i+2j+1}}=\frac{2 (j+1) }{2 i+2 j+1}w^{X^{2i},Y^{2j}}_{W^{2i+2j-1}}.
		\end{equation*}
		Similarly, we extract the coefficient of $W^{2i+2j+1}$ arising in identities $J_{1,0}(W^{4},W^{2i-1},X^{2j})$ and $J_{0,1}(W^{4},W^{2i-1},X^{2j})$ yields the following recursions
		\begin{equation*}\label{xwx}
		w^{W^{2i+1},X^{2j}}_{X^{2i+2j}}=\frac{i+1}{i+j}w^{W^{2i-1},X^{2j}}_{X^{2i+2j-2}},\quad w^{W^{2i-1},X^{2j+2}}_{X^{2i+2j}}=\frac{2 j+1 }{2 i+2 j}w^{W^{2i-1},X^{2j}}_{X^{2i+2j-2}}.
	\end{equation*}
    This completes the proof of inductive hypothesis.
	\end{proof}

	\subsection{Step 3: Induction}\label{induction}
	The main result of this subsection is Theorem (\ref{thm:induction}). 
	It is proved by induction, and  the process is similar to that of \cite{Lin}, \cite{KL}, and \cite{CKoL2}.
        First, recall the notation from (\ref{eq:products}), and let us augment it by setting
        \[D_n(1,1)=\bigcup_{r=1}^n D^n_r(1,1), \quad D_n(1,S)=\bigcup_{r=1}^n D^n_r(1,S), \quad D_n(S,S)=\bigcup_{r=1}^n D^n_r(S,S).\]
	We take as our base case is the Proposition (\ref{prop:base case}).
	By inductive data we mean the set of OPEs in $D_{2n}\cup D_{2n+1}$, where
	\[D_{2n} = D_{2n}(S,S)\cup D_{2n}(1,S)\cup D_{2n}(1,1), \quad D_{2n+1} = D_{2n+1}(1,S)\cup D_{2n+1}(1,1),\] and that they are fully expressed in terms of parameters $c$ and $k$. 
	At this stage, the OPEs in $D^{2n+2}\cup D^{2n+3}$ are yet undetermined.
	We will use a subset of Jacobi identities in $J^{2n+4}\cup J^{2n+5}$ to express $D^{2n+2}\cup D^{2n+3}$ in terms of inductive data.
	We write $A\equiv 0$ to denote that $A$ is computable from inductive data $D_{2n}\cup D_{2n+1}$.
	
	In the Lemmas (\ref{recursionsSym}-\ref{lemma:recursionsr}) we use the raising property of $W^4$ to establish the following.
	\begin{proposition}\label{aa}
		\
		\begin{enumerate}
			\item  $W^4(z)W^{2n-2}(w)$ and $W^3(z)W^{2n-1}(w)$ with inductive data together determine $D^{2n+2}(1,1)$.
			\item $W^4(z)W^{2n-1}(w)$ with inductive data together determine $D^{2n+3}(1,1)$.
			\item  $X^2(z)W^{2n-1}(w)$ with inductive data together determine $D^{2n+1}(1,S)$.
			\item  $W^4(z)X^{2n-2}(w)$ with inductive data together determine $D^{2n+2}(1,S)$.
			\item  $X^2(z)Y^{2n-2}(w)$ with inductive data together determine $D^{2n+2}(S,S)$.
		\end{enumerate}
	\end{proposition}
    In broad strokes, this is proved by making use of some Jacobi identities of the form $J(W^4,\tilde W^{i,\mu},\tilde W^{j,\nu})$ and $J(W^3,\tilde W^{i,\mu},\tilde W^{j,\nu})$. First, recall that $H(z)W^{2n+2}(w),L(z)W^{2n+1}(w) \in D^{2n+3}(1,1)$, $H(z)W^{2n+1}(w),L(z)W^{2n}(w) \in D^{2n+2}(1,1)$, $\in D^{2n+3}(1,1)$.
	First, we show how the action of symmetry algebra $V^k(\fr{so}_2)\otimes \vir$ is determined.
	\begin{lemma}\label{recursionsSym}
		Let $n$ be a positive integer. Then we have the following.
		\begin{enumerate}
			\item OPEs $W^4(z)\tilde W^{n,\mu}(w)$ and $L(z)\tilde W^{n,\mu}(w)$ together determine $L(z)\tilde W^{n+2,\mu}(w)$.
			\item OPEs $W^{3}(z)\tilde W^{n,\mu}(w)$ and $H(z)\tilde W^{n,\mu}(w)$ together determine $H(z)\tilde W^{n+2,\mu}(w)$.	
		\end{enumerate}
	\end{lemma}
	
	\begin{proof}
		The Jacobi identity $J_{1,r}(W^4,H,\tilde W^{n,\mu})$ gives rise to the following relation
		\[ H_{(r)}\tilde W^{n+2,\mu}=W^4_{(1)}H_{(r)}\tilde W^{n,\mu}-rW^{3}_{(r)}\tilde W^{n,\mu}.\]
		Note it also reproduces the affine action when restricted to $r=0$.
		Similarly, the Jacobi identity $J_{1,r}(W^4,L,\tilde W^{n,\mu})$ gives rise to a relation
		\[L_{(r)}\tilde W^{n+2,\mu} = W^4_{(1)}L_{(r)}\tilde W^{n,\mu}-(3r-1)W^4_{(r)}\tilde W^{n,\mu},\]
		which also reproduces the Virasoro action when restricted to $r=1$ and $r=0$.
	\end{proof}
    
	Now we explain how the remaining OPEs in $D^{2n+2} \cup D^{2n+3}$ are determined.
	
	We begin with the first order poles.
	In the Proposition (\ref{prop:structure constants}) we have already determined some structure constants arising in the first order poles.
	Therefore, to determine first order poles it is sufficient to analyze the normally ordered differential monomials $W^{i,j}_{0}$, $X^{i,j}_{0}$, and $M^{i,j}_{0}$, defined in (\ref{ansatzWW},\ref{ansatzWWeo}), (\ref{ansatzWH},\ref{ansatzWHeo}), and (\ref{ansatzXH}), respectively.
	Note that by Proposition (\ref{prop:structure constants}), the coefficients of $\partial \tilde W^{i+j,\mu+\nu}$ and $\tilde W^{i+j+1,\mu+\nu}$ arising in Jacobi identities $J_{1,0}(W^4,\tilde W^{i,\mu},\tilde W^{j,\nu})$ are zero.

	\begin{lemma}\label{lemma:recursions0}
		Modulo the inductive data, the following hold.
		\begin{enumerate}
			\item  $W^4_{(0)}X^{2n-2}$  determine $\{W^{2i+4}_{(0)}X^{2n-2i-2} | i\geq 1\}$.
			\item $X^2_{(0)}Y^{2n}$, $W^4_{(0)}W^{2n-2}$ and $W^4_{(0)}X^{2n-2}$ determine $\{X^{2i+2}_{(0)}Y^{2n-2i-2} | i\geq 1\}$.
			\item $X^2_{(0)}X^{2n}$  determine $\{X^{2i+2}_{(0)}X^{2n-2i} | i\geq 1\}$.
			\item $W^3_{(0)}X^{2n-1}$  determine $\{W^{2i+3}_{(0)}X^{2n-2i-2} | i\geq 1\}$.
			\item $W^4_{(0)}W^{2n-1}$  determine $\{W^{2i+4}_{(0)}W^{2n-2i-1}| i\geq 1\}$.
			\item $W^{3}_{(0)}W^{2n-1}$ and $W^4_{(0)}W^{2n-2}$ determine $\{ W^{2i+3}_{(0)}W^{2n-2i-1}| i\geq 1\}$.
		\end{enumerate}
	\end{lemma}
	
	\begin{proof}
		Consider a general Jacobi identity $J_{1,0}(W^4,\tilde W^{i,\mu},\tilde W^{j,\nu})$ which reads
		\begin{equation}\label{Recurion0Full}
			\left(1-w^{W^4,\tilde W^{i,\mu}}_{\partial \tilde W^{i+2,\mu}}\right)W^{i+2,\mu;j,\nu}_0 =
	w^{\tilde W^{i,\mu},\tilde W^{j,\nu}}_{\partial \tilde W^{i+j-2,\mu+\nu}} W^{4,0;i+j-2,\mu+\nu}_0-W^{i,\mu;j+2,\nu}_0+R^{i,\mu;j,\nu}_0,
		\end{equation}
		where 
		\begin{equation}
		R^{i,\mu;j,\nu}(0)=W^4_{(1)}W^{i,\mu;j,\nu}_0 - {W^{4,0;i,\mu}_0}_{(1)}\tilde W^{j,\nu}.
		\end{equation}
		Note that term $R_{1,0}$ is known from inductive data, so we may write
		\begin{equation}\label{Recurion0Partial}
			W^{i+2,\mu;j,\nu}_0\equiv \frac{w^{\tilde W^{i,\mu},\tilde W^{j,\nu}}_{\tilde W^{i+j-2,\mu+\nu}}}{1-w^{W^4,\tilde W^{i,\mu}}_{\partial \tilde W^{i+2,\mu}}} W^{4,0;i+j-2,\mu+\nu}_0-\frac{1}{1-w^{W^4,\tilde W^{i,\mu}}_{\partial \tilde W^{i+2,\mu}}}W^{i,\mu;j+2,\nu}_0.
		\end{equation}
		Finally, we iterate recursion (\ref{Recurion0Partial}) to obtain the desired relations.
	\end{proof}

	Next, we consider second order poles.
	As before, by Proposition (\ref{prop:structure constants}) we have already determined some structure constants arising in the second order poles.
	Therefore, to determine first order poles it is sufficient to analyze the normally ordered differential monomials $W^{i,j}_{1}$, $M^{i,j}_{1}$, and $X^{i,j}_{1}$.
	Moreover by Proposition (\ref{prop:structure constants}), the coefficients of $\tilde W^{i+j,\mu+\nu}$ arising in Jacobi identities $J_{1,1}(W^4,\tilde W^{i,\mu},\tilde W^{j,\nu})$ vanish.
	\begin{lemma}\label{lemma:recursions1}
	Modulo the inductive data, the following hold.
	\begin{enumerate}
		\item  $W^6_{(1)}X^{2n-4}$  determine $\{W^{2i+6}_{(0)}X^{2n-2i-4} | i\geq 1\}$.
		\item $X^2_{(1)}Y^{2n}$ and $W^4_{(0)}X^{2n-2}$ determine $\{X^{2i+2}_{(1)}Y^{2n-2i-2} | i\geq 1\}$.
		\item $X^2_{(1)}X^{2n}$  determine $\{X^{2i+2}_{(1)}X^{2n-2i} | i\geq 1\}$.
		\item $W^3_{(1)}X^{2n-1}$  determine $\{W^{2i+3}_{(1)}X^{2n-2i-2} | i\geq 1\}$.
		\item $W^4_{(1)}W^{2n-1}$  determine $\{W^{2i+4}_{(1)}W^{2n-2i-1}| i\geq 1\}$.
		\item $W^{3}_{(1)}W^{2n-1}$ determine $\{ W^{2i+3}_{(1)}W^{2n-2i-1}| i\geq 1\}$.
	\end{enumerate}
\end{lemma}
	
	\begin{proof}
		Consider a general Jacobi identity $J_{1,1}(W^4,\tilde W^{i,\mu},\tilde W^{j,\nu})$ which reads
		\begin{equation}\label{Recursion1Full}
			\left(1-2w^{\tilde W^4,\tilde W^{i,\mu}}_{\partial \tilde W^{i+2,\mu}}\right)W^{i+2;\mu,j,\nu}_1 =W^{i,\mu;j+2,\nu}_1+R^{i,\mu;j,\nu}_1,
		\end{equation}
		where
		\[R^{i,\mu;j,\nu}_1=W^4_{(1)}W^{i,\mu;j,\nu}_1-{W^{4,0;i,\mu}_0}_{(2)}\tilde W^{j,\nu}.
		\]
		Note that term $R^{i,\mu;j,\nu}_1$ is known from inductive data, so we may write
		\begin{equation}\label{Recurion1Partial}
			W^{i+2,\mu;j,\nu}_1 \equiv \frac{1}{1-2w^{W^4,\tilde W^{i,\mu}}_{\partial \tilde W^{i+2,\mu}}}W^{i,\mu;j+2,\nu}_1.
		\end{equation}
		Iterating recursion (\ref{Recurion0Partial}) gives rise to the desired linear relations.
	\end{proof}
	
	Lastly, we consider the higher order products.
	\begin{lemma}\label{lemma:recursionsr}
	Let $r>1$. Modulo the inductive data, the following hold.
	\begin{enumerate}
		\item  $W^4_{(r)}X^{2n-4}$  determine $\{W^{2i+6}_{(0)}X^{2n-2i-4} | i\geq 1\}$.
		\item $X^2_{(r)}Y^{2n}$ and $W^4_{(r)}W^{2n-2}$ determine $\{X^{2i+2}_{(r)}Y^{2n-2i-2} | i\geq 1\}$.
		\item $X^2_{(r)}X^{2n}$  determine $\{X^{2i+2}_{(r)}X^{2n-2i} | i\geq 1\}$.
		\item $W^3_{(r)}X^{2n-1}$  determine $\{W^{2i+3}_{(r)}X^{2n-2i-2} | i\geq 1\}$.
		\item $W^4_{(r)}W^{2n-1}$  determine $\{W^{2i+4}_{(r)}W^{2n-2i-1}| i\geq 1\}$.
		\item $W^{3}_{(r)}W^{2n-1}$ determine $\{ W^{2i+3}_{(r)}W^{2n-2i-1}| i\geq 1\}$.
	\end{enumerate}
\end{lemma}
	
	\begin{proof}
		Consider Jacobi identity $J_{r,1}(W^4,\tilde W^{i,\mu},\tilde W^{j,\nu})$ which reads
		\begin{equation}\label{rRecution}
			\begin{split}
				\left(r-(r+1)w^{W^4,\tilde W^{i,\mu}}_{\partial \tilde W^{i+2,\mu}}\right){\tilde W^{i+2,\mu}}_{(r)}\tilde W^{j,\nu} ={w^{\tilde W^{i,\mu},\tilde W^{j,\nu}}_{\tilde W^{i+j-2,\mu+\nu}}} W^4_{(r)}\tilde W^{i+j-2,\mu+\nu}+ R^{i,\mu;j,\nu}(r),
			\end{split}
		\end{equation}
		where \[R^{i,\mu;j,\nu}_r=W^4_{(r)}W^{i,\mu;j,\nu}_1-{W^{4,0;i\mu}_0}_{(r+1)}\tilde W^{j,\nu}-\sum_{j=2}^r\binom{r}{j}(W^4_{(j)}\tilde W^{i,\mu})_{(r+1-j)}\tilde W^{j,\nu}.\]
		Note that term $R^{i,\mu;j,\nu}(r)$ is known from inductive data, so we may write
		\begin{equation}\label{Recurionrpartial}
			(\tilde W^{i+2,\mu})_{(r)}\tilde W^{j,\nu} \equiv \frac{w^{\tilde W^{i,\mu},\tilde W^{j,\nu}}_{\tilde W^{i+j-2,\mu+\nu}}}{r-(r+1)w^{W^4,\tilde W^{i,\mu}}_{\partial \tilde W^{i+2,\mu}}}W^{4}_{(r)}\tilde W^{i+j-2,\mu+\nu}.
		\end{equation}
		Iterating recursions (\ref{Recurionrpartial}) gives rise to the desired linear relations.
	\end{proof}
	
	From now on we assume that Jacobi identities used in the proof of Proposition (\ref{aa}) have been imposed.
	In the next series of Lemmas (\ref{0product}-\ref{rproduct}), we write down a small set of Jacobi identities to obtain linear relations among the desired products. 
	This reduces our problem to solving a linear system of equations.

	First, the first-order poles are also uniquely determined.
	\begin{lemma}\label{0product}
		First-order poles in $D^{2n+3}_0$ and $D^{2n+2}_0$ are uniquely determined from the inductive data with $D^{2n+3}_1$ and $D^{2n+2}_1$. Specifically, assuming the relations appearing in Lemma (\ref{lemma:recursions0}) we have the following.
		\begin{enumerate}
			\item Jacobi relations $J_{1,0}(W^3,W^3,W^{2n-2})$ and $J_{0,1}(W^4,W^3,W^{2n-3})$ express $W^{3}_{(0)}W^{2n-1}$, and $W^{4}_{(0)}W^{2n-2}$ in terms of inductive data, and $W^{3,2n-1}_1$. Specifically, we have the following relations.
			\begin{equation*}
				\begin{split}
			W^{3,2n-1}_0\equiv & \frac{2 (n-1)}{n (2 n+1)}\partial W^{3,2n-1}_1  \  \T{mod} \ D^{2n+1}(1,1)\cup D^{2n}(1,1),\\
			W^{4,2n-2}_0 \equiv & \frac{(n-1)!}{n(2n+1)\left(\frac{1}{2}\right)_{n-1}}\partial W^{3,2n-1}_1  \  \T{mod} \ D^{2n+1}(1,1)\cup D^{2n}(1,1).
				\end{split}
			\end{equation*}
		\item Jacobi relation $J_{0,1}(W^4,W^6,W^{2n-5})$ expresses $W^{4}_{(0)}W^{2n-1}$ in terms of inductive data and $W^{6,2n-3}_1 $. Specifically, we have the following relation.
		\[W^{4,2n-1}_0\equiv \frac{5 (2 n-1)}{4 (n+1) (2 n+1)}\partial W^{6,2n-3}_1 \  \T{mod} \ D^{2n+1}(1,1).\]
			\item Jacobi relations $J_{0,1}(W^4,W^6,X^{2n-6})$ and $J_{0,1}(W^4,X^2,W^{2n-3})$ express $W^4_{(0)}X^{2n-2}$, and $X^2_{(0)}W^{2n-1}$ in terms of inductive data, $X^{6,2n-4}_1$, and $X^{2}_{(1)}W^{2n-1}$, respectively. Specifically, we have the following relations.
		\begin{equation*}
			\begin{split}
				X^{4,2n-2}_0\equiv &\frac{5 (n-1)}{2 n (2 n+1)}\partial X^{6,2n-4}_1  \  \T{mod} \ D^{2n}(1,S),\\
				X^{2,2n-1}_0\equiv &-\frac{1}{2n}\partial (X^{2}_{(1)}W^{2n-1}) \  \T{mod} \ D^{2n+1}(1,1).\\
			\end{split}
		\end{equation*}
			\item Jacobi relation  $J_{0,1}(W^4,X^2,Y^{2n-2})$ expresses $X^{2}_{(0)}Y^{2n}$ in terms of inductive data, $M^{2,2n}_1$, $W^{4,2n-2}_0$ and $X^{4,2n-2}_0$. Specifically, we have the following relation.
			\[M^{2,2n}_0\equiv \frac{1}{6} n W^{4,2n-2}_0 -\frac{1}{2 n+1}\partial M^{2,2n}_1-\frac{2 n}{3 (2 n+1)}X^2_{(1)}\sigma X^{4,2n-2}_0\  \T{mod} \ D^{2n}(1,1)\cup D^{2n}(S,S).\]
		\end{enumerate}
	\end{lemma}
		
	Next, the second-order poles are also uniquely determined.
	
	\begin{lemma}\label{1product}
		Second-order poles in $D^{2n+3}_1$ and $D^{2n+2}_1$ are uniquely determined from the inductive data with $D^{2n+3}_2$ and $D^{2n+2}_2$. Specifically, assuming the relations appearing in Lemma (\ref{lemma:recursions1}) we have the following.
		\begin{enumerate}
			\item Jacobi relations $J_{0,2}(W^3,W^3,W^{2n-2})$ and $J_{0,2}(W^4,W^3,W^{2n-3})$ express $W^{3}_{(1)}W^{2n-1}$, and $W^{4}_{(1)}W^{2n-2}$ in terms of inductive data, $W^{3}_{(2)}W^{2n-1}$, and $W^{4}_{(2)}W^{2n-2}$. Specifically, we have the following relations.
			\begin{equation*}
				\begin{split}
					W^{3,2n-1}_1\equiv & \frac{1}{4 (n-1)} \partial (W^{3}_{(2)}W^{2n-1})  \  \T{mod} \ D^{2n+1}(1,1)\cup D^{2n}(1,1),\\
					W^{6,2n-4}_1 \equiv & \frac{3}{2 (n-1)}\partial (W^{4}_{(2)}W^{2n-2})  \  \T{mod} \ D^{2n+1}(1,1)\cup D^{2n}(1,1).
				\end{split}
			\end{equation*}
			\item Jacobi relation $J_{0,2}(W^4,W^6,W^{2n-5})$ expresses $W^{4}_{(1)}W^{2n-1}$ in terms of inductive data and $W^6_{(2)}W^{2n-3} $. Specifically, we have the following relation.
			\[W^{6,2n-3}_1\equiv -\frac{1}{2 n-1} \partial (W^6_{(2)}W^{2n-3}) \  \T{mod} \ D^{2n+1}(1,1).\]
			\item Jacobi relations $J_{0,2}(W^4,W^6,X^{2n-6})$ and $J_{0,2}(W^4,X^2,W^{2n-3})$ expresses $W^4_{(1)}X^{2n-2}$, and $X^{2}_{(1)}W^{2n-1}$  in terms of inductive data, $W^{6}_{(2)}X^{2n-4}$, and $X^{2}_{(2)}W^{2n-1}$, respectively. Specifically, we have the following relations.
			\begin{equation*}
				\begin{split}
					X^{6,2n-4}_1\equiv &-\frac{1}{2 (n-1)}\partial (W^{6}_{(2)}X^{2n-4})  \  \T{mod} \ D^{2n}(1,S),\\
					X^{2}_{(1)}W^{2n-1}\equiv &			\frac{1}{2 n-3}\partial (X^{2}_{(2)}W^{2n-1})  \  \T{mod} \ D^{2n+1}(1,S).
				\end{split}
			\end{equation*}
			\item Jacobi relation  $J_{0,2}(W^4,X^2,Y^{2n-2})$ expresses $X^{2}_{(1)}Y^{2n}$ in terms of inductive data, $X^{2}_{(2)}Y^{2n}$ and $X^{4,2n-2}_0$. Specifically, we have the following relation.
			\[M^{2,2n}_1\equiv \frac{1}{2 (n-1)}\partial (X^{2}_{(2)}Y^{2n})+\frac{n}{3 (n-1)}X^2_{(2)}\sigma X^{4,2n-2}_0\  \T{mod} \ D^{2n}(1,1)\cup D^{2n}(S,S).\]
		\end{enumerate}
	\end{lemma}
	
	Finally, the higher-order poles are also uniquely determined.
\begin{lemma}\label{rproduct}
	Let $r\geq 2$. The $r^{\T{th}}$-order poles in $D^{2n+3}_r$ and $D^{2n+2}_r$ are uniquely determined from the inductive data with $D^{2n+3}_{r+1}$ and $D^{2n+2}_{r+1}$. Specifically, assuming the relations appearing in Lemma (\ref{lemma:recursionsr}) we have the following.
	\begin{enumerate}
		\item Jacobi relations $J_{r,1}(W^3,W^3,W^{2n-2})$, $J_{0,r+1}(W^4,W^3,W^{2n-3})$ and $J_{0,r+1}(W^4,W^4,W^{2n-4})$ express $W^{3}_{(r)}W^{2n-1}$, and $W^{4}_{(r)}W^{2n-2}$ in terms of inductive data, $W^{3}_{(r+1)}W^{2n-1}$, and $W^{4}_{(r+1)}W^{2n-2}$. Specifically, we have the following relations.
		\begin{equation*}
			\begin{split}
				W^{3}_{(r)}W^{2n-1} \equiv &\frac{(r-1) (2 r-3)}{(r+1) (2 n-2 r+1) (n+r-2)}\partial (W^{3}_{(r+1)}W^{2n-1})  \  \T{mod} \ D^{2n+1}(1,1)\cup D^{2n}(1,1),\\
				W^{4}_{(r)}W^{2n-2} \equiv & \frac{3(2 r-3) (n-1)!}{4(r+1)  \left(\frac{1}{2}\right)_{n-1} (n+r-2) (2 n-2 r+1)}\partial (W^{3}_{(r+1)}W^{2n-1})  \  \T{mod} \ D^{2n+1}(1,1)\cup D^{2n}(1,1).
			\end{split}
		\end{equation*}
		\item Jacobi relation $J_{0,r+1}(W^4,W^6,W^{2n-5})$ expresses $W^{4}_{(r)}W^{2n-1}$ in terms of inductive data and $W^{6}_{(r+1)}W^{2n-3}$. Specifically, we have the following relation.
		\[W^{4}_{(r)}W^{2n-1}\equiv -\frac{(r-1) (5 r-3)}{4 r (r+1)}\partial(W^{6}_{(r+1)}W^{2n-3})\  \T{mod} \ D^{2n+1}(1,1).\]
		\item Jacobi relations $J_{0,r+1}(W^4,W^6,X^{2n-6})$ and $J_{0,r+1}(W^4,X^2,W^{2n-3})$ expresses $W^6_{(r)}X^{2n-4}$, and $X^{2}_{(r)}W^{2n-1}$  in terms of inductive data, $W^{4}_{(r+1)}X^{2n-2}$, and $X^{2}_{(r+1)}W^{2n-1}$, respectively. Specifically, we have the following relations.
		\begin{equation*}
			\begin{split}
				W^{4}_{(r)}X^{2n-2}\equiv &-\frac{r-1}{2 (r+1)}\partial (W^{6}_{(r+1)}X^{2n-4})  \  \T{mod} \ D^{2n}(1,S),\\
				X^{2}_{(r)}W^{2n-1}\equiv &	-\frac{1}{r+1}\partial (X^{2}_{(r+1)}W^{2n-1})  \  \T{mod} \ D^{2n+1}(1,S).
			\end{split}
		\end{equation*}
		\item Jacobi relation $J_{1,r}(W^4,X^2,Y^{2n-2})$ and $J_{0,r+1}(W^4,X^4,Y^{2n-4})$ express $X^{2}_{(r)}Y^{2n}$ in terms of inductive data, $X^{2}_{(r+1)}Y^{2n}$, and $W^{4}_{(r)}W^{2n-2}$.  Specifically, we have the following relation.
	\begin{equation*}
		\begin{split}
				X^{2}_{(r)}Y^{2n}\equiv& \frac{(n-1) (3 r-1)}{6 (r-1)}W^{4}_{(r)}W^{2n-2}-\frac{3 r-1}{4 (r+1)}\partial (X^{2}_{(r+1)}Y^{2n}) \  \T{mod} \ D^{2n}(1,1)\cup D^{2n}(S,S).
		\end{split}
	\end{equation*}
	\end{enumerate}
\end{lemma}

Therefore we have proven the following.
	\begin{theorem}\label{thm:induction}
	There exists a nonlinear conformal algebra $\nlcalg$ over the localization of a polynomial ring $D^{-1}\mathbb{C}[c,k]$, with $D$ being the set defined in Proposition \ref{prop:base case}, and satisfying features (\ref{eq:heis}-\ref{eq:raise}) whose universal enveloping vertex algebra $\Wso$ has the following properties.
	\begin{enumerate}\label{Wsp 2 parameters}
		\item It has conformal weight grading \[\Wso=\bigoplus_{N=0}^{\infty}\Wso[N],\quad \Wso[0]=D^{-1}\mathbb{C}[c,k].\]
		\item It is strongly generated by fields $\{X^{2i},Y^{2i} | i\geq 1\}\cup \{H,L,W^{i} | i \geq 3\}$ and satisfies the OPE relations in Proposition \ref{prop:base case}, Jacobi identities in $J_{11}$ and those which appear in Lemmas \ref{lemma:recursions0} -\ref{rproduct}.
		\item It is the unique initial object in the category of vertex algebras with the above properties.
	\end{enumerate} 
	\end{theorem}

	\subsection{Step 4: free generation}
	There are more Jacobi identities than those imposed in Lemmas (\ref{lemma:recursions0}-\ref{rproduct}).
	So it is not yet clear that all Jacobi identities among the strong generators hold as a consequence of (\ref{conformal identity}-\ref{quasi-derivation}) alone, or equivalently, that $\nlcalg$ is a nonlinear Lie conformal algebra \cite{DSK} and $\Wso$ is freely generated. 
	In order to prove it, we will consider certain simple quotients of $\Wso$. 
	Let $I\subseteq R \cong \Wso [0]$
	be an ideal, and let $I\cdot \Wso$ denote the vertex algebra ideal generated by $I$. 
	The quotient 
	\begin{equation}
		\cW^{\fr{so}_2,I}_{\infty} = \Wso / I\cdot \Wso
	\end{equation}
	has strong generators $\{X^{2i},Y^{2i} | i\geq 1\}\cup \{H,L,W^{i} | i \geq 3\}$ satisfying the same OPE relations as the corresponding  generators of $\Wso$ where all structure constants in $R$ are replaced by their images in $R/I$.
	
	We now consider a localization of $\cW^{\fr{so}_2,I}_{\infty}$.
	Let $E \subseteq R/I$ be a multiplicatively closed set, and let $S= E^{-1}(R/I)$ denote the localization of $R/I$ along $S$.
	\\
	Thus we have localization of $R/I$-modules
	\[\cW^{\fr{so}_2,I}_{\infty, S}= S\otimes_{R/I}\cW^{\fr{so}_2,I}_{\infty},\]
	which is a vertex algebra over $S$.
	
	\begin{theorem}\label{one-parameter quotients theorem}
		Let $R$, $I$, $E$, and $S$ be as above, and let $\cW$ be a simple vertex algebra over $S$ with the following properties.
		\begin{enumerate}
			\item $\cW$ is generated by affine fields $\bar {H}^1$, Virasoro field $\bar{L}$ of central charge $c$ and weight 2 primary fields $\bar{X}^2,\bar{Y}^2$.
			\item Setting $\bar{W}^{n+2} =\bar{W}^4_{(1)} \bar{W}^{n}$ for all $i\geq 2$, the OPE relations for $\bar{W}^{n}(z)\bar{W}^{m}(z)$ for $n+m\leq 9$ are the same as in $\Wso$ if the structure constants are replaced with their images in $S$.
		\end{enumerate}
		Then $\cW$ is the simple quotient of $\cW^{\fr{so}_2,I}_{\infty,S}$ by its maximal graded ideal $\cI$.
	\end{theorem}
	The proof is similar to that of $\Winf$,$\Wev$ and $\Wsp$, and is omitted.
	
\subsubsection{$\cW$-algebras of $\gs\go_2$-rectangular type}
	
	In this subsection, we will prove that the $\cW$-algebras $\cC^{\psi}_{CD} = \cW^{\psi - 2n-1}(\gs\gp_{4n}, f_{2n, 2n})$ and $\cW^{\psi}_{BD} = \cW^{\psi -  4n}(\gs\go_{4n+2}, f_{2n+1, 2n+1})$ both arise as quotients of $\cW^{\gs\go_2}_{\infty}$ as $1$-parameter vertex algebras. This requires proving that they have the weak generation property, i.e, they are generated by the fields in weights one and two.

\begin{theorem} \label{weakgeneration:Walgebras} For all $m\geq 2$, consider the following $\cW$-algebras: 
\begin{enumerate}
\item $\cC^{\psi}_{CC}(0,m) = \cW^{\psi - 2m-2}(\gs\gp_{4m}, f_{2m, 2m})$, 
\item $\cC^{\psi}_{BD}(0,m) = \cW^{\psi -  4n}(\gs\go_{4n+2}, f_{2n+1, 2n+1})$.
\end{enumerate}
These have the weak generation property $\cC^{\psi}_{BD}(0,m) = \tilde{\cC}^{\psi}_{BD}(0,m)$ and $\cC^{\psi}_{CC}(0,m) = \tilde{\cC}^{\psi}_{CC}(0,m)$ for all $\psi \in \mathbb{C}$, with the possible exception of 
\begin{enumerate}
\item The critical value $\psi = 0$,
\item The values of $\psi$ where the denominator of any structure constants in $\cW^{\gs\go_2}_{\infty}$ vanishes (this is a finite set).
\end{enumerate}
\end{theorem}

\begin{proof} 

We only give the proof for $\cC^{\psi}_{BD}(0,n)$ since the proof for  $\cC^{\psi}_{CC}(0,n)$ is the same. 
Let $n$ be fixed, and recall that $\cW^{\psi - 4n}(\gs\go_{4n+2},f_{2n+1, 2n+1})$ has strong generating type 
\begin{equation}\label{gentype}
    \cW(1,2^3,3,4^3,\dots,(2n-2)^3,2n-1,(2n)^3,2n+1).
\end{equation}
The strong generating set is in bijection with a basis for Lie algebra $\gs\go_{4n+2}^f$.
Let $\{w^{i}|i\leq 2n+1\}\cup \{x^{2i}, y^{2i}|i\leq n\}$ be a basis of $\gs\go_{4n+2}^f$, and $\{W(w^{i})|i\leq 2n+1\}\cup \{W(x^{2i}), W(y^{2i})|i\leq n\}$ be some choice of strong generators of the $\cW$-algebra, so that the structure constants are given by polynomial functions of $\psi$.
This implies that for all $j\leq n-1$, we have the following relations in the $\cW$-algebra
$$W(w^3)_{(0)} W(x^{2j}) = c_{3,2j} W(x^{2j+2}) + \cdots,\quad W(w^3)_{(0)} W(y^{2j}) = d_{3,2j} W(y^{2j+2}) + \cdots,$$ where structure constants $c_{3,2j}$ and $d_{3,2j}$ are nonzero, and arise via the following Lie algebra relations
$$[w^{3}, x^{2j} ] = c_{3,2j} x^{2j+2},\qquad [w^{3}, y^{2j}]= d_{3,2j} y^{2j+2},$$ 
and the remaining terms are normally ordered monomials in the strong generators $\{W(w^{i})|i\leq 2n\}\cup \{W(x^{2i}), W(y^{2i})|i< n\}$ and their derivatives. 
Similarly, for $j\leq n$, we have $$W(x^2)_{(0)} W(y^{2j}) = e_{2,2j} W(w^{2j+1}) + \cdots,\quad [{x^{2}}, {y^{2j}}] = e_{2,2j} {w^{2j+1}},\quad e_{2,2j}\neq 0.$$ 
Using the nonzero Lie algebra structure constants $\{c_{3,2j},d_{3,2j}|j\leq n-1\}\cup \{e_{2,2j}|j\leq n\}$ and Proposition \ref{prop:structure constants}, we will establish the weak generation property for the $\cW$-algebra.

Vertex algebra identities (\ref{conformal identity}-\ref{Jacobi}) together with the known strong generating type (\ref{gentype}), imply that there exists unique choice for Heisenberg generator $W(w^1)=H$ with a given level, and for generators $W(x^2)$ and $W(y^2)$, that are primary of charge $2$, and $-2$, and are Virasoro primary of weight two. So we may write 
$$H = W(w^1),\qquad L = W(w^2),\qquad X^2 = W(x^2),\qquad Y^2 = W(y^2).$$ 
By our choice of strong generators, any OPE structure constant among $\{W(w^i), W(x^{2j}),W(y^{2j})| \ i \geq 3, \ j \geq 2\}$ that does not involve $H, L, X^2, Y^2$, is a polynomial in $\psi$. For $1\leq i\leq 2n+1$ and $1\leq j\leq n $, we can write
$$W^{i} = \lambda_{w^i}W(w^{i}) + \dots,\quad X^{2j} = \lambda_{x^{2j}} W(x^{2j}) + \cdots, \quad Y^{2j} = \lambda_{y^{2j}} W(y^{2j}) + \cdots,$$ where $\lambda_{w^i}$, $\lambda_{x^{2j}}$, $\lambda_{y^{2j}}$ are polynomials in $\psi$, and the remaining terms are normally ordered polynomials in previous generators. 
To prove the theorem, it suffices to show that $\lambda_{w^i}$, $\lambda_{x^{2i}}$, $\lambda_{y^{2i}}$ are all nonzero constants for the defined range of indices.

According to Proposition \ref{prop:structure constants},
$$X^2_{(0)} Y^{2} = \frac{9}{4}\frac{ 2!! 2!!}{3!!} W^{3}+\dotsb = 3 W^3+\dotsb,$$ while on the other hand
$$X^2_{(0)} Y^{2} = W(x^2)_{(0)} W(y^{2}) = e_{2,2} W(w^{3}) + \cdots = \frac{e_{2,2}}{\lambda_{w^3}}W^{3}+\dotsb.$$
Therefore $\frac{e_{2,2}}{\lambda_{w^3}} =  3$, so $\lambda_{w^3} = -\frac{e_{2,2}}{3}$, which is a nonzero constant.

Next, by Proposition \ref{prop:structure constants}, we have $$W^3_{(0)} Y^{2j} = \frac{3!! (2j)!!}{(2j+2)!!} Y^{2j+2}+\dotsb,$$ while on the other hand $$W^3_{(0)} Y^{2j} = \lambda_{w^3}\lambda_{y^{2j}}  W(w^3)_{(0)} W(y^{2j} ) + \cdots = \lambda_{w^3} \lambda_{y^{2j}}  d_{3,2j} W(y^{2j+2}) + \cdots = \frac{\lambda_{w^3} \lambda_{y^{2j}}  d_{3,2j}}{\lambda_{w^{2j+2}}} Y^{2j+2} + \cdots.$$ We get $$ \frac{\lambda_{w^3} \lambda_{y^{2j}}  d_{3,2j}}{\lambda_{y^{2j+2}}} = \frac{3!! (2j)!!}{(2j+2)!!} .$$ Inductively, we get that $\lambda_{y^{2j}}$ is nonzero constant for all $2\leq j\leq n$.
A similar argument shows that  $\lambda_{x^{2j}}$ is nonzero constant for all $2\leq j\leq n$.

Next, by Proposition \ref{prop:structure constants}, for $j \geq 2$ we have
$$X^2_{(0)} Y^{2j} = \frac{9}{4}\frac{2!! (2j)!!}{(2j+1)!!} W^{2j+1}+\dotsb,$$ while on the other hand
$$X^2_{(0)} Y^{2j} = \lambda_{y^{2j}} W(x^2)_{(0)}W(y^{2j})  + \cdots =\lambda_{y^{2j}}  e_{2,2j} W(w^{2j+1}) + \cdots = \frac{\lambda_{y^{2j}}  e_{2,2j}}{\lambda_{w^{2j+1}}} W^{2j+1} + \cdots.$$
We get $$\frac{\lambda_{y^{2j}}  e_{2,2j}}{\lambda_{w^{2j+1}}} =\frac{9}{4} \frac{2!! (2j)!!}{(2j+1)!!}.$$
Inductively, we get that $\lambda_{w^{2j+1}}$ is a nonzero constant for all $1\leq j \leq n$.

It remains to show that $\lambda_{w^{2j}}$ is a nonzero constant for all $2\leq j \leq n$, and for this we need a different argument. First, we observe that 
$$W(w^3)_{(1)} W(w^{2j}) = d_{2j} W(w^{2j+1}) + \cdots,$$ for some structure constant $d_{2j}$, which a priori is a polynomial in $\psi$.
Then we have
 $$W^3_{(1)} W^{2j} = \frac{3!! (2j)!!}{8(2j-1)!!} W^{2j+1} + \cdots=\frac{\lambda_{w^3} \lambda_{w^{2j}}  d_{2j} }{\lambda_{w^{2j+1}}} W^{2j+1} + \cdots,$$ so that
$$\frac{\lambda_{w^3} \lambda_{w^{2j}}  d_{2j} }{\lambda_{w^{2j+1}}}  =  \frac{3!! (2j)!!}{8(2j-1)!!}.$$
In particular, $\lambda_{w^{2j}}  d_{2j}$ is a nonzero constant, and since both $\lambda_{w^{2j}}$ and $d_{2j}$ are polynomials in $\psi$, they must both be nonzero constants.
\end{proof}

	\begin{corollary}\label{Wso freely generated}
		All Jacobi identities among generators $\{\tilde W^{i,\mu} | i\geq 1\}$ holds as consequences of (\ref{conformal identity}-\ref{quasi-derivation}) alone, so $\nlcalg$ is a nonlinear Lie conformal algebra with generators $\{W^{i} | i\geq 1\}$. Equivalently, $\Wso$ is freely generated by even fields  $\{X^{2i},Y^{2i} | i\geq 1\}\cup \{H,L,W^{i} | i \geq 3\}$, and has graded character
		\begin{equation}\label{character}
			\chi (\Wso,q) =\sum_{n=0}^{\infty}rank_{R}(\Wso [n])q^n = \prod_{i,j,l=1}^{\infty}\frac{1}{(1-q^{2i+2j+1})(1-q^{2i+2l+2})^3}.
		\end{equation}
		For any prime ideal $I \subseteq R$, $\cW^{\fr{so}_2,I}_{\infty}$ is freely generated by $\{W^{i} | i\geq 1\}$ as a vertex algebra over $R/I$ and
		\[\chi (\cW^{\fr{so}_2,I}_{\infty},q) =\sum_{n=0}^{\infty}rank_{R/I}(\cW^{\fr{so}_2,I}_{\infty}[n])q^n = \prod_{i,j,l=1}^{\infty}\frac{1}{(1-q^{2i+2j+1})(1-q^{2i+2l+2})^3}.\]
		For any localization $S=(E^{-1}R)/I$ along a multiplicatively closed set $E\subseteq R/I$, $\cW^{\fr{so}_2,I}_{\infty,S}$ is freely generated by $\{W^{i} | i\geq 1\}$ and 
		\[\chi (\cW^{\fr{so}_2,I}_{\infty,S},q) = \sum_{n=0}^{\infty}rank_{S}(\cW^{\fr{so}_2,I}_{\infty,S}[n])q^n = \prod_{i,j,l=1}^{\infty}\frac{1}{(1-q^{2i+2j+1})(1-q^{2i+2l+2})^3}.\]
	\end{corollary}
	\begin{corollary} \label{cor:simplicity} 
		The vertex algebra $\Wso$ is simple.
	\end{corollary}
    \begin{proof}
		If $\Wso$ were not simple, it would have a singular vector $\omega$ in some weight $N$. Let $p\in R$ be an irreducible polynomial, and $I$ be the ideal $(p)\subseteq R$. By rescaling if necessary, we may assume that $\omega$ is not divisible by $p$, and hence descends to a nontrivial singular vector in $\cW^{\gs\go_2,I}_{\infty}$. For any localization $S$ of $R/I$, the simple quotient of $\cW^{\gs\go_2,I}_{\infty,S}$ would then have a smaller weight $N$ submodule than $\cW^{\gs\go_2,I}_{\infty,S}$ for all such $I$. This contradicts Theorem \ref{weakgeneration:Walgebras}, since $\cC^{\psi}_{BD}(0,m)$ is a quotient of $\cW^{\gs\go_2,I}_{\infty,S}$ for some $I$ and $S$, and has the same character as $\cW^{\gs\go_2}_{\infty}$ up to weight $2m$.
	\end{proof}

	Next, we describe the automorphisms of $\Wso$.
	\begin{proposition}\label{prop:autos}
		The vertex algebra $\Wso$ has full automorphism group $\text{O}_2(\mathbb{C})$.
	\end{proposition}
	\begin{proof}
		Let $g$ be any automorphism of $\Wso$ that is not contained in $\T{SO}_2$-subgroup of automorphisms, so it fixes the Heisenberg generator $H$.  By definition, $g$ preserves the Virasoro generator $L$. Then, the most general deformation of weight 2 fields $X^2$ and $Y^2$ compatible with OPEs
		\[H(z) X^2(w) \sim X^2(w) (z-w)^{-1} ,\quad H(z) Y^2(w) \sim -Y^2(w) (z-w)^{-1},\quad X^2(z)Y^{2}(w)\sim (3W^{3}+\dotsb)(w)(z-w)^{-1}+\dotsb,\] are the scaling maps $g_+$ and $g_-$ defined by
		\[g_+(X^2) = x Y^2,\quad g_+(Y^2) = x^{-1} X^2, \quad g_-(X^2) = x X^2,\quad g_-(Y^2) = -x^{-1} Y^2\quad x  \in\mathbb{C}\setminus \{0\}.\]
		Since weight 2 fields weakly generate $\Wso$, the map $g$ is fully determined by $x$. Thanks to OPE $X^2(z)Y^{2}(w)$ we obtain that $g_+(W^3)=W^3$ and $g_-(W^3)=-W^3$.
	\end{proof}

	\subsection{Quotients by maximal ideals of $\Wso$}
	So far, we have considered quotients of the form $\cW^{\fr{so}_2,I}_{\infty}$ which are 1-parameter vertex algebras in the sense that $R$ has Krull dimension 1. Here, we consider simple quotients of $\cW^{\fr{so}_2,I}_{\infty}$ where $I \subseteq R$ is a maximal ideal.
	Such an ideal always has the form $I=(c-c_0,k-k_0)$ for $c_0,k_0 \in\mathbb{C}$, and $\cW^{\fr{so}_2,I}_{\infty}$ is a vertex algebra over $\mathbb{C}$.
	We first need a criterion for when the simple quotients of two such vertex algebras are isomorphic.
	\begin{theorem}\label{max quotients}
		Let $c_0,c_1,k_0,k_1$ be complex numbers and let \[I_0 = (c-c_0,k-k_0),\quad I_1 =(c-c_1,k-k_1)\]
		be the corresponding maximal ideals in $R$.
		Let $\cW_0$ and $\cW_1$ be the simple quotients of $\cW^{\fr{so}_2,I_0}_{\infty}$ and $\cW^{\fr{so}_2,I_1}_{\infty}$, respectively. 
		Then $\cW_0 \simeq \cW_1$ if and only if:
        \begin{itemize}
            \item $c_0 = c_1$ and $k_0=k_1$,
            \item $c_0 =-1= c_1$, and $k_0=k_1$,
            \item $k_0 =0= k_1$, and $c_0=c_1$.
        \end{itemize}
	\end{theorem}
    \begin{proof}
        Clearly condition (1) implies an isomorphism. 
        The only way it can fail is if one of the parameters disappears in the simple quotient.
        This happens precisely in the two cases (2) and (3).
        In (2), $L$ becomes singular and the simple quotient is the Heisenberg algebra $\pi^k$.
        In (3), $H$ becomes singular and the simple quotient is the Virasoro algebra $\vir$.
    \end{proof}
	
	\begin{corollary}\label{max quotients2}
		Let $I=(p)$ and $J=(q)$ be prime ideals in $R$ such that $\cW^{\fr{so}_2,I}_{\infty}$ and $\cW^{\fr{so}_2,J}_{\infty}$ are not simple. Then any pointwise coincidences between the simple quotients of $\cW^{\fr{so}_2,I}_{\infty}$ and $\cW^{\fr{so}_2,J}_{\infty}$ must correspond to intersection points of the truncation curves $V(I)\cap V(J)$.
	\end{corollary}
	
	\begin{corollary}
		Suppose that $\cA$ is a simple, 1-parameter vertex algebra which is isomorphic to the simple quotient of $\cW^{\fr{so}_2,I}_{\infty}$ for some prime ideal $I\subseteq R$, possibly after localization. 
		Then if $\cA$ is the quotient of $\cW^{\fr{so}_2,I}_{\infty}$ for some prime ideal $J$, possibly localized, we must have $I=J$.
	\end{corollary}
	\begin{proof}
		This is immediate from Theorem (\ref{max quotients}) and Corollary (\ref{max quotients2}), since if $I$ and $J$ are distinct prime ideals, their truncation curves $V(I)$ and $V(J)$ can intersect in at most finitely many points. The simple quotients of $\cW^{\fr{so}_2,I}_{\infty}$ and $\cW^{\fr{so}_2,J}_{\infty}$ therefore cannot coincide as one-parameter families. \end{proof}

	\section{One-parameter quotients of $\cW^{\fr{so}_2}_{\infty}$} \label{sect:1paramquot}
	
In this section, we prove that the $8$ families of $Y$-algebras of type $\gs\go_2$ introduced in Section \ref{sect:YtypeSO}, as well as the $4$ additional families of diagonal cosets introduced in Section \ref{sect:Diagonal}, all arise as $1$-parameter quotients of $\cW^{\fr{so}_2}_{\infty}$. This means that there exist localizations $E^{-1} R/I_{XY,n,m}$ and $F^{-1}\mathbb{C}[\ell]$ such we have isomorphisms of $1$-parameter vertex algebras
$$ E^{-1} R/I_{XY,n,m} \otimes_{R/I_{XY,n,m}} \cW^{\gs\go_2}_{\infty, I_{XY,n,m}} \cong \cC^{\psi}_{XY}(n,m),\qquad F^{-1} R/I_n \otimes_{R/I_n} \cW^{\gs\go_2}_{\infty, I_n} \cong \cC^{\ell}(n),$$
 where $E$ and $F$ are the (finite) sets of denominators of structure constants of $\cW^{\gs\go_2}_{\infty}$ after replacing $c, k$ with the corresponding functions of $\psi$ (respectively $\ell$). Throughout this section, we omit these localizations from our notation. We have already proven this for $\cC^{\psi}_{CC}(0,m)$ and $\cC^{\psi}_{BD}(0,m)$ in Theorem \ref{weakgeneration:Walgebras}. Of the remaining cases, the easiest are the following.
	 
	 \begin{theorem} The following algebras have this property
\begin{enumerate}
		\item For $n\geq 1$, $\text{Com}(V^{\ell}(\fr{so}_{2n+1}), V^{\ell}(\fr{so}_{2n+3}))^{\mathbb{Z}_2}$ (Case $\cC^{\psi}_{BB}(n,0)$),
		\item For $n\geq 1$, $\text{Com}(V^{\ell}(\fr{so}_{2n}), V^{\ell}(\fr{so}_{2n+2}))^{\mathbb{Z}_2}$ (Case $\cC^{\psi}_{BD}(n,0)$),
		\item For $n\geq 1$, $\text{Com}(V^{\ell}(\fr{sp}_{2n}), V^{\ell}(\fr{osp}_{2|2n}))$ (Case $\cC^{\psi}_{BC}(n,0)$),	
		\item For $n\geq 1$, $\text{Com}(V^{\ell}(\fr{osp}_{1|2n}), V^{\ell}(\fr{osp}_{3|2n}))^{\mathbb{Z}_2}$ (Case $\cC^{\psi}_{BO}(n,0)$),
		\item For $n\geq 2$, $\text{Com}(V^{\ell}(\fr{so}_{n}), V^{\ell-2}(\fr{so}_{n}) \otimes \cF(2n))^{\mathbb{Z}_2}$,
		\item For $n\geq 1$, $\text{Com}(V^{\ell}(\fr{sp}_{2n}), V^{\ell+1}(\fr{sp}_{2n}) \otimes \cS(2n))$,
		\item For $n\geq 1$, $\text{Com}(V^{\ell}(\fr{osp}_{1|2n}), V^{\ell+1}(\fr{osp}_{1|2n}) \otimes \cS(2n)\otimes \cF(2))^{\mathbb{Z}_2}$.
	
	\end{enumerate}
\end{theorem}
	
\begin{proof} In these cases, the free field limits are as follows:
\begin{enumerate}
		\item $\lim_{\psi \rightarrow \infty} \cC^{\psi}_{BB}(n,0) \cong \cO_{\text{ev}}(2(2n+1),2)^{\text{O}_{2n+1}}$, 
		\item $\lim_{\psi \rightarrow \infty} \cC^{\psi}_{BD}(n,0) \cong  \cO_{\text{ev}}(2(2n),2)^{\text{O}_{2n}}$, 
		\item $\lim_{\psi \rightarrow \infty} \cC^{\psi}_{BC}(n,0) \cong \cO_{\text{odd}}(2n,2)^{\text{Sp}_{2n}}$,
		\item $\lim_{\psi \rightarrow \infty} \cC^{\psi}_{BO}(n,0) \cong  (\cO_{\text{odd}}(2n,2) \otimes \cO_{\text{ev}}(2,2))^{\text{Osp}_{1|2n}}$,
		
		\item For $n\geq 2$, $\lim_{\ell \rightarrow \infty} \cC^{\ell}(n) \cong \cO_{\text{odd}}(2n,1)^{\text{O}_n}$, 
		\item For $n\geq 1$, $\lim_{\ell \rightarrow \infty} \cC^{\ell}(-2n) \cong \cS_{\text{ev}}(2n,1)^{\text{Sp}_{2n}}$, 
		\item For $n\geq 1$, $\lim_{\ell \rightarrow \infty} \cC^{\ell}(-2n+1) \cong \cS_{\text{ev}}(2n,1) \otimes \cO_{\text{odd}}(2,1))^{\text{Osp}_{1|2n}}$. 
	\end{enumerate} It is straightforward to check that these free field limits have the desired weak generation property, which therefore holds at generic levels.
\end{proof}

However, this argument does not work for $\cC^{\psi}_{XY}(n,m)$ when $m \geq 1$ since the weak generation property fails in the large level level limit, so we need a different approach.

\begin{lemma} Suppose $\cW$ is a $1$-parameter vertex algebra satisfying the symmetry and strong generation requirements, but not necessarily the weak generation hypotheses. Consider the subalgebra $\tilde{\cW} \subseteq \cW$ generated by the fields in weights at most $4$. Then $\tilde{\cW}$ is of type $\cW(1, 2^3, 3, 4^3,\dots)$, and is a $1$-parameter quotient of $\cW^{\fr{so}_2}_{\infty}$ which need not be simple. 
\end{lemma}

\begin{proof} This is similar to the proof of Lemma 5.10 and Theorem 5.4 of \cite{CL4}, and is omitted.
\end{proof}

\begin{lemma} \label{upto2m+2} For $X = B,C$ and $Y = B,C,D,O$, let $\tilde{\cC}^{\psi}_{XY}(n,m) \subseteq \cC^{\psi}_{XY}(n,m)$ be the subalgebra generated by the fields in weight at most $4$. Suppose that $\tilde{\cC}^{\psi}_{XY}(n,m)$ contains the strong generators up to weight $2m+2$. Then $\tilde{\cC}^{\psi}_{XY}(n,m) = \cC^{\psi}_{XY}(n,m)$, and in particular, $\tilde{\cC}^{\psi}_{XY}(n,m)$ is simple.
\end{lemma}

\begin{proof} If $\tilde{\cC}^{\psi}_{XY}(n,m)$ contains the strong generators up to weight $2m+2$, so does its large level limit $\lim_{\psi \ra \infty} \tilde{\cC}^{\psi}_{XY}(n,m)$. We will show that the subalgebra of $\lim_{\psi \ra \infty} \cC^{\psi}_{XY}(n,m)$ generated by the fields in weights up to $2m+2$ is all of $\lim_{\psi \ra \infty} \cC^{\psi}_{XY}(n,m)$, which proves the lemma. For $X=B$, in the large level limit the fields in weights $1,2,\dots,2m+1$ decouple from the ones in higher weight. These higher weight fields are the generators of the following orbifolds
\begin{enumerate}
\item $\cO_{\T{ev}}(4n+2,2m+2)^{\text{SO}_{2n+1}}$, for $\cC^{\psi}_{BB}(n,m)$,
\item $ \cS_{\T{odd}}(2n,2m+2)^{\text{Sp}_{2n}}$, for $\cC^{\psi}_{BC}(n,m)$,
\item $\cO_{\T{ev}}(4n,2m+2)^{\text{SO}_{2n}}$, for $\cC^{\psi}_{BD}(n,m)$,
\item $(\cS_{\T{ev}}(1,2m+2)\otimes \cS_{\T{odd}}(2n,2m+2))^{\text{Osp}_{1|2n}}$, for $\cC^{\psi}_{BO}(n,m)$.
\end{enumerate}
In all cases, it can be checked easily that these orbifolds are generated by the fields in weights $2m+2$. The proof is similar to the proof of Lemmas 5.1-5.8 of \cite{CL4}, and is omitted. Similarly, in the case $X=C$, the fields in weights $1,2,\dots, 2m$ decouple from the ones in higher weight. The fields in higher weights are the generators of the following orbifolds
\begin{enumerate}
\item $\cO_{\T{odd}}(4n+2,2m+1)^{\text{SO}_{2n+1}}$, for $\cC^{\psi}_{CB}(n,m)$,
\item $ \cS_{\T{ev}}(2n,2m+1)^{\text{Sp}_{2n}}$, for $\cC^{\psi}_{CC}(n,m)$,
\item $\cO_{\T{odd}}(4n,2m+1)^{\text{SO}_{2n}}$, for $\cC^{\psi}_{CD}(n,m)$,
\item $( \cS_{\T{ev}}(2n,2m+1)\otimes \cO_{\T{odd}}(1,2m+1))^{\text{Osp}_{1|2n}}$, for $\cC^{\psi}_{CO}(n,m)$.
\end{enumerate}
Again, it is straightforward to check that these orbifolds are generated by the fields in weights $2m+1$ and $2m+2$.
\end{proof}

\begin{theorem} \label{quotientproperty} For $m\geq 1$, $X = B,C$, and $Y = B,C,D,O$, $\tilde{\cC}^{\psi}_{XY}(n,m) = \cC^{\psi}_{XY}(n,m)$ as $1$-parameter vertex algebras. In particular, $\cC^{\psi}_{XY}(n,m)$ is a simple, $1$-parameter quotient of $\cW^{\fr{so}_2}_{\infty}$, and we have 
$$\cC^{\psi}_{XY}(n,m) \cong \cW^{\fr{so}_2}_{\infty, I_{XY,n,m}}.$$ In this notation, $I_{XY,n,m} \subseteq \mathbb{C}[c,k]$ is the ideal generated by $c = c_{XY}$, where $c_{XY}$ is given by \eqref{eq:cBD}-\eqref{eq:cCO}, and $\cW^{\fr{so}_2}_{\infty, I_{XY,n,m}}$ is the simple quotient of $\cW^{\fr{so}_2}_{\infty} / I_{XY,n,m} \cdot \cW^{\fr{so}_2}_{\infty}$ by its maximal graded ideal. \end{theorem}
	
\begin{proof} First, to show that $\tilde{\cC}^{\psi}_{XY}(n,m) = \cC^{\psi}_{XY}(n,m)$ as $1$-parameter vertex algebras, it is enough to find a level $\psi$ where the simple quotients coincide, and there is no singular vector in weight below $2m+2$. Consider first the case of $\tilde{\cC}^{\psi}_{BD}(n,m)$ for $n,m\geq 1$. 

The truncation curves for $\tilde{\cC}^{\psi}_{BD}(n,m)$ and $\tilde{\cC}^{\psi'}_{BD}(0,r)$ intersect at the point when 
$\psi = \frac{1 + 2 m + n}{1 + m + r}$ and $\psi' = \frac{1 + 2 r - n}{1 + m + r}$, 
so that $k =\frac{1 - 4 m n - r - 2 n r}{1 + m + n}$. It follows from Theorem \ref{weakgeneration:Walgebras} that for infinitely many choices of $r$, this is a point where $\cC^{\psi'}_{BD}(0,r)$ has the weak generation property $\tilde{\cC}^{\psi'}_{BD}(0,r) = \cC^{\psi'}_{BD}(0,r)$. 


Using Weyl's second fundamental of invariant theory for the standard representation of $\text{SO}_{2n}$ \cite{W}, one checks that first relation among the generators of $\tilde{\cC}^{\psi}_{BD}(n,m)$ occurs at weight
$w = 2 (1 + m + 2 n + 2 m n + n^2)$ for generic $\psi$. There can be relations in lower weight only for finitely many values of $k$. Therefore we can find $r$ such that at the corresponding value of $\psi$ the first relation occurs at weight $w$. Since $\tilde{\cC}^{\psi}_{BD}(n,m)$ has the same OPE algebra as $\cC^{\psi}_{BD}(0,r)$ and there are no relations below weight $w$, $\tilde{\cC}^{\psi}_{BD}(n,m)$ must contain the strong generators of $\cC^{\psi}_{BD}(n,m)$ up to weight $w > 2m+2$. By Lemma \ref{upto2m+2}, $\tilde{\cC}^{\psi}_{BD}(n,m) = \cC^{\psi}_{BD}(n,m)$ at this value of $\psi$, and hence this holds as $1$-parameter vertex algebras. The proof for the other families other than $\cC^{\psi}_{CC}(0,m)$ and $\cC^{\psi}_{BD}(0,m)$, is similar and is based on the following computations.


\begin{enumerate}
\item The truncation curves for $\tilde{\cC}^{\psi}_{BB}(n,m)$, $\tilde{\cC}^{\psi'}_{BD}(0,r)$ intersect when 
$\psi = \frac{3 + 4 m + 2 n}{2 (1 + m + r)},\ \psi' = \frac{1 + 4 r - 2 n}{2 (1 + m + r)}$.
\item The truncation curves for $\tilde{\cC}^{\psi}_{BC}(n,m)$, $\tilde{\cC}^{\psi'}_{BD}(0,r)$ intersect when 
$\psi = \frac{1 + 2 m - n}{1 + m + r}$ and $\psi' = \frac{-m + r + n}{1 + m + r}$.
\item The truncation curves for $\tilde{\cC}^{\psi}_{BO}(n,m)$, $\tilde{\cC}^{\psi'}_{BD}(0,r)$ intersect when 
$\psi = \frac{3 + 4 m - 2 n}{2 (1 + m + r)}$ and $\psi' = \frac{1 + 4 r + 2 n}{2 (1 + m + r)}$. 
\item The truncation curves for $\tilde{\cC}^{\psi}_{CC}(n,m)$, $\tilde{\cC}^{\psi'}_{BD}(0,r)$ intersect when 
$\psi = \frac{1 + 2 m + n}{1 + 2 m + 2 r}$ and $\psi' = \frac{2 (2 r - n)}{1 + 2 m + 2 r}$.
\item The truncation curves for $\tilde{\cC}^{\psi}_{CB}(n,m)$, $\tilde{\cC}^{\psi'}_{BD}(0,r)$ intersect when 
$\psi = \frac{1 + 4 m - 2 n}{2 (1 + 2 m + 2 r)}$ and $\psi' = \frac{1 + 4 r + 2 n}{1 + 2 m + 2 r}$.
\item The truncation curves for $\tilde{\cC}^{\psi}_{CD}(n,m)$, $\tilde{\cC}^{\psi'}_{BD}(0,r)$ intersect when
$\psi = \frac{1 + 2 m - n}{1 + 2 m + 2 r}$ and $\psi' = \frac{2 (2 r + n)}{1 + 2 m + 2 r}$.
\item The truncation curves for $\tilde{\cC}^{\psi}_{CO}(n,m)$, $\tilde{\cC}^{\psi'}_{BD}(0,r)$ intersect when 
$\psi =\frac{1 + 4 m + 2 n}{2 (1 + 2 m + 2 r)}$ and $\psi' = \frac{1 + 4 r - 2 n}{1 + 2 m + 2 r}$.
\end{enumerate}\end{proof}
	
\begin{corollary} \label{weakgeneration:finiteness} For generic values of $\psi_0 \in \mathbb{C}$ such that the shifted level of $V^a(\ga)$ does not lie in $\mathbb{Q}_{\leq 0}$, $\cC^{\psi_0}_{XY}(n,m)$ is a quotient of $\cW^{\gs\go_2}_{\infty}$. Similarly, for $n \in \mathbb{Z}$, $\cC^{\ell}(n)$ is a quotient of $\cW^{\fr{so}_2}_{\infty}$ for generic values of $\ell$. \end{corollary}

\begin{proof} Fix $\psi_0 \in \mathbb{C}$ such that the shifted level of $V^a(\ga) \notin \mathbb{Q}_{\leq 0}$. By \cite[Theorem 8.1]{CL2} that $\cC^{\psi_0}_{XY}(n,m)$ is obtained from the $1$-parameter vertex algebra $\cC^{\psi}_{XY}(n,m)$ by setting $\psi = \psi_0$.
\end{proof}

\begin{corollary} \label{notrialities} There are no isomorphisms between distinct $1$-parameter vertex algebras of the form $\cC^{\psi}_{XY}(n,m)$.
\end{corollary}
	
\begin{proof} Since the algebras $\cC^{\psi}_{XY}(n,m)$ all arise as $1$-parameter quotients of $\cW^{\fr{so}_2}_{\infty}$ of the form $\cW^{\fr{so}_2}_{\infty, I_{XY,n,m}}$, it suffices to show that the ideals $I_{XY,n,m}$ are all distinct. But this is straightforward to check using the formulas for these ideals given by \eqref{eq:cBD}-\eqref{eq:cCO}.
\end{proof}

	\section{Reconstruction} \label{sect:recon} 
	Recall the $\fr{so}_2$-rectangular $\cW$-(super)algebra $\cW^{\psi}_{XY}(n,m)$ from Section \ref{sect:YtypeSO} whose coset by $V^{a}(\fr{a})$ is the $\fr{so}_2$-rectangular $Y$-algebra $\cC^{\psi}_{XY}(n,m)$, where $\fr{a}$-level $a$ is determined in terms of $\fr{so}_2$-level $k$, as in the Table (\ref{tab:Walgebras}). 
    Proposition \ref{extension features} implies that $\cW^{\psi}_{XY}(n,m)$ is an extension of $V^{a}(\mathfrak{a}) \otimes \cW$, where $\cW \cong \cC^{\psi}_{XY}(n,m)$ is a simple, $1$-parameter quotient of $\cW^{\mathfrak{so}_2}_{\infty}$. 
    We will use the same notation for the generators of $\cW$; in particular, we have fields of strong generators $\{X^{2i},Y^{2i}|i\geq 1\}\cup\{H,L,W^i|i\geq 3\}$, which we refer to uniformly as $W^{n,\mu}$, see (\ref{eq:notation}). The extension is generated by fields in a fixed weight and parity (determined by $X$ and $Y$), which transform as $\mathbb{C}^2 \otimes \rho_{\mathfrak{a}}$ as a module for $\mathfrak{so}_2 \oplus \mathfrak{a}$, see Table (\ref{tab:Walgebras}); recall we denote these by $P^{\mu,j}$ for $\mu=\pm\frac{1}{2}$ and $j=1,\dots,m$, and normalize in accordance with (\ref{normalization}). The remaining OPEs $W^{n,\nu}(z)P^{\mu,j}(w)$ and $P^{\mu,i}(z)P^{\nu,j}(w)$ are assumed to be as general as possible, compatible with the affine and conformal symmetries, and non-degeneracy assumption (\ref{normalization}).

	The main result of this section is a {\it reconstruction theorem}, which states that the full OPE algebra of $\cW^{\psi}_{XY}(n,m)$ is determined by the structure of $\cW$ and the action of $\fr{so}_2 \oplus\fr{a}$ on the generating fields. 
	\begin{theorem}\label{thm:reconstruction}
		Let $\cA^{\psi}_{XY}(n,m)$ be a simple $1$-parameter vertex (super)algebra with the following properties, which are shared with $\cW^{\psi}_{XY}(n,m)$.
		
		\begin{enumerate}\label{features of extension}\label{assumptions}
			\item 
			$\cA^{\psi}_{XY}(n,m)$ is a conformal extension of $\cW\otimes V^a(\fr{a})$, where $\cW$ is some 1-parameter quotient of $\Wso$.
			
			\item The extension is generated by fields $P^{\mu,j}$ in weight $m$ which are primary for the conformal vector $L+L^{\fr{a}}$ of $\cA^{\psi}_{XY}(n,m)$, primary for the affine subVOA $V^k(\fr{so}_2)\otimes V^{a}(\fr{a})$, and transform as $\mathbb{C}^2\otimes \rho_{\fr{a}}$ under $\fr{so}_2\oplus\fr{a}$.
			
			\item The extension fields have the same parity as the corresponding fields of $\cW^{\psi}_{XY}(n,m)$ appearing in Table (\ref{tab:Walgebras}).
			Algebra $\cA^{\psi}_{XY}(n,m)$  is strongly generated by these fields, together with the generators of  $\cW\otimes V^a(\fr{a})$.
			\item The restriction of the Shapovalov form to the extension fields is non-degenerate.
		\end{enumerate}
		Then $\cA^{\psi}_{XY}(n,m)$ is isomorphic to $\cW^{\psi}_{XY}(n,m)$ as $1$-parameter vertex (super)algebras.
	\end{theorem}

\begin{proof}
    We would like to handle cases when the nilpotent is of type $B$ and $C$ in a uniform way. 
    In the above, $m$ is the rank of the Lie algebra in which the nilpotent element is principal, and it enters our procedure as a degree of the leading pole in (\ref{normalization}).
    For the purpose of this proof, we adopt the following convention. 
    Here, let $m$ be the conformal weight of the extension fields, which is either an integer or half-integer. 
	\begin{itemize}
        \item If $m=\frac{1}{2}$ then we are in the case of diagonal cosets discussed in Section \ref{sect:Diagonal}.
		\item If $m\geq 1$ is an integer, then nilpotent is principal in $\fr{so}_{2m-1}$.
		\item If $m\geq \frac{3}{2}$ is a half-integer, then nilpotent is principal in $\fr{sp}_{2m-2}$.
	\end{itemize}
	   In $V^{a}(\fr{a})\otimes \cW$ the total Virasoro field is $T= W^{2,0}+L^{\fr{so}_2}+L^{\fr{a}}$, so we have OPEs
		\begin{equation*}\label{LP}
			\begin{split}
				 W^{2,0}(z)P^{\mu,j}(w)\sim  \lambda  P^{\mu,j}(w)(z-w)^{-2}+{(\partial P^{\mu,j}-L_{(0)}^{\fr{so}_2}P^{\mu,j}- L_{(0)}^{\fr{a}}P^{\mu,j})(w)}(z-w)^{-1},
			\end{split}
		\end{equation*}
		where $\lambda$ is the constant 
		\begin{equation} \label{def:lambda}
		 \lambda=m- \frac{1}{2k}-\frac{\T{Cas}}{a+h^{\vee}_{\fr{a}}},\end{equation}
		and $\T{Cas}$ is the eigenvalue of the Casimir in $U(\fr{a})$ of the standard representation $\rho_{\fr{a}}$.
    Next, we proceed to set-up OPEs among the generators of $\cW$ and extension fields $P^{\mu,i}$.
	For our computation we need only a few structure constants, defined in the following OPEs, employing same notation as in, (\ref{OPEs refined})
    	\begin{equation}\label{ansatz extension}
		\begin{split}
			 W^{\mp2,2}\times P^{\pm\frac{1}{2},i}=  a P^{\mp\frac{1}{2},i}, \quad 
W^{3,0}\times P^{\pm\frac{1}{2},i}=  \pm v P^{\pm\frac{1}{2},i}(w)(z-w)^{-3}\pm v_{3,2}(:\!W^{2,\pm2}P^{\mp\frac{1}{2},i}\!:),
	\end{split}
\end{equation}
with remaining terms being fixed by the conformal and affine symmetries.
Now, imposing the Jacobi relations $J(W^{2,2},W^{2,-2},P^{\pm\frac{1}{2},i})$, yields us a system of quadratic equations. 
Solving it, we find
\[a=-\frac{3}{8} k^3 \lambda,\quad v= -\frac{3}{16} (k-2) (k-1) k^2 \lambda, \quad v_{3,2}=\frac{3 k (2 \lambda  k-k+1)}{2 k (k+4) \lambda +(k-4) (k-1)},\]
and crucially, the truncation relation
			\begin{equation}\label{trunk}
			\begin{split}
	c=\frac{4 k \lambda  \left((k-2) k \lambda -(k-1)^2\right)}{4 k \lambda +(k-2) (k-1)}.
			\end{split}
		\end{equation} where $\lambda$ is given by \eqref{def:lambda}.
Here, we have the following choice of scalings to make radicals clear.
\begin{equation}\label{eqn:scale reconstruction}
\omega_2=-\frac{3}{8} (k-2) (k-1) k (2 k \lambda -3 (k-1)),\quad \omega_3=\frac{1}{4} (2 k (k+4) \lambda +(k-4) (k-1)).
\end{equation}
Therefore, $\cW$ is determined by the extension data, which enters above through parameter $\lambda$, see (\ref{def:lambda}).

At this stage, the OPEs $W^{2,\pm2}(z)P^{\mu,i}(w)$ and $W^{3,0}(z)P^{\mu,i}(w)$ are expressed as rational functions of the $\fr{so}_2$-level $k$.
In fact, this property propagates to all OPEs $W^{n,\mu}(z)P^{\nu,i}(w)$ for $n\geq 4$, $\mu=0,\pm2$.
For example, imposing the Jacobi identity $J_{r,1}(W^{3,0},\tilde W^{n,\mu},P^{\nu,j})$ we obtain a relation
		\begin{equation*}\label{XPFull}
			\begin{split}
				\tilde W^{n+1,\mu}_{(r)}P^{\nu,i}=\frac{1}{rw^{W^3,\tilde W^{n,\mu}}_{\tilde W^{n+1,\mu}}-(r+1)w^{W^3,\tilde W^{n,\mu}}_{\partial \tilde W^{n+1,\mu}}} \left(W^{3,0}_{(r)}\tilde W^{n,\mu}_{(1)}P^{\nu,i}-\tilde W^{n,\mu}_{(1)}W^{3,0}_{(r)}P^{\nu,i}-R\right),
			\end{split}
		\end{equation*}
		where
		\begin{equation*}
			R=(W^{3,0;n,\mu}_0)_{(r+1)}P^{\nu,i}
			-\sum_{j=2}^r\binom{r}{j}(W^3_{(j)}\tilde W^{n,\mu})_{(r+1-j)}P^{\nu,i},
		\end{equation*}
inductively determining all OPEs $W^{n,\mu}(z)P^{\nu,i}(w)$.

It remains to see that OPEs $P^{\mu,i}(z)P^{\nu,j}(w)$ are fully determined from $W^{\mu,i}(z)P^{\nu,j}(w)$ and $W^{n,\mu}(z)W^{m,\nu}(w)$. To see this, consider the Jacobi identity $J_{1,r}(W^{3,0}, P^{-\frac{1}{2},i}, P^{\frac{1}{2},j})$, which reads as follows.
		\begin{equation}\label{extensionXP}
			W^{3,0}_{(1)}P^{-\frac{1}{2},i}_{(r)}P^{\frac{1}{2},j}   =
			(W^{3,0}_{(1)}P^{-\frac{1}{2},i})_{(r)}P^{\frac{1}{2},j} +(W^{3,0}_{(0)}P^{-\frac{1}{2},i})_{(r+1)}P^{\frac{1}{2},j}+P^{-\frac{1}{2},i}_{(r)}W^{3,0}_{(1)}P^{\frac{1}{2},j}.
		\end{equation}
		Note that the left side of relation (\ref{extensionXP}) depends only on $P^{-\frac{1}{2},i}_{(r)}P^{\frac{1}{2},j}$, while the right side gives rise to a contribution of the products $P^{-\frac{1}{2},i}_{(r-1)}P^{-\frac{1}{2},j}$;
		each such contribution arises from monomials in the products $W^{3,0}_{(1)}P^{-\frac{1}{2},i}, W^{3,0}_{(0)}P^{-\frac{1}{2},i}$ and $W^{3,0}_{(1)}P^{\frac{1}{2},i}$, corresponding to descendants in (\ref{ansatz extension}).
		Upon collecting every contribution of the desired product $P^{\frac{1}{2},i}_{(r-1)}P^{\frac{1}{2},j}$, we obtain a recursive formula expressing lower order poles in terms of the higher ones, with base case given by (\ref{normalization}).  This completes the proof. \end{proof}

	\begin{remark}
		The Jacobi identities $J_{1,0}(W^{3,0}, P^{\mu,i}, P^{\nu,j})$ gives rise to relations expressing strong generators $W^{n,\mu}$ for $n\geq m$, in terms of $\fr{a}$-invariants in the extension fields.
	\end{remark}

\section{$\cW^{\fr{so}_2}_{\infty}$ as an extension of two copies of $\cW^{\text{ev}}_{\infty}$} \label{sect:twocopies}
We recall the universal 2-parameter vertex algebra $\cW^{\rm{ev}}_{\infty}$ which was constructed and denoted by $\cW^{\text{ev}}(c,\lambda)$ in \cite{KL}. It is freely generated of type $\cW(2,4,6,\dots)$, and is weakly generated by a Virasoro field $L$ of central charge $c$, and a weight $4$ primary field $W^4$. The fields $W^{2i}$ for $i \geq 3$ are defined recursively by $W^{2i} = W^4_{(1)} W^{2i-2}$. The second parameter $\lambda$ arises via the following OPE relation

\begin{equation*}\label{eq:Wev}
	\begin{split}
		W^4_{\pm}(z)W^4_{\pm}(w)\sim&  \frac{c(5c+22)\omega_{\pm}}{840}(z-w)^{-8}+ \frac{ (5c+22)\omega_{\pm}}{105}L(w)(z-w)^{-6}+ \frac{(5c+22)\omega_{\pm}}{210}\partial L(w)(z-w)^{-5}\\
		&+ \Big(\frac{\lambda_{\pm}}{8} W^4_{\pm}+\frac{ \omega_{\pm}}{5}:\!LL\!:+\frac{(c-4)\omega_{\pm}}{140}\partial^2 L_{\pm}\Big)(w)(z-w)^{-4}\\
		&+ \Big(\frac{\lambda_{\pm}}{16} W^4_{\pm}+\frac{ \omega_{\pm}}{5}:\!\partial LL\!:+\frac{(c-4)\omega_{\pm}}{630}\partial^2 L_{\pm}\Big)(w)(z-w)^{-3}+W^6_{\pm}(w)(z-w)^{-2}\\
		&+\Big(\frac{1}{2}\partial W^{6}_{\pm} - \frac{\omega_{\pm}}{60}:\!\partial^3 L L\!:- \frac{\omega_{\pm}}{20}:\!\partial^2 L \partial L\!:- \frac{\lambda_{\pm}}{192}\partial^3 W^4_{\pm}- \frac{(c-4)\omega_{\pm}}{10080}\partial^5 L L\Big)(w)(z-w)^{-1}.
	\end{split}
\end{equation*}
Moreover, by \cite[Theorem 3.10]{KL}, if $\{L, W^{2i}|\ 2\leq i \leq 7\}$ are fields in some vertex algebra which satisfy the OPEs of $\cW^{\text{ev}}_{\infty}$, along some quotient $\mathbb{C}[c,\lambda]/I$ of $\mathbb{C}[c,\lambda]$, then $\{L, W^{2i}|\ 2\leq i \leq 7\}$ generate a homomorphic image of $\cW^{\text{ev}}_{\infty}$. Our goal in this section is to show that $\cW^{\fr{so}_2}_{\infty}$ is a conformal extension of two commuting copies of $\cW^{\text{ev}}_{\infty}$.

Recall the orthosymplectic $Y$-algebras of \cite{GR} were denoted by $\cC^{\psi}_{iZ}(a,b)$ in \cite{CL4}, where $i = 1,2$ and $Z = B,C,D,O$. They arise as $1$-parameter quotients of $\cW^{\rm{ev}}_{\infty}$ via the following procedure. There is an ideal $I_{iZ,a,b} \subseteq \mathbb{C}[c,\lambda]$ given in parametric form in terms of the parameter $\psi$ in \cite{CL4}. This generates a vertex algebra ideal $I_{iZ,a,b} \cdot  \cW^{\text{ev}}_{\infty}$, and
$\cC^{\psi}_{iZ}(a,b)$ is isomorphic to the simple graded quotient of $$\cW^{\text{ev}}_{\infty} / I_{iZ, a,b} \cdot \cW^{\text{ev}}_{\infty},$$ i.e., the quotient of $\cW^{\text{ev}}_{\infty}$ by its maximal ideal $\cI_{iZ,a,b}$ containing $I_{iZ,a,b}$. We recall the following examples from \cite{CL4}:
\begin{equation}
\begin{split}
\cC^{\psi}_{2D}(n,0) & \cong \text{Com}(V^{\ell}(\gs\go_{2n}), V^{\ell-1}(\gs\go_{2n}) \otimes \cF(2n))^{\mathbb{Z}_2},\qquad \qquad \ell = -2\psi-2n+3,
\\ \cC^{\psi}_{2B}(n,0) & \cong \text{Com}(V^{\ell}(\gs\go_{2n+1}), V^{\ell-1}(\gs\go_{2n+1}) \otimes \cF(2n+1))^{\mathbb{Z}_2},\qquad \ell = -2\psi-2n+2,
\\ \cC^{\psi}_{2C}(n,0) & \cong \text{Com}(V^{\ell}(\gs\gp_{2n}), V^{\ell+1/2}(\gs\gp_{2n}) \otimes \cS(n)),\qquad \qquad \ell = \psi-n-3/2,
\\ \cC^{\psi}_{2O}(n,0) & \cong \text{Com}(V^{\ell}(\go\gs\gp_{1|2n}), V^{\ell+1/2}(\go\gs\gp_{1|2n}) \otimes \cS(n) \otimes \cF(1))^{\mathbb{Z}_2},\qquad \ell = \psi-n-1,
\end{split}
\end{equation}

\begin{lemma} \label{twocopiesofW} For each $n \geq 1$, we have the following conformal embeddings.
\begin{equation} \label{twocopiesconstant}
\begin{split}
\cC^{\psi}_{2D}(n,0) \otimes \cC^{\psi-1}_{2D}(n,0) & \hookrightarrow \cC^{\ell}(2n),\qquad \ell = -2\psi-2n+3,
\\ \cC^{\psi}_{2B}(n,0) \otimes \cC^{\psi-1}_{2B}(n,0) & \hookrightarrow \cC^{\ell}(2n+1),\qquad \ell = -2\psi-2n+2,
\\ \cC^{\psi}_{2C}(n,0) \otimes \cC^{\psi-1}_{2C}(n,0) & \hookrightarrow \cC^{\ell}(-2n),\qquad \ell = \psi-n-3/2,
 \\ \cC^{\psi}_{2O}(n,0) \otimes \cC^{\psi-1}_{2O}(n,0) & \hookrightarrow \cC^{\ell}(-2n+1),\qquad \ell = \psi-n-1.
 \end{split}
\end{equation}
\end{lemma}

\begin{proof} We only prove the first statement since the proofs of the others are the same. Recall that $\cC^{\ell}(2n) = \text{Com}(V^{\ell}(\gs\go_{2n}), V^{\ell-2}(\gs\go_{2n}) \otimes \cF(4n))^{\mathbb{Z}_2}$, and write $$V^{\ell-2}(\gs\go_{2n}) \otimes \cF(4n) = (V^{\ell-2}(\gs\go_{2n}) \otimes \cF(2n)) \otimes \cF(2n).$$ Note that 
$(V^{\ell-2}(\gs\go_{2n}) \otimes \cF(2n))$ is a conformal extension of $$V^{\ell-1}(\gs\go_{2n}) \otimes \text{Com}(V^{\ell-1}(\gs\go_{2n}), V^{\ell-2}(\gs\go_{2n}) \otimes \cF(2n)) = V^{\ell-1}(\gs\go_{2n}) \otimes \cC^{\psi-1}_{2D}(n,0),\qquad \ell = -2\psi-2n+3.$$ So $V^{\ell-2}(\gs\go_{2n}) \otimes \cF(4n)$ is a conformal extension of 
\begin{equation}  V^{\ell-1}(\gs\go_{2n}) \otimes \cC^{\psi-1}_{2D}(n,0) \otimes \cF(2n) \cong  V^{\ell-1}(\gs\go_{2n})  \otimes \cF(2n) \otimes \cC^{\psi-1}_{2D}(n,0). \end{equation} Since $V^{\ell-1}(\gs\go_{2n})  \otimes \cF(2n)$ is a conformal extension of $V^{\ell}(\gs\go_{2n}) \otimes \cC^{\psi}_{2D}(n,0)$, it follows that $V^{\ell-2}(\gs\go_{2n}) \otimes \cF(4n)$ is a conformal extension of $V^{\ell}(\gs\go_{2n}) \otimes \cC^{\psi}_{2D}(n,0) \otimes \cC^{\psi-1}_{2D}(n,0)$, as claimed.
\end{proof}

We now consider the tensor product $\cW^{\rm{ev}}_{\infty} \otimes \cW^{\rm{ev}}_{\infty}$. Here the first copy of $\cW^{\rm{ev}}_{\infty}$ has parameters $c,\lambda$, and the second copy has parameters $c', \lambda'$, which we suppress from the notation. We use the notation $I_{iX, n,m}$ and $I'_{iX, n,m}$, for ideals in the first (respectively second) copy of $\cW^{\rm{ev}}_{\infty}$. We can rephrase the result of Lemma \ref{twocopiesofW} as follows: For each $n \in \mathbb{Z}$, there exists a homomorphism of vertex algebras
\begin{equation} \label{def:psin} \Psi_n: \cW^{\rm{ev}}_{\infty} \otimes \cW^{\rm{ev}}_{\infty} \rightarrow \cW^{\gs\go_2}_{\infty, I_{n}},\end{equation} 
which induces the conformal embeddings given by \eqref{twocopiesofW}.

Let us introduce a quadratic irrational expression
\begin{equation}\label{eq:extension}
    \omega=\sqrt{\frac{c-(k-1)}{c-(k-1)^3}} ,\quad c=\frac{( k-1) ( k \omega -\omega -1) (k \omega -\omega +1)}{(\omega -1) (\omega +1)}.
\end{equation}
Further, let us define a scaling parameter
\begin{equation*}
   \tilde\omega_3= -\frac{(k-2) (k-1) (k \omega +2 \omega -2) (k \omega +2 \omega +2)}{6 k^2 \omega ^2 \omega _2 \omega _3^2}.
\end{equation*}
This scaling will be useful in the explicit description of the two commuting copies of $\cW^{\text{ev}}_{\infty}$ in $\cW^{\gs\go_2}_{\infty}$ because all radicals arising during our computations will clear.

\begin{theorem} \label{completion:twocopies} After adjoining $\omega$ to the ring $\mathbb{C}[c,k]$ and then inverting finitely many polynomials which are given explicitly in Appendix \ref{app:wt4}, there exists a  vertex algebra homomorphism
\begin{equation} \label{2paramembedding} \Psi: \cW^{\text{ev}}_{\infty} \otimes \cW^{\text{ev}}_{\infty} \rightarrow \cW^{\fr{so}_2}_{\infty}\end{equation} such that for all $n \in \mathbb{N}$, composing this map with the quotient map $\cW^{\fr{so}_2}_{\infty} \rightarrow \cW^{\gs\go_{2}}_{\infty, I_n}$, we recover the conformal embeddings \eqref{twocopiesofW}. Here we have omitted from our notation the localization of the quadratic extension $\mathbb{C}[c,k][\omega]$. Moreover, the map $\Psi$ is unique up to the automorphism of our localized extension of $\cW^{\fr{so}_2}_{\infty}$ given by 
\begin{equation} \label{Galois} \omega \mapsto -\omega,\qquad L_{\pm} \mapsto L_{\mp},\qquad W^4_{\pm} \mapsto W^4_{\mp},\end{equation} which we regard as a Galois automorphism.
 \end{theorem}

\begin{proof} First, a direct computer computation shows that $\Wso$ contains two commuting Virasoro subalgebras generated by the fields $L_-$ and $L_+$ with central charges $c_-$ and $c_+$, respectively, so that 
\[L_{\pm}(z)L_{\mp}(w)\sim 0, \quad L_++L_-=L.\]
These are of the form
\begin{equation} \label{vir:splitting}
\begin{split}
    L_{\pm}=&\frac{1}{2}L\pm\frac{1}{2}\omega \left(\frac{k-1}{k}:\!HH\!:-\tilde\omega_3W^{2,-2}-\tilde\omega_3 W^{2,2}- W^{2,0}\right),\quad 	c_{\pm}=\frac{1}{2}c+\frac{1}{2}\pm (c-(k-1))\omega.
\end{split}
\end{equation}
Moreover, up to interchanging $L_-$ and $L_+$, this is the unique splitting of $L$ into two commuting nontrivial Virasoro fields. It is straightforward to check that the central charges of $L_+$ and $L_-$ agree with those of the commuting Virasoro fields in $\cC^{\ell}(n)$ given by \eqref{twocopiesofW}, for $n\in \mathbb{Z}$. Furthermore, there are two commuting weight four fields $W_{\pm}^4$ which satisfy 
\[L_{\pm}(z)W_{\pm}^4(w) \sim 4 W^4_{\pm}(w)(z-w)^{-2} + \partial W^4_{\pm}(w)(z-w)^{-1},\quad L_{\pm}(z)W_{\mp}^4(w)\sim 0.\]
These conditions determine $W_{\pm}^4$ up to scaling parameter $\omega_{\pm}$, and the explicit fields are given in Appendix \ref{app:wt4}. Next, for $r\geq 3$, we define the higher weight fields 
$$W^{2r}_{\pm}= (W_{\pm}^{4})_{(1)}W^{2r-2}_{\pm}.$$
Since $\{L_{\pm}, W^{2r}_{\pm}|\ r \geq 2\}$ both satisfy the OPE relations of $\cW^{\text{ev}}_{\infty}$ in the quotient $\cC^{\ell}(2n)$ for all $n \in \mathbb{N}$ and all $\ell$ such that $\ell-2$ is admissible, and the graded characters of $\cC^{\ell}(2n)$ and  $\cW^{\gs\go_2}_{\infty}$ agree up to arbitrarily high conformal weight for a suitable choice of $n$ and $\ell$, this forces these OPE relations to hold in the $2$-parameter vertex algebras $\cW^{\gs\go_2}_{\infty}$ as well. Finally, one can verify using the explicit forms of $L_{\pm}$ and $W^4_{\pm}$ that \eqref{Galois} is indeed an automorphism. \end{proof}

We need the explicit form of the fields $W^4_{\pm}$ in the case of $\cC^{\ell}(n) = \text{Com}(V^{\ell}(\gs\go_n), V^{\ell-2}(\gs\go_n) \otimes \cF(n))$ because we need to know that when $n$ is even and $\ell$ is rational, the formulas for $W^4_{\pm}$ are well-defined. In this case, $W^4_{\pm}$ can be expressed in the following form modulo terms that depend on the fields in weights $1,2,3$: 
$$W^4_{\pm} \equiv a_{\pm}(n,\ell) W^{4} \pm (X^4 + Y^4).$$ The coefficients $a_{\pm}(n,\ell)$ are given by the following explicit formulas:
\begin{equation}
    \begin{split}
        a_+(n,\ell) =&\frac{2 (\ell-3) (n-3) (n+2) (\ell+n-2) (\ell+2 n-3)Y(\ell,n)}{(\ell+n-3)X(\ell,n)},\\
        a_-(n,\ell) =&\frac{2 (\ell-3) (n-3) (n+2) (\ell+n-4) (\ell+2 n-3)Y(\ell,n)}{(\ell+n-3)X(\ell,n)},\\
   X(\ell,n)=\ &5 \ell^4 n^4+71 \ell^4 n^3-170 \ell^4 n^2-32 \ell^4
   n+384 \ell^4+20 \ell^3 n^5+224 \ell^3 n^4-1532 \ell^3 n^3+1912 \ell^3 n^2\\&+1920 \ell^3 n-4608 \ell^3+20 \ell^2 n^6+170 \ell^2
   n^5-3012 \ell^2 n^4+9236 \ell^2 n^3-4360 \ell^2 n^2-16896 \ell^2 n\\&+19968 \ell^2+12 \ell n^6-1652 \ell n^5+10304 \ell
   n^4-17456 \ell n^3-9792 \ell n^2+51456 \ell n-36864 \ell\\&-192 n^6+2464 n^5-8128 n^4+2688 n^3+29440 n^2-50176,\\
  Y(\ell,n)=\ &5 \ell^2 n^2+40 \ell^2 n-64 \ell^2+10 \ell n^3+50 \ell n^2-368 \ell
   n+384 \ell+24 n^3-272 n^2+704 n-512
   n+24576.
    \end{split}
\end{equation}
Note that all structure constants are rational functions whose denominators have no additional irreducible factors, so they are defined whenever $a_{\pm}(n,\ell)$ are defined.

\begin{lemma} \label{rationalellcaseofc2n} When $\ell$ is a rational number and $\ell \neq 3$, $2-n$, $4-n$, or $3-2n$,  $a_{\pm}(n,\ell)$ is well-defined and nonzero for any even positive integer $n$.
\end{lemma}
\begin{proof} We need to show that $X(\ell,n) = 0$ and $Y(\ell,n)=0$ both have no solutions for $\ell \in \mathbb{Q}$ and $n \in \mathbb{Z}_{\geq 1}$. First, $X(\ell,n) = 0$ has the following solutions:
\begin{equation} \begin{split} \ell & = 3-n \pm \sqrt{ \frac{p(n) \pm \sqrt{q(n)}}{384 - 32 n - 
    170 n^2 + 71 n^3 + 5 n^4}},
    \\  p(n) & = 384 + 672 n - 682 n^2 + 263 n^3 - 147 n^4 + 38 n^5 + 5 n^6
\\  q(n) & =28704768 - 21659648 n - 10852352 n^2 + 13628416 n^3 - 2849536 n^4 - 925824 n^5 + 493456 n^6 
\\ &- 127360 n^7 + 4456 n^8 - 1112 n^9 + 969 n^{10}.\end{split} \end{equation}
   If $\ell$ is rational for some even positive integer $n$, a necessary condition is that $q(n)$ is a perfect square. Reducing $q(n)$ modulo $16$, we get
   $$q(n) \equiv n^8 (8 - 8 n + 9 n^2)\ \text{mod}\ 16,$$ so if $q(n)$ is a perfect square, then $\bar{q}(n) = 8 - 8 n + 9 n^2$ must be a perfect square mod $16$. However, the perfect squares mod $16$ are only $0,1,4,9$, and for even $n$ we have $\bar{q}(n)\equiv 8\  \text{mod}\ 16$ if $n \equiv 0\ \text{mod}\ 4$ and $\bar{q}(n)\equiv 12\  \text{mod}\ 16$ if $n \equiv 2\ \text{mod}\ 4$. Therefore there are no solutions to $X(\ell,n) = 0$ for $\ell$ rational and $n$ an even positive integer. 

Next, $Y(\ell,n) = 0$ has the following solutions:
\begin{equation} \begin{split} \ell & = \frac{-192 + 184 n - 25 n^2 - 5 n^3 \pm \sqrt{g(n) h(n)}}{-64 + 40 n + 5 n^2},
\\ g(n) & = -64 + 40 n + 5 n^2,
\\ h(n) & = -25152 + 552 n + 13 n^2 - 14 n^3 + 5 n^4.\end{split}\end{equation}
     If $\ell$ is rational for some even positive integer $n$,  $h(n)$ and $g(n)$ must both be perfect squares. Reducing $h(n)$ mod $5$ shows that $h(n) \equiv 4 \ \text{mod}\ 5$ when $n\equiv  6\ \text{mod}\ 10$, and otherwise $h(n) \equiv 2 \ \text{mod}\ 5$ or $h(n) \equiv 3 \ \text{mod}\ 5$. Therefore $h(n)$ can only be a perfect square if $n \equiv 6\ \text{mod}\ 10$. 
     
     Next, letting $\bar{h}(k) = h(6 + 10 k)$, we see that $\bar{h}(k) \equiv k \ \text{mod}\ 3$, so $h(n)$ can be perfect square only if $n$ has the form $6+10(3k)$ or $6+10(3k+1)$. Finally, $g(6 + 10 (3 k)) = 4 (89 + 750 k + 1125 k^2)$ and $g(6 + 10 (3 k+1)) =4 (464 + 1500 k + 1125 k^2)$. Neither can be a perfect square for $k \in \mathbb{Z}_{\geq 0}$ because $g(6 + 10 (3 k)) \equiv 2\ \text{mod}\ 3$ and  $g(6 + 10 (3 k+1)) \equiv 2\ \text{mod}\ 3$. Therefore both $g(n)$ and $h(n)$ cannot simultaneously be perfect squares for any even positive integer $n$.
     \end{proof}

For later use, we record the following:
   \begin{corollary} For all $n \in \mathbb{N}$ and $d\geq 3$, we have the following form of the fields $W^{2d}_{\pm} \in \cC^{\ell}(2n)$:
     \begin{equation} \label{evenspinhigherweightsfirst} W^{2d}_{\pm} = a_{\pm}(n,\ell)^{d-3} (a_{\pm}(n,\ell)^2 + \mu_d)H^{2d} \pm a_{\pm}(n,\ell)^{d-2} (X^{2d} + Y^{2d}),\end{equation} for a positive rational number $\mu_d$. In particular, when $\ell$ is a rational number and $\ell \neq 3$, $4-n$, or $3-2n$, we have 
 \begin{equation}  \label{evenspinhigherweights} a_{d,\pm} W^{2d}_{\pm} = b_{d,\pm} H^{2d} \pm  (X^{2d} + Y^{2d}),\qquad a_{d,\pm} = \frac{1}{a_{\pm}(n,\ell)^{d-2}},\qquad b_{d,\pm} = \frac{a_{\pm}(n,\ell)^2 + \mu_d}{a_{\pm}(n,\ell)}.\end{equation} Here $a_{d,\pm}$ and $b_{d,\pm}$ are nonzero rational numbers $a_d,b_d$.
\end{corollary} 

\begin{proof} Equation \eqref{evenspinhigherweights} is immediate from the fact that 
$$(W^4_{\pm})_{(1)} W^{2d-2}_{\pm} = W^{2d}_{\pm},\quad H^4_{(1)} H^{2d-2} = H^{2d},\quad X^{4}_{(1)} Y^{2d-2} = \frac{16}{9} H^{2d} =  Y^{4}_{(1)} X^{2d-2},$$ modulo normally ordered polynomials in generators of lower weight. If $\ell$ is rational as above, we have already seen that $a_{\pm}(n,\ell)$ is a nonzero rational number, so the coefficient $b$ above is nonzero. Similarly, $a_{\pm}(n,\ell)^2 + \mu_d$ cannot vanish because both $a_{\pm}(n,\ell)^2$ and $\mu_d$ are positive rational numbers.
\end{proof}

Let \begin{equation} \label{defpsixynm} \Psi_{XY,n,m}: \cW^{\text{ev}}_{\infty} \otimes \cW^{\text{ev}}_{\infty} \rightarrow \cC^{\psi}_{XY}(n,m)\end{equation} denote the composition of $\Psi$ with the quotient map $\cW^{\fr{so}_2}_{\infty} \rightarrow \cW^{\gs\go_2}_{\infty, I_{XY,n,m}}$.

\begin{theorem} \label{pairsofcurves} For each of the vertex $\cC^{\psi}_{XY}(n,m)$ there is a pair of orthosymplectic $Y$-algebras $\cC^{\psi'}_{i_1 Z_1}(a_1, b_1)$ and $\cC^{\psi''}_{i_2 Z_2}(a_2, b_2)$ where $i_1, i_2 \in \{1,2\}$ and $Z_1, Z_2 \in \{B,C,D,O\}$, such that $\Psi_{XY,n,m}$ descends to a map of $1$-parameter vertex algebras
\begin{equation} \label{psiXYnmInducedtilde} \cC^{\psi'}_{i_1 Z_1}(a_1, b_1) \otimes \cC^{\psi''}_{i_2 Z_2}(a_2, b_2)  \hookrightarrow \cC^{\psi}_{XY}(n,m).\end{equation}
They are listed as follows:

\begin{equation} \begin{split} 
\cC^{\psi}_{1B}(n+m,m) \otimes \cC^{\psi-1}_{1D}(n,m) & \hookrightarrow \cC^{\psi}_{BD}(n,m),
\\ \cC^{\psi}_{1D}(n+m+1, m) \otimes \cC^{\psi-1}_{1B}(n,m) & \hookrightarrow \cC^{\psi}_{BB}(n,m),
\\ \cC^{\psi}_{2C}(n+m,m) \otimes \cC^{\psi-1/2}_{2C}(n,m) & \hookrightarrow \cC^{\psi}_{CC}(n,m),
\\ \cC^{\psi}_{2O}(n+m,m) \otimes \cC^{\psi-1/2}_{2O}(n,m) & \hookrightarrow \cC^{\psi}_{CO}(n,m),
\end{split} \end{equation}

\begin{equation} \begin{split} 
\cC^{\psi}_{1B}(m-n,m) \otimes \cC^{\psi-1}_{1C}(n,m) & \hookrightarrow \cC^{\psi}_{BC}(n,m),\qquad m \geq n,
\\ \cC^{\psi}_{1O}(n-m,m) \otimes \cC^{\psi-1}_{1C}(n,m) & \hookrightarrow \cC^{\psi}_{BC}(n,m),\qquad m \leq n,
\end{split} \end{equation}

\begin{equation} \begin{split} 
\\ \cC^{\psi}_{1D}(m-n+1,m) \otimes \cC^{\psi-1}_{1O)}(n,m) & \hookrightarrow \cC^{\psi}_{BO}(n,m),\qquad m \geq n-1,
\\ \cC^{\psi}_{1C}(n-m-1,m) \otimes \cC^{\psi-1}_{1O)}(n,m) & \hookrightarrow \cC^{\psi}_{BO}(n,m),\qquad m \leq n-1,
\end{split} \end{equation}

\begin{equation} \begin{split} 
\\ \cC^{\psi}_{2O}(m-n,m) \otimes \cC^{\psi-1/2}_{2B}(n,m) & \hookrightarrow \cC^{\psi}_{CB}(n,m),\qquad m\geq n,
\\ \cC^{\psi}_{2B}(n-m,m) \otimes \cC^{\psi-1/2}_{2B}(n,m) & \hookrightarrow \cC^{\psi}_{CB}(n,m),\qquad m \leq n,
\end{split} \end{equation}

\begin{equation} \begin{split} 
\\ \cC^{\psi}_{2C}(m-n,m) \otimes \cC^{\psi-1/2}_{2D}(n,m) & \hookrightarrow \cC^{\psi}_{CD}(n,m),\qquad m\geq n,
\\ \cC^{\psi}_{2D}(n-m,m) \otimes \cC^{\psi-1/2}_{2D}(n,m) & \hookrightarrow \cC^{\psi}_{CD}(n,m),\qquad m\leq n.
\end{split} \end{equation}
\end{theorem}

\begin{proof} By using the data in Appendix B of \cite{CL4} on intersections between the truncation curves for $\cC^{\psi}_{iZ}(a,b)$ and $\cC^{\psi}_{2D}(n,0)$ (as well as the isomorphism  $\cC^{\psi}_{2D}(n,0)\cong \cC^{\psi'}_{1O}(0,n-1)$ where $\frac{1}{2\psi} + \frac{1}{\psi'} = 1$), we can check that for each of the algebras $\cW^{\psi}_{XY}(n,m)$, there are infinitely many values of $k \in \mathbb{N}$ such that the truncation curves for $\cC^{\psi}_{2D}(n,0)$ and $\cC^{\psi-1}_{2D}(n,0)$ are coincident with the truncation curves for the pair of algebras $\cC^{\psi'}_{i_1 Z_1}(a_1,b_1)$ and $\cC^{\psi''}_{i_2 Z_2}(a_2,b_2)$ on the list above. This proves that the kernel of the map $\Psi_{XY,n,m}$ contains the ideal $I_{i_1Z_1, a_1, b_1} \otimes I_{i_2Z_2, a_2, b_2}$ given above. This shows that we get a map of appropriate truncation curves, so at worst we get embeddings
$$ \overline{\cC^{\psi'}_{i_1 Z_1}(a_1, b_1)} \otimes \overline{\cC^{\psi''}_{i_2 Z_2}(a_2, b_2)}  \hookrightarrow \cC^{\psi}_{XY}(n,m),$$ where $\overline{\cC^{\psi'}_{i_1 Z_1}(a_1, b_1)}$ and $\overline{\cC^{\psi''}_{i_2 Z_2}(a_2, b_2)}$ are possibly non-simple quotients of $\cW^{\text{ev}}_{\infty}$ along the ideals $I_{i_1Z_1, a_1, b_1}$ and $ I_{i_2Z_2, a_2, b_2}$, respectively.

It remains to prove that the image of $\Psi_{XY,n,m}$ is in fact the simple quotient $\cC^{\psi'}_{i_1 Z_1}(a_1, b_1)\otimes \cC^{\psi''}_{i_2 Z_2}(a_2, b_2)$. The argument is the same in all cases, so we only give the proof for $\cC^{\psi}_{BD}(n,m)$. We have an intersection point on the truncation curves for $\cC^{\psi}_{BD}(n,m)$ and $\cC^{\ell}(2r)$ at the point
$$\psi = \frac{2 (2 m + n + r)}{1 + 2 m},\qquad \ell = -(2r-2) + \frac{4 (m - m n + r + m r)}{1 + 2 m}.$$ This corresponds to the point
\begin{equation} \label{pointinparam} k = 2r,\qquad c = -\frac{(-1 + 2 r) (-2 m - n + 2 m r) (-1 + n + 2 r + 2 m r)}{(-1 + n + r) (2 m + n + r)},\end{equation} in the parameter space for $\cW^{\fr{so}_2}_{\infty}$, which is on the truncation curve for $\cC^{\psi}_{BD}(n,m)$.

Note that $\ell-2$ is admissible for infinitely many values of $r$, and whenever $\ell-2$ is admissible, $\cC_{\ell}(2r)$ is a conformal extension of the product of simple vertex algebra
$\cW_{\ell_1}(\gs\go_{2r})^{\mathbb{Z}_2} \otimes \cW_{\ell_2}(\gs\go_{2r})^{\mathbb{Z}_2}$ where
$$ \ell_1 = -(2r-2) + \frac{1 - 2 m + 4 m n - 4 r - 4 m r}{4 (-m + m n - r - m r)},\quad \ell_2 = -(2r-2) + \frac{2 (1 + 2 m n - 2 r - 2 m r)}{1 - 2 m + 4 m n - 4 r - 4 m r}.$$
If $\cC_{\ell}(2r)$ has the weak generation property, since the image $\Psi_{BD,n,m}(\cW^{\text{ev}}_{\infty} \otimes \cW^{\text{ev}}_{\infty}) \cong \cW_{\ell_1}(\gs\go_{2r})^{\mathbb{Z}_2} \otimes \cW_{\ell_2}(\gs\go_{2r})^{\mathbb{Z}_2}$, it follows that the maximal ideals $\cI_{1B,n+m,m}$ in the first copy of $\cW^{\text{ev}}_{\infty}$ and  $\cI_{1D,n,m}$ in the second copy $\cW^{\text{ev}}_{\infty}$, vanish at the point \eqref{pointinparam}.

Suppose that $\cC_{\ell}(2r)$ does not have the weak generation property for some $\ell$ as above. This can only happen if the subalgebra $\tilde{\cC}_{\ell}(2r) \subseteq \cC_{\ell}(2r)$ generated by the fields in weight at most $2$, is proper, and therefore need not be simple. However, $\tilde{\cC}_{\ell}(2r)$ maps surjectively to the simple quotient $\cW^{\gs\go_2}_{\infty, \cI_{2r,c}}$ of $\cW^{\fr{so}_2}_{\infty}$ along the maximal ideal $\cI_{2r,c}$ corresponding to the point \eqref{pointinparam}. Since $\Psi_{BD,n,m}(\cW^{\text{ev}}_{\infty} \otimes \cW^{\text{ev}}_{\infty})$ is already simple in $\tilde{\cC}_{\ell}(2r)$, it must be simple in $\cW^{\gs\go_2}_{\infty, \cI_{2r,c}}$. In particular, the maximal ideals $\cI_{1B,n+m,m}$ in the first copy of $\cW^{\text{ev}}_{\infty}$ and $\cI_{1D,n,m}$ in the second copy $\cW^{\text{ev}}_{\infty}$, and also vanish in $\cW^{\gs\go_2}_{\infty, \cI_{2r,c}}$. Since these two maximal ideals vanish at infinitely many points along the truncation curve for $\cC^{\psi}_{BD}(n,m)$, it follows that $\overline{\cC^{\psi'}_{i_1 Z_1}(a_1, b_1)}$ and $\overline{\cC^{\psi''}_{i_2 Z_2}(a_2, b_2)}$ must be simple.
\end{proof}

Next, we return to the set-up of Theorem \ref{thm:reconstruction}, where we consider vertex algebras $\cA^{\psi}_{XY}(n,m)$ which are extensions of $V^a(\mathfrak{a}) \otimes \cW$, where $\cW$ is some $1$-parameter quotient of $\cW^{\fr{so}_2}_{\infty}$ and the extension is generated by fields $P^{\mu,i}$ of fixed weight and parity, and action of $\gs\go_2 \oplus \mathfrak{a}$. As above, $\cW$ is a conformal extension of two $1$-parameter quotients of $\cW^{\text{ev}}_{\infty}$, which we denote by $\cW_+$ and $\cW_-$, with generating sets $\{L_+, W^{2r}_+\}$ and $\{L_-, W^{2r}_-\}$, respectively.

    \begin{theorem} The extension fields $P^{\mu,i}$ commute with $\cW_+$, and the generators of $V^a(\gs) \otimes \cW_-$ together with $\{P^{\mu,i}\}$ close under OPE, and therefore generate a hook-type $\cW$-(super)algebra. 
    \end{theorem}

  \begin{proof}
  First, using the Jacobi identities $J(L_+,L_-,P)$, we obtain
    \begin{equation}\label{eq:1vir}
        L_+(z)P(w)\sim 0,\quad L_-(z)P(w)\sim \lambda P(w)(z-w)^{-2}+\partial P(w)(z-w)^{-1}.
    \end{equation}
    Next, we us the Jacobi identities $J(L_+,W_-^4,P)$; specifically, it is enough to consider $J_{4,0}(L_+,W_-^4,P)$. 
    Using (\ref{eq:1vir}) we conclude that 
     \begin{equation}\label{eq:wt4}
        W_+^4(z)P(w)\sim 0,\quad W_-^4(z)P(w)\sim \lambda P(w)(z-w)^{-4}+\dotsb.
    \end{equation}
    Since $\cW_+$ is weakly generated by $L_+$ and $W^{4}_+$, relations (\ref{eq:1vir}), and (\ref{eq:wt4}), imply that $\cW_+$ commutes with the extension fields $P$. 
    Now we are in position to apply the reconstruction \cite[Theorem 6.3]{CL4} for hook-type $\cW$-(super)algebras in orthosymplectic types, thus concluding that there exists a a hook-type subVOA.
    \end{proof}

\begin{corollary}
 \label{pairsofcurvesrefined} The conformal embeddings $\cC^{\psi'}_{i_1 Z_1}(a_1, b_1) \otimes \cC^{\psi''}_{i_2 Z_2}(a_2, b_2)  \hookrightarrow \cC^{\psi}_{XY}(n,m)$ given by Theorem \ref{pairsofcurves} extend to conformal embeddings
\begin{equation} \label{psiXYnmInducedrefined } \cC^{\psi'}_{i_1 Z_1}(a_1, b_1) \otimes \cW^{\psi''}_{i_2 Z_2}(a_2, b_2)  \hookrightarrow \cW^{\psi}_{XY}(n,m).\end{equation}
They are listed as follows:

\begin{equation} \begin{split} 
\cC^{\psi}_{1B}(n+m,m) \otimes \cW^{\psi-1}_{1D}(n,m) & \hookrightarrow \cW^{\psi}_{BD}(n,m),
\\ \cC^{\psi}_{1D}(n+m+1, m) \otimes \cW^{\psi-1}_{1B}(n,m) & \hookrightarrow \cW^{\psi}_{BB}(n,m),
\\ \cC^{\psi}_{2C}(n+m,m) \otimes \cW^{\psi-1/2}_{2C}(n,m) & \hookrightarrow \cW^{\psi}_{CC}(n,m),
\\ \cC^{\psi}_{2O}(n+m,m) \otimes \cW^{\psi-1/2}_{2O}(n,m) & \hookrightarrow \cW^{\psi}_{CO}(n,m),
\end{split} \end{equation}

\begin{equation} \begin{split} 
\cC^{\psi}_{1B}(m-n,m) \otimes \cW^{\psi-1}_{1C}(n,m) & \hookrightarrow \cW^{\psi}_{BC}(n,m),\qquad m \geq n,
\\ \cC^{\psi}_{1O}(n-m,m) \otimes \cW^{\psi-1}_{1C}(n,m) & \hookrightarrow \cW^{\psi}_{BC}(n,m),\qquad m \leq n,
\end{split} \end{equation}

\begin{equation} \begin{split} 
\\ \cC^{\psi}_{1D}(m-n+1,m) \otimes \cW^{\psi-1}_{1O)}(n,m) & \hookrightarrow \cW^{\psi}_{BO}(n,m),\qquad m \geq n-1,
\\ \cC^{\psi}_{1C}(n-m-1,m) \otimes \cW^{\psi-1}_{1O)}(n,m) & \hookrightarrow \cW^{\psi}_{BO}(n,m),\qquad m \leq n-1,
\end{split} \end{equation}

\begin{equation} \begin{split} 
\\ \cC^{\psi}_{2O}(m-n,m) \otimes \cW^{\psi-1/2}_{2B}(n,m) & \hookrightarrow \cW^{\psi}_{CB}(n,m),\qquad m\geq n,
\\ \cC^{\psi}_{2B}(n-m,m) \otimes \cW^{\psi-1/2}_{2B}(n,m) & \hookrightarrow \cW^{\psi}_{CB}(n,m),\qquad m \leq n,
\end{split} \end{equation}

\begin{equation} \begin{split} 
\\ \cC^{\psi}_{2C}(m-n,m) \otimes \cW^{\psi-1/2}_{2D}(n,m) & \hookrightarrow \cW^{\psi}_{CD}(n,m),\qquad m\geq n,
\\ \cC^{\psi}_{2D}(n-m,m) \otimes \cW^{\psi-1/2}_{2D}(n,m) & \hookrightarrow \cW^{\psi}_{CD}(n,m),\qquad m\leq n.
\end{split} \end{equation}

Similarly, the maps \eqref{twocopiesconstant} from Lemma \ref{twocopiesofW} extend to conformal embeddings
\begin{equation} \label{twocopieskconstant} 
\begin{split}
\cC^{\psi}_{2D}(n,0) \otimes \cW^{\psi-1}_{2D}(n,0) & \hookrightarrow V^{\ell-2}(\gs\go_{2n}) \otimes \cF(4n), \qquad \ell = -2\psi-2n+3,
\\ \cC^{\psi}_{2B}(n,0) \otimes \cW^{\psi-1}_{2B}(n,0) & \hookrightarrow V^{\ell-2}(\gs\go_{2n+1}) \otimes \cF(4n+1),\qquad \ell = -2\psi-2n+2,
\\ \cC^{\psi}_{2C}(n,0) \otimes \cW^{\psi-1}_{2C}(n,0) & \hookrightarrow V^{\ell+1}(\gs\gp_{2n}) \otimes \cS(n),\qquad \ell = \psi-n-3/2,
 \\ \cC^{\psi}_{2O}(n,0) \otimes \cW^{\psi-1}_{2O}(n,0) & \hookrightarrow V^{\ell+1}(\go\gs\gp_{1|2n}) \otimes \cS(n) \otimes \cF(2),\qquad \ell = \psi-n-1.
 \end{split}
\end{equation}
\end{corollary}

\section{Rationality results} \label{sect:rationality}
In this section, we show that $\cW^{\mathfrak{so}_2}_{\infty}$ has many quotients which are strongly rational, and we use this to prove some new strong rationality results for $\cW$-superalgebras.

\begin{lemma} \label{rationalfamily} For $n \in \mathbb{Z}_{\geq 1}$, and $\ell-2$ an admissible level for $\gs\go_{2n}$, $\cC_{\ell}(2n) = \text{Com}(L_{\ell}(\gs\go_{2n}), L_{\ell - 2}(\gs\go_{2n}) \otimes \cF(4n))^{\mathbb{Z}_2}$ is strongly rational.
\end{lemma}

\begin{proof} If $\ell-2$ an admissible level for $\gs\go_{2n}$, so is $\ell-1$, and we have embeddings of simple vertex algebras 
$$L_{\ell-1}(\gs\go_{2n}) \hookrightarrow L_{\ell-2}(\gs\go_{2n}) \otimes \cF(2n),\qquad L_{\ell}(\gs\go_{2n}) \hookrightarrow L_{\ell-1}(\gs\go_{2n}) \otimes \cF(2n),
$$ and isomorphisms \cite{ACL}
\begin{equation}
\begin{split} &
\text{Com}(L_{\ell-1}(\gs\go_{2n}), L_{\ell-2}(\gs\go_{2n}) \otimes \cF(2n))^{\mathbb{Z}_2} \cong \cW_{k}(\gs\go_{2n})^{\mathbb{Z}_2},\qquad   \frac{1}{\ell + 2n-3} + \frac{1}{k+2n-2} =1 ,
\\ & \text{Com}(L_{\ell}(\gs\go_{2n}), L_{\ell-1}(\gs\go_{2n}) \otimes \cF(2n))^{\mathbb{Z}_2} \cong \cW_{m}(\gs\go_{2n})^{\mathbb{Z}_2},\qquad  \frac{1}{\ell + 2n-2} + \frac{1}{m+2n-2} =1. 
\end{split}
\end{equation}
By the same argument as the proof of Lemma \ref{twocopiesofW}, $L_{\ell-2}(\gs\go_{2n}) \otimes \cF(4n)$ is a conformal extension of 
$L_{\ell}(\gs\go_{2n}) \otimes \cW_{m}(\gs\go_{2n})^{\mathbb{Z}_2}  \otimes  \cW_{k}(\gs\go_{2n})^{\mathbb{Z}_2}$, so $\cC_{\ell}(2n)$ is a conformal extension of $\cW_{m}(\gs\go_{2n})^{\mathbb{Z}_2}  \otimes  \cW_{k}(\gs\go_{2n})^{\mathbb{Z}_2}$. Since $\cW_{m}(\gs\go_{2n})$ and $\cW_{k}(\gs\go_{2n})^{\mathbb{Z}_2}$ are both strongly rational \cite{Ar1, Ar2}, so are their $\mathbb{Z}_2$-orbifolds \cite{Mi, McR2}.  Therefore $\cC_{\ell}(2n)$ is a conformal extension of a strongly rational vertex algebra, and hence is strongly rational \cite{CMSY}.
\end{proof}

Although $\cC^{\ell}(2n)$ is a $1$-parameter quotient of $\cW^{\fr{so}_2}_{\infty}$, at a particular value of $\ell$ it is possible for the image of $\cW^{\fr{so}_2}_{\infty}$ in $\cC_{\ell}(2n)$, which is the subalgebra $\tilde{\cC}_{\ell}(2n)$ generated by the field $\{H, X^2, Y^2, H^2\}$, to be a proper subalgebra. Moreover, $\tilde{\cC}_{\ell}(2n)$ need not be simple even though $\cC_{\ell}(2n)$ is simple. For $\ell -2$ an admissible level, let $M_{\ell} \subseteq \mathbb{C}[c,t]$ be the maximal ideal generated by $ c - \frac{(-2 + t) (-1 + 2 n) (-4 + t+ 4 n)}{(-4 + t + 2 n) (-2 + t + 2 n)}$ and $t - \ell$, so that the simple quotient $\cW^{\gs\go_2}_{\infty,M_{\ell}}$ of $\cW^{\fr{so}_2}_{\infty} / M_{\ell} \cdot \cW^{\fr{so}_2}_{\infty}$, coincides with the simple quotient of $\tilde{\cC}_{\ell}(2n)$. We need a criterion for when $\cW^{\gs\go_2}_{\infty,M_{\ell}}$ is strongly rational. 

Recall that for all $k$, $\cW^{k}(\gs\go_{2n})$ is freely generated of type $\cW(2,4,\dots, 2n-2, n)$, and has a $\mathbb{Z}_2$-action that fixes the generators in weights $2,4,\dots, 2n-2$ and acts by $-1$ on the field $P$ of weight $n$. Then $\cW^{k}(\gs\go_{2n})^{\mathbb{Z}_2}$ is strongly generated by the generators in weights $2,4,\dots, 2n-2$, together with the composite fields
$$\omega^{2n+2d} =\ : \partial^{2d} P P:,\qquad d \geq 0.$$ In fact, for generic $k$ is suffices to take $0\leq d \leq n$, but we shall not need this fact.

Recall that $\cW^{k}(\gs\go_{2n})^{\mathbb{Z}_2}$ is a $1$-parameter quotient of $\cW^{\text{ev}}_{\infty}$, so that for generic values of $k$, it is weakly generated by the fields in weights $2$ and $4$. We will take the weight $2$ field to be the Virasoro field $L$, the weight $4$ field to be primary and normalized as in \cite{KL}, and we take $\{W^{2i}|\ i= 3,\dots, n-1\}$ to be defined recursively by $W^{2i} = W^4_{(1)} W^{2i-2}$, as in \cite{KL}. In the proof of Theorem 3.3 of \cite{CKoL}, it was stated erroneously that $\cW^{k}(\gs\go_{2n})^{\mathbb{Z}_2}$ is weakly generated by $\{L, W^4\}$ for all noncritical levels $k$. We correct this statement as follows.

\begin{theorem} \label{weakgenWso2nfirst} $\cW^{k}(\gs\go_{2n})^{\mathbb{Z}_2}$ is weakly generated by $\{L, W^4\}$ for all noncritical $k$ except for the following:
\begin{enumerate}
\item $k = -(2n-2) + \frac{d}{d-1}$ and $k = -(2n-2) + \frac{d-1}{d}$ for integers $d \geq 2$,
\item The two roots of $2 k^2 (2 n^2- 2 n-3) + k (16 n^3- 40 n^2+29) + 
 2 (8 n^4- 32 n^3 + 30 n^2 + 11 n -20)$.
\end{enumerate}
Therefore at these values of $k$, the simple quotient $\cW_{k}(\gs\go_{2n})^{\mathbb{Z}_2}$ is weakly generated by $\{L, W^4\}$ as well.
\end{theorem}

\begin{proof}
The Poisson bracket on Zhu's commutative algebra $R_{\cW^{k}(\gs\go_{2n})}$ given by $\{\bar{a},\bar{b}\} = \overline{a_{(0)} b}$ is easily seen to be trivial. The refined Poisson bracket given by  $\{\bar{a},\bar{b}\} = \overline{a_{(1)} b}$ is nontrivial, and by the same argument as Proposition A.3 of \cite{ALY}, it is independent of the level $k$. Moreover, all the fields $W^{2i}$ for $2 < i \leq n-1$ can be generated by $L, W^4$ for noncritical values of $k$. For $r \geq n$, we may write $$W^{2r} = \lambda_{2r} \omega^{2r} + \cdots,$$ where the remaining terms depend only on $L, W^4,\dots, W^{2s}$ for $s < r$, and $\lambda_{2r}$ is a rational function of $k$. Since $\omega^{2n} = \ :PP:$ which is nontrivial in $R_{\cW^{k}(\gs\go_{2n})}$, and the above Poisson structure on $R_{\cW^{k}(\gs\go_{2n})}$ is independent of $k$, $\lambda_{2n}$ must be independent of $k$. Also, since $\cW^{k}(\gs\go_{2n})^{\mathbb{Z}_2}$ is generated by $L, W^4$ for generic values of $k$, $\lambda_{2n}$ must be nonzero. For $r > n$ it is no longer the case that $\lambda_{2r}$ is independent of $k$. However, using the OPE $W^4(z) P(w)$ which is completely determined in \cite{CL4}, for $d \geq 2$ we can compute

\begin{equation} \begin{split} & W^4_{(1)} \omega^{2d-2+2n}  = \frac{16 (d + n) (k + 2 n-2) (1 - 6 d + 2 d k + 4 d n) (2 - 6 d - k + 
     2 d k - 2 n + 4 d n)}{21 d (2 d-1) (n-1) (2 - k - 8 n + 
     2 k n + 4 n^2)f(n,k) g(n,k)} \omega^{2d+2n} + \dots,
     \\ & f(n,k) =  1 - 6 n + 2 k n + 4 n^2,
     \\ & g(n,k) =  2 k^2 (2 n^2- 2 n-3) + k (16 n^3- 40 n^2+29) + 
 2 (8 n^4- 32 n^3 + 30 n^2 + 11 n -20).
     \end{split} \end{equation}
In particular, this shows that $\lambda_{2d+2n} \neq 0$ except for the levels when the above rational functions vanish or are undefined. This completes the proof.
\end{proof}

\begin{corollary} \label{weakgenWso2n} Suppose that $\ell-2$ is an admissible level for $\gs\go_{2n}$ and $\ell = -2\psi -2n+3$. If $(2n-2) + \ell \notin \mathbb{Z}_{\geq 1}$, $\cC_{\psi,2D}(n,0) = \text{Com}(L_{\ell}(\gs\go_{2n}), L_{\ell-1}(\gs\go_{2n}) \otimes \cF(2n))^{\mathbb{Z}_2}$ is weakly generated by the fields $L, W^4$. \end{corollary} 

\begin{proof} This is apparent because the isomorphism 
$$\cW^{k}(\gs\go_{2n})^{\mathbb{Z}_2} \cong \cC^{\psi}_{2D}(n,0),\qquad \ell = -2\psi -2n+3, \qquad \frac{1}{k+2n-2} + \frac{1}{\ell+2n-2} =1,$$ holds generically, and the corresponding isomorphism of simple quotients holds whenever $\cW_{k}(\gs\go_{2n})^{\mathbb{Z}_2}$ and $\cC_{\psi,2D}(n,0)$ are both quotients of the specialization of the $1$-parameter algebras. This is always the case for $\cW_{k}(\gs\go_{2n})^{\mathbb{Z}_2}$ when $k$ is noncritical, and by \cite[Thm. 8.1]{CL2}, it holds for $\cC_{\psi,2D}(n,0)$ whenever $\ell +2n-2$ is not a negative rational number.
\end{proof}

\begin{corollary} \label{cor:rationalityatell} For all admissible levels $\ell-2$ for $\gs\go_{2n}$ such that $(2n-2) + \ell \notin \mathbb{Z}_{\geq 1}$, $\cW^{\gs\go_2}_{\infty, M_{\ell}}$, which is isomorphic to the simple quotient of $\tilde{\cC}_{\ell}(2n)$, is strongly rational.
\end{corollary}

\begin{proof} Recall that $\cC_{\ell}(2n)$ is a conformal extension of  \begin{equation} \label{rationalsub} \text{Com}(L_{\ell}(\gs\go_{2n}), L_{\ell-1}(\gs\go_{2n}) \otimes \cF(2n))^{\mathbb{Z}_2} \otimes \text{Com}(L_{\ell-1}(\gs\go_{2n}), L_{\ell-2}(\gs\go_{2n}) \otimes \cF(2n))^{\mathbb{Z}_2},\end{equation} We will use the notation $\{L_+, W^{2d}_+| \ d\geq 2\}$ and $\{L_-, W^{2d}_-|\ d\geq 2\}$ to denote the images of the generators of $\cW^{\text{ev}}_{\infty}$ in $\text{Com}(L_{\ell}(\gs\go_{2n}), L_{\ell-1}(\gs\go_{2n}) \otimes \cF(2n))^{\mathbb{Z}_2}$ and $\text{Com}(L_{\ell-1}(\gs\go_{2n}), L_{\ell-2}(\gs\go_{2n}) \otimes \cF(2n))^{\mathbb{Z}_2}$, respectively.

Suppose that $\ell-2$ is admissible and $(2n-2) + \ell \notin \mathbb{Z}_{\geq 1}$. By Theorem \ref{completion:twocopies} and Lemma \ref{rationalellcaseofc2n}, the generators $L_{\pm}, W^4_{\pm}$ for these subalgebras lie in the subalgebra generated by $\{H, X^2, Y^2, H^2\}$, and by Corollary \ref{weakgenWso2n}, $\{L_{+}, W^4_{+}\}$ and $\{L_{-}, W^4_{-}\}$ weakly generate the subalgebras $\text{Com}(L_{\ell}(\gs\go_{2n}), L_{\ell-1}(\gs\go_{2n}) \otimes \cF(2n))^{\mathbb{Z}_2}$ and $\text{Com}(L_{\ell-1}(\gs\go_{2n}), L_{\ell-2}(\gs\go_{2n}) \otimes \cF(2n))^{\mathbb{Z}_2}$ of $\cC_{\ell}(2n)$, respectively. It follows that $\tilde{\cC}_{\ell}(2n)$ is a conformal extension of \eqref{rationalsub}, which is simple and strongly rational. Even though $\tilde{\cC}_{\ell}(2n)$ need not be simple, its simple quotient is still a conformal extension of \eqref{rationalsub}, and hence is strongly rational.
\end{proof}

\subsection{Some new rational $\cW$-superalgebras} We consider $\cC_{\psi,CO}(0,m) = \cW_{\psi-h^{\vee}}(\go\gs\gp_{1|4m}, f_{2m,2m})^{\mathbb{Z}_2}$, for $h^{\vee} = 2m+\frac{1}{2}$.

\begin{theorem} \label{Rationality:COcase} Suppose that $m,2n-1$ are coprime, $m \geq n-1$, and $\psi = \frac{4m+2n-1}{4m}$. Then $\cC_{\psi,CO}(0,m)$ is strongly rational.
\end{theorem}

\begin{proof}
There is an intersection point on the truncation curves for $\cC^{\psi}_{CO}(0,m)$ and $\cC^{\ell}(2n)$ when
$$\ell = -2n +2+\frac{4m+2n-1}{2m},\qquad \psi = \frac{4m+2n-1}{4m}.$$ 
Although $\cC^{\ell}(2n)$ need not coincide with the simple quotient $\cW^{\gs\go_2}_{\infty, M_{\ell}}$ at this point, Corollary \ref{cor:rationalityatell} shows that $\cW^{\gs\go_2}_{\infty, M_{\ell}}$ is strongly rational. Therefore it suffices to prove that $\cC_{\psi,CO}(0,m)$ is weakly generated by the fields in weights $1$ and $2$, since this would imply that $\cC_{\psi,CO}(0,m) \cong \cW^{\gs\go_2}_{\infty, M_{\ell}}$.

Recall that $\cW^{r}(\fr{osp}_{1|4m},f_{1|2m,2m})$ is freely generated of type $\cW\big(1, 2^3, 3, \dots, 2m-1,(2m)^3; \big(\frac{2m+1}{2}\big)^2\big)$, where the two odd fields of weight $\frac{2m+1}{2}$ are denoted by $P^{\pm}$. The orbifold $\cW^{r}(\fr{osp}_{1|4m},f_{1|2m,2m})^{\mathbb{Z}_2}$ is a $1$-parameter quotient of $\cW^{\fr{so}_2}_{\infty}$. By the same argument as the proof of Theorem \ref{weakgeneration:Walgebras}, as long as $r$ is noncritical and not in the (finite) set of points where some denominator of a structure constant of $\cW^{\gs\go_2}_{\infty}$ vanishes, the fields in weight at most $3$ weakly generate all generators of $\cW^{r}(\fr{osp}_{1|4m},f_{1|2m,2m})$ except for $P^{\pm}$, which have eigenvalue $-1$ with respect to the $\mathbb{Z}_2$-action. The orbifold $\cW^{r}(\go\gs\gp_{1|4m}, f_{2m,2m})^{\mathbb{Z}_2}$ is strongly generated by the fields in weights $1,2,\dots, 2m$ together with the following composite fields for $n\geq 0$:

\begin{equation} \label{COorbifoldgen}
\begin{split} & w^{2n+2m+1} =\ :\!\partial^{2n}P^{+}P^{-}\!: \ -\ :\!P^{+}\partial^{2n} P^{-}\!:, \qquad h^{2n+2m+2} =\  :\!\partial^{2n+1}P^{+}P^{-}\!:\ +\ :\!P^{+}\partial^{2n+1} P^{-}\!:,
\\ & x^{2n+2m+2} =\  :\!\partial^{2n+1}P^{+}P^{+}\!:, \qquad y^{2n+2m+2} = \ :\!\partial^{2n+1}P^{-}P^{-}\!:.\end{split}
\end{equation}
Let $W^{2n+2m+1}, X^{2n+2m+2}, Y^{2n+2m+2}, H^{2n+2m+2}$ be the images of the corresponding generators of $\cW^{\fr{so}_2}_{\infty}$ in $\cW^{r}(\fr{osp}_{1|4m},f_{1|2m,2m})^{\mathbb{Z}_2}$. 
We have
\begin{equation} \begin{split} & W^{2n+2m+1} = \lambda_{2n+2m+1} w^{2n+2m+1} + \cdots, \qquad H^{2n+2m+2} = \lambda_{2n+2m+2} h^{2n+2m+2} + \cdots,
\\ & X^{2n+2m+2} = \lambda_{2n+2m+2} x^{2n+2m+2} + \cdots,\qquad Y^{2n+2m+2} = \lambda_{2n+2m+2} y^{2n+2m+2} + \cdots,
\end{split}
\end{equation}
for some structure constants $\lambda_{2n+2m+1}, \lambda_{2n+2m+2}$, where the remaining terms depend only on the generators in weights $1,2,\dots, 2m$ and their derivatives.

Since the field $w^{2m+1}$ is nontrivial in Zhu's commutative algebra of $\cW^{r}(\fr{osp}_{1|4m},f_{1|2m,2m})$, and the Poisson structure is independent of $r$, the same argument in the proof of Theorem \ref{weakgeneration:Walgebras} that shows that $\lambda_{2m+1}$ is a nonzero constant for all $i < m$, we also get that $\lambda_{2m+1}$ is a nonzero constant. Therefore $w^{2m+1}$ is generated by the fields in weight at most $3$.

It is not true $\lambda_{2n+2m+3}, \lambda_{2n+2m+2}$ are independent of $\ell$ for $n \geq 0$. However, using the OPEs $W^3(z) P^{\pm}(w)$ and $X^2(z) P^{\pm}(w)$, we compute
\begin{equation}
\begin{split}
    W^3_{(1)}w^{2n+2m+1} =&\frac{9 (k-1) k^3 (m+n+1) (2 k n+k-2 n+2 m-1)}{16 (2 n+1) (k m+k+4 m-1)}h^{2n+2m+2}+\dotsb\\
    W^3_{(1)}h^{2n+2m+2} =& \frac{9 k^3 (k n+k-n-m-1) (2 k n+2 k m+3 k-2 n+10 m-3)}{32 (n+1) (k m+k+4 m-1)}w^{2n+2m+3}+\dotsb\\
     X^2_{(0)}w^{2n+2m+1} =&\frac{3 k^2 (2 k n+k-2 n+2 m-1)}{4 (2 n+1)}x^{2n+2m+2}+\dotsb,\\
     X^2_{(0)}y^{2n+2m+2}=&\frac{3 k^2 (k n+k-n-m-1)}{8 (n+1)}w^{2n+2m+3}+\dotsb.
\end{split}
\end{equation}
Here $k$ is the level of the Heisenberg field $H$. Again, by the same argument in the proof of Theorem \ref{weakgeneration:Walgebras}, $\lambda_{2n+2m+3}, \lambda_{2n+2m+2}$ for $n\geq 0$ are nonzero whenever the numerator and denominator in the above expressions do not vanish. This never happens for positive integer values of $k = 2n$, so all the fields \eqref{COorbifoldgen} in  $\cW^{r}(\fr{osp}_{1|4m},f_{1|2m,2m})^{\mathbb{Z}_2}$ are generated by the fields in weight at most $3$ for $r = \frac{k+4m-1}{4m}-h^{\vee}$ as above. Therefore the simple quotient $\cW_{r}(\fr{osp}_{1|4m},f_{1|2m,2m})^{\mathbb{Z}_2}$ has the weak generation property, and hence is isomorphic to $\cC_{\ell}(2m)$. Since $\cW_{r}(\fr{osp}_{1|4m},f_{1|2m,2m})$ is a simple current extension of $\cC_{\ell}(2m)$, it is also strongly rational.
\end{proof}

\subsection{Weak generation of $\cC_{\ell}(2n)$} We now address the more difficult question of when $\cC_{\ell}(2n)$ has the weak generation property and hence is a quotient of $\cW^{\fr{so}_2}_{\infty}$. We begin by proving a weak generation statement for the algebra
$$\cD_\ell(2n) = \text{Com}\left(L_{\ell-2}(\gs\go_{2n}), L_{\ell}(\gs\go_{2n}) \otimes \cF(4n)\right),$$ which is a simple current extension of $\cC_{\ell}(2n)$. By the same argument as the proof of Theorem \ref{twocopiesofW}, $\cD_\ell(2n)$ is a conformal extension of the tensor product
$$ \text{Com}(L_{\ell}(\gs\go_{2n}), L_{\ell-1}(\gs\go_{2n}) \otimes \cF(2n)) \otimes \text{Com}(L_{\ell-1}(\gs\go_{2n}), L_{\ell-2}(\gs\go_{2n}) \otimes \cF(2n)),$$ which is isomorphic to $\cW_{k}(\gs\go_{2n}) \otimes \cW_{m}(\gs\go_{2n})$ where $k,m$ satisfy 
\begin{equation} \label{condition:km} \frac{1}{\ell-1 + h^\vee} + \frac{1}{k+h^\vee} =1,\qquad \frac{1}{\ell + h^\vee} + \frac{1}{m+h^\vee} =1.\end{equation}

\begin{theorem} \label{weakgenerationD} 
Let $\ell-2$ be an admissible level for $\gs\go_{2n}$. Then $\cD_\ell(2n)$ is weakly generated by the weak generators of $\cW_{k}(\gs\go_{2n}) \otimes \cW_{k'}(\gs\go_{2n})$,
together with the weight one field $H$.
\end{theorem}
\begin{proof}
Recall the decomposition of $\cF(2n)$ as a module for $L_1(\gs\go_{2n})$, that is, $\cF(2n) \cong L_1(\gs\go_{2n}) \oplus  L_1(\omega_1)$, 
where $\omega_1$ is the first fundamental weight of $\gs\go_{2n}$. Also, note that the root lattice $Q$ is a sublattice of index two inside the lattice $\mathbb Z^n$, and that $\mathbb Z^n = Q \cup (Q+ \omega_1)$.
We will use \cite[Main Theorem 3]{ACL}, so we first introduce the notation. 
\begin{itemize}
    \item $L_k(\lambda)$ denotes the simple module of highest-weight $\lambda$ of $L_k(\gs\go_{2n})$.
    \item $P^k_+$ denotes the set of dominant weights $\lambda \in P_+$ of $\gs\go_{2n}$
    such that $L_k(\lambda)$ is a module of $L_k(\gs\go_{2n})$.
    \item $R^k_+ = P^+_k \cap \mathbb Z$.
    \item $\mathbf L_{k}(\chi_{\mu})$ denotes the simple module of the simple principal $\cW$-algebra of $\gs\go_{2n}$ at level $k$ of highest-weight $\chi_\mu$. Here $\chi_\mu$ is the central character associated to the highest-weight $\mu$.
    \item $h^\vee = 2n-2$ is the dual Coxeter number of $\gs\go_{2n}$.
\end{itemize}
Then \cite[Main Theorem 3]{ACL} for $\gg= \gs\go_{2n}$ and $\ell-2$ an admissible level reads
\begin{equation}\label{eq:levrel}
L_{\ell-2}(\mu) \otimes \cF(2n) \cong \bigoplus_{\lambda \in R^{\ell-1}_+}
L_{\ell-1}(\lambda) \otimes \mathbf L_{k}(\chi_{\mu - (k + h^\vee)\lambda}),\qquad   \frac{1}{\ell-1 + h^\vee} + \frac{1}{k+h^\vee} =1.
\end{equation}
Lemma 2.4 of \cite{CGL}, which is a reformulation of a main result of \cite{CKM2}, states the following:
Let $V, W$ be vertex operator algebras and $\cC$ a subcategory of a semisimple and rigid vertex tensor category of $V$-modules and
\[
A = \bigoplus_{X \in \text{Irr}(\cC)} X \otimes \tau(X)^*
\]
a vertex operator superalgebra that is a conformal extension of $V \otimes W$. Here $\tau$ is a map from $\text{Irr}(\cC)$ to modules of $W$.
Assume that $W$ is strongly rational.
Then $\cC$ is in fact closed under tensor product. 
Let $U \in \text{Irr}(\cC)$, Then $A$ is weakly generated by $V \otimes W$ together any field in $U$ if and only if the Grothendieck ring of $\cC$ is generated by $U$.
Moreover the main result of \cite{CKM2} in this setting is that there is a braid-reversed equivalence between $\cC$ and the category of $W$-modules
whose irreducibles appear in the decomposition of $A$. On objects this equivalence is given by the map $\tau$.

In our case $A = L_{\ell-2}(\gs\go_{2n}) \otimes \cF(2n)$, $V = L_{\ell-1}(\gs\go_{2n})$, $W = \cW_k(\gs\go_{2n})$. 
That the category of ordinary modules of $L_{\ell-1}(\gs\go_{2n})$ at admissible level $\ell-1$ is a vertex tensor category is \cite{CHY}, that it is semisimple is \cite{Ar1} and that it is rigid is \cite{C}. The strong rationality of $\cW_k(\gs\go_{2n})$ is \cite{Ar1, Ar2}.
 $L_{\ell-2}(\gs\go_{2n}) \otimes \cF(2n)$ is even strongly generated by the weight one fields of $L_{\ell-1}(\gs\go_{2n})$ together with the weight one fields of  $\cF(2n)$. Thus in particular the fusion ring of 
$L_{\ell-1}(\gs\go_{2n})$ whose simple objects are the $L_{\ell-1}(\lambda)$ with $\lambda \in R^{\ell-1}_+$ is generated by $L_{\ell-1}(\omega_1)$ and hence the fusion ring of $\cW_k(\gs\go_{2n})$ whose irreducible objects are the $\mathbf L_{k}(\chi_{- (k + h^\vee)\lambda})$ with $\lambda \in R^{\ell-1}_+$ is generated by $\mathbf L_{k}(\chi_{- (k + h^\vee)\omega_1})$.

We now apply \cite[Main Theorem 3]{ACL} twice:
\begin{equation}
    \begin{split}
        L_{\ell-2}(\gs\go_{2n}) \otimes \cF(4n) &\cong  L_{\ell-2}(\gs\go_{2n}) \otimes \cF(2n)\otimes \cF(2n) \\
&\cong         
     \bigoplus_{\lambda \in R^{\ell-1}_+}
L_{\ell-1}(\lambda) \otimes \cF(2n)\otimes \mathbf L_{k}(\chi_{\mu - (k + h^\vee)\lambda}) \\
&\cong \bigoplus_{\lambda \in R^{\ell-1}_+} \bigoplus_{\nu \in R^{\ell}_+}
L_{\ell}(\nu) \otimes \mathbf L_{m}(\chi_{\lambda - (m + h^\vee)\nu}) \otimes \mathbf L_{k}(\chi_{- (k + h^\vee)\lambda})
    \end{split}
\end{equation}
with $\frac{1}{m+h^\vee} + \frac{1}{\ell+h^\vee} =1$. We verify that this together with the relation between levels given by \eqref{eq:levrel}, implies that 
\[
\frac{1}{k+h^\vee} + (m+h^\vee) =2.
\]
Thus 
\[
\cD_\ell(2n) \cong \bigoplus_{\lambda \in R^{\ell-1}_+} 
 \mathbf L_{m}(\chi_{\lambda}) \otimes \mathbf L_{k}(\chi_{- (k + h^\vee)\lambda}).
\]
We can apply Lemma 2.4 of \cite{CGL} again, that is since 
the fusion ring of $\cW_k(\gs\go_{2n}, f_{\text{prin}})$ whose irreducible objects are the $\mathbf L_{k}(\chi_{- (k + h^\vee)\lambda})$ with $\lambda \in R^{\ell-1}_+$ is generated by $\mathbf L_{k}(\chi_{- (k + h^\vee)\omega_1})$ it follows that $\cD_\ell(2n)$ is weakly generated by $\cW_m(\gs\go_{2n}, f_{\text{prin}}) \otimes\cW_k(\gs\go_{2n})$ together with any element of $\mathbf L_{m}(\chi_{\omega_1}) \otimes \mathbf L_{k}(\chi_{- (k + h^\vee)\omega_1})$. 
We can choose this vector to be the top level vector. 
The conformal weight of the top level of $\mathbf L_{m}(\chi_{\lambda})$ is \cite[(30)]{ACL}
\[
h_\lambda = \frac{(\lambda+\rho)^2 - \rho^2}{2(m+h^\vee)} - \lambda \rho^\vee
\]
with $\rho, \rho^\vee$ Weyl vector and dual Weyl vector. For $\gg = \gs\go_{2n}$ these coincide and computing $\omega_1^2=1, \rho \omega_1 = n-1, 2\rho^2 = {n(n-1)}$ it follows that the top level of $\mathbf L_{m}(\chi_{\omega_1}) \otimes \mathbf L_{k}(\chi_{- (k + h^\vee)\omega_1})$
has indeed conformal weight one. 
\end{proof}

For the subalgebras $\cW_m(\gs\go_{2n})$ and $\cW_k(\gs\go_{2n})$, we denote the weight $n$ generator by $P^+$ and $P^-$, respectively. We use the same notation for the corresponding fields in $\cC^{\ell}(2n)$. Up to scaling, they are easily seen to have the form 
$$P^+ = \ :\phi^{1,1} \phi^{1,2} \cdots \phi^{1,2n}: + \cdots,\qquad P^- =  \ :\phi^{2,1} \phi^{2,2} \cdots \phi^{2,2n}: + \cdots,$$ where the remaining terms vanish in the limit $\ell \rightarrow \infty$. We have the following consequence of Theorem \ref{weakgenerationD}.

\begin{corollary} \label{first:orbifoldgen}  Suppose that $\ell-2$ is admissible for $\gs\go_{2n}$. Then  $\cC_{\ell}(2n)$ is generated as a module over $\cW_k(\gs\go_{2n})^{\mathbb{Z}_2} \otimes \cW_m(\gs\go_{2n})^{\mathbb{Z}_2}$ by the weight $1$ field $H$ together with the weight $2n$ field $:P_+ P_-:$.
\end{corollary}

\begin{proof} Since $\cC_{\ell}(2n) = \cD_{\ell}(2n)^{\mathbb{Z}_2}$, it is the direct sum of $4$ components:
\begin{enumerate}
\item $\cW_k(\gs\go_{2n})^{\mathbb{Z}_2} \otimes \cW_m(\gs\go_{2n})^{\mathbb{Z}_2}$,
\item The $\cW_k(\gs\go_{2n})^{\mathbb{Z}_2} \otimes \cW_m(\gs\go_{2n})^{\mathbb{Z}_2}$-module generated by $H$,
\item The span of elements of the form $A_{(r)} B_{(s)} 1$ where $A$ and $B$ are in the $-1$-eigenspaces of $\cW_m(\gs\go_{2n})$ and $\cW_k(\gs\go_{2n})$, respectively,
\item The span of elements of the form $A_{(r)} B_{(s)} H$ where $A$ and $B$ are in the $-1$-eigenspaces of $\cW_m(\gs\go_{2n})$ and $\cW_k(\gs\go_{2n})$, respectively.
\end{enumerate}
It is apparent that $A$ and $B$ can be written linear combinations of elements $:A_i \partial^i P_+:$ and $:B_i \partial^iP_-:$, respectively, where $A_i \in \cW_m(\gs\go_{2n})^{\mathbb{Z}_2}$ and $B_i \in \cW_k(\gs\go_{2n})^{\mathbb{Z}_2}$. Any element of the form $$(:A_i \partial^i P_+:)_{(r)} (:B_i \partial^i P_-):_{(s)} 1,$$ lies in the $\cW_k(\gs\go_{2n})^{\mathbb{Z}_2} \otimes \cW_m(\gs\go_{2n})^{\mathbb{Z}_2}$-module generated by elements of the form $$:(\partial^i P_+) \partial^j P_-:\ = ((L_+)_{(0)})^i ((L_-)_{(0)})^j (:P_+ P_-:),$$ so it lies in the $\cW_k(\gs\go_{2n})^{\mathbb{Z}_2} \otimes \cW_m(\gs\go_{2n})^{\mathbb{Z}_2}$-module generated by $:P_+ P_-:$. The same argument shows that $A_{(r)} B_{(s)} H$ lies in the $\cW_k(\gs\go_{2n})^{\mathbb{Z}_2} \otimes \cW_m(\gs\go_{2n})^{\mathbb{Z}_2}$-module generated by $(:P_+ P_-:)_{(s)} H$ for $s \leq 1$, which completes the proof. \end{proof}

Recall that for generic $\ell$, $\cC^{\ell}(2n)$ is simple and strongly generated by $\{H, X^{2i}, Y^{2i}, H^{2i}, W^{2i+1}|\ i \geq 1\}$. Also, for generic $\ell$ the first normally ordered relation among these generators and their derivatives appears at weight $4n+2$, so there are only finitely many levels $\ell$ such that there is a normally ordered relation among these generators and their derivatives in weight $d \leq 2n$. These levels are just the roots of determinant of the matrix associated to the bilinear form $B(u,v) = u_{(4n-1)} v$ on the weight $2n$ space $\cC^{\ell}(2n)[2n]$. For $d \in \mathbb{Z}_{\geq 1}$, let $S_d \subseteq \cC^{\ell}(2n)$ be span of all normally ordered monomials in $\{H, X^{2i},Y^{2i},H^{2i},W^{2i+1}|\ 1\leq i < d\}$ and their derivatives. 
This space is preserved by $H_{(0)}$, so the quotient $\cC^{\ell}(2n)[d] / S_d$ also has a well-defined action of $H_{(0)}$. In weight $2n$, $\text{dim}\ \cC^{\ell}(2n)[2n] / S_{2n} = 3$, and the action of $H_{(0)}$ has eigenvalues  $0,\pm 2$.

By \eqref{evenspinhigherweights}, we have a relation 
$$a_{n,+}W^{2n}_+ + a_{n,-} W^{2n}_- \equiv (b_{n,+} + b_{n,-}) H^{2n} \ \text{mod}\ S_{2n},$$ where $a_{n,\pm}$ and $b_{n,\pm}$ are nonzero constants. Since $H_{(0)} H^{2n} = 0$, $a_{n,+}H_{(0)} W^{2n}_+ \equiv -a_{n,-} H_{(0)} W^{2n}_- \ \text{mod}\ S_{2n}$. We define
\begin{equation} U^{2n} = H_{(0)} W^{2n}_+ . \end{equation} 
\begin{lemma} For generic values of $\ell$, we have a relation
\begin{equation} \label{u2n:first} U^{2n} \equiv h(\ell) :P_+ P_-:  \ \text{mod}\ S_{2n},\end{equation} for some nonzero rational function $\mu(\ell)$.
\end{lemma}

\begin{proof} We will use the same notation for the fields in $\cC^{\ell}(2n)$ and the corresponding fields in the large level limit $\cF(4n)^{\text{O}_{2n}}$, and we also use the notation $S_d$ as above for the corresponding subspace of $\cF(4n)^{\text{O}_{2n}}[d]$. Observe that in $\cF(4n)^{\text{O}_{2n}}$, we have
$$W^{2n}_+ \equiv \lambda \sum_{i=1}^{2n} :(\partial^{2n-1}\phi^{1,i})\phi^{i,i}: \ \text{mod}\ S_{2n},$$ for some nonzero constant $\lambda$. This follows from Theorem \ref{weakgenWso2nfirst}, because $\cF(2n)^{\text{O}_{2n}} \cong \cW_{-h^{\vee} + 1}(\gs\go_{2n})$, and both $W^{2n}_+$ and $\sum_{i=1}^{2n}:(\partial^{2n-1}\phi^{1,i})\phi^{1,i}:$ do not lie in $S_{2n}$. Applying $H_{(0)}$ yields
$$U^{2n} \equiv \lambda \sum_{i=1}^{2n} :(\partial^{2n-1}\phi^{1,i}) \phi^{2,i}: \ \text{mod}\ S_{2n},$$ in $\cF(4n)^{\text{O}_{2n}}$. Finally, we claim that in $\cF(4n)^{\text{O}_{2n}}$, we have 
$$ :P_+ P_-: \ =  \mu \sum_{i=1}^{2n}:(\partial^{2n-1}\phi^{1,i} )\phi^{2,i}: \ \text{mod}\ S_{2n},$$ for some nonzero constant $\mu$. This can be seen by adapting the methods of \cite[Section 11]{Lin3}, and the details are omitted. In particular, it follows that in $\cF(4n)^{\text{O}_{2n}}$ we have the relation 
$$U^{2n} \equiv \frac{\lambda}{\mu} :P_+ P_-:  \ \text{mod}\ S_{2n}.$$ This implies that the relation \eqref{u2n:first} exists in $\cC^{\ell}(2n)$. \end{proof}

\begin{theorem} Suppose that $\ell-2$ is an admissible level for $\gs\go_{2n}$, and is also chosen so that 
\begin{enumerate}
\item $(2n-2) + \ell \notin \mathbb{Z}_{\geq 1}$,
\item The pairing on the weight $2n$ space of $\cC^{\ell}(2n)$ is nondegenerate, so there are no normally ordered relations in weight $d \leq 2n$. 
\end{enumerate} Then $\cC_{\ell}(2n)$ is weakly generated by $\{H, X^2, Y^2, H^2\}$, and hence is a quotient of $\cW^{\fr{so}_2}_{\infty}$.

\end{theorem}

\begin{proof}
First, we claim that for each $d < 2n$, the span of normally ordered monomials in $\{H, X^{2i},Y^{2i},H^{2i},W^{2i+1}\}$ and their derivatives of total weight $d$, coincides with the weight $d$ subspace $\cC_{\ell}(2n)[d]$. Otherwise, there would be a relation among these monomials, contradicting our hypotheses. So all elements of $\cC_{\ell}(2n)[d]$ for $d < 2n$, can be generated from $\{H, X^2, Y^2, H^2\}$.

Next, regarding \eqref{u2n:first} as a relation over the field of rational functions of $\ell$, we can clear denominators and also divide by any common linear factors appearing in the structure constants, obtaining a relation 
\begin{equation} \label{u2n:second} a(\ell) U^{2n} = b(\ell) :P_+ P_-: + Q,\end{equation}
 where $a(\ell), b(\ell)$ are polynomials in $\ell$, and $Q$ is a normally ordered polynomial in $\{H, X^{2i},Y^{2i},H^{2i},W^{2i+1}|\ 1\leq i < n\}$ and their derivatives, where all coefficients are polynomials in $\ell$. Moreover, these coefficients, together with $a(\ell), b(\ell)$, have no common roots. Specializing to a fixed level $\ell$,  there are four possibilities:
\begin{enumerate}
\item $a(\ell) = 0 = b(\ell)$,
\item $a(\ell) = 0$ and $b(\ell) \neq 0$, 
 \item $a(\ell)  \neq 0$ and $b(\ell) = 0$, 
  \item $a(\ell)  \neq 0$ and $b(\ell) \neq 0$.
  \end{enumerate}
  
 In Case (1), we would have $Q = 0$, which yields a nontrivial relation in weight $2n$ among the generators $\{H, X^{2i},Y^{2i},H^{2i},W^{2i+1}|\ 1\leq i < n\}$, which contradicts our assumptions. If $a(\ell) = 0$ and $b(\ell) \neq 0$, then we have a relation $:P_+ P_-: \ = -\frac{1}{b(\ell)} Q$. Since $Q \in S_{2n}$, this implies that $:P_+ P_-:$ is generated by $\{H, X^2, Y^2, H^2\}$, so by Corollary \ref{first:orbifoldgen} our conclusion holds. If $a(\ell) \neq 0$ and $b(\ell) = 0$, then we have a relation $U^{2n} = \frac{1}{a(\ell)} Q$. Recall from \eqref{evenspinhigherweights} that 
 $$a_{n,\pm} W^{2n}_{\pm} = b_{n,\pm} H^{2n} \pm  (X^{2n} + Y^{2n})\ \text{mod}\ S_{2n},$$ where $a_{n,\pm}, b_{n,\pm} \neq 0$. Applying $H_{(0)}$, we obtain 
$$a_{n,+}U^{2n} \equiv 2b_{n,+} (X^{2n} - Y^{2n}) \ \text{mod}\ S_{2n}.$$
Since $U^{2n} = \frac{1}{a(\ell)} Q \equiv 0\ \text{mod}\ S_{2n}$, we have $X^{2n} \equiv Y^{2n}\ \text{mod}\ S_{2n}$. Since $H_{(0)}$ preserves $S_{2n}$, and $X^{2n}, Y^{2n}$ have $H_{(0)}$-eigenvalues $2,-2$, this can only happen if $X^{2n} \equiv 0\ \text{mod}\ S_{2n}$ and $Y^{2n}\equiv 0\ \text{mod}\ S_{2n}$. But then we have $a_{n,\pm} W^{2n}_{\pm} \equiv b_{n,\pm} H^{2n}\ \text{mod}\ S_{2n}$, so $\frac{a_{n,+}}{b_{n,+}}W^{2n}_{+} \equiv \frac{a_{n,-}}{b_{n,-}}W^{2n}_{-}\ \text{mod}\ S_{2n}$, which is impossible.  


Therefore we are left with Case (4), which says that we have a relation
$$  :P_+ P_-: \ = \frac{a(\ell)}{b(\ell)} U^{2n}  -  \frac{1}{b(\ell)}Q.$$ Since $W^{2n}_{\pm}$ are generated by $\{H, X^2, Y^2, H^2\}$ and $U^{2n}= H_{(0)} W^{2n}_+$, $:P_+ P_-:$ lies in the subalgebra generated by $\{H, X^2, Y^2, H^2\}$, and we are done by Corollary \ref{first:orbifoldgen}.
\end{proof}

\subsection{More rational quotients of $\cW^{\fr{so}_2}_{\infty}$ and some further conjectures} 

In addition to $\cW^{\gs\go_2}_{\infty, M_{\ell}}$, which is the simple quotient of $\tilde{\cC}_{\ell}(2n)$ when $\ell -2$ is admissible for $\gs\go_{2n}$, $\cW^{\fr{so}_2}_{\infty}$ is expected to have another large family of quotients which are strongly rational. Recall from \cite{CL4} that 
\begin{enumerate} 
\item $\cW^{k}(\go\gs\gp_{1|2n})^{\mathbb{Z}_2} \cong \cC^{\psi}_{2B}(0,n)$ for $\psi = k = n+\frac{1}{2}$,
\item $\text{Com}(V^{\ell}(\gs\go_{2n+1}), V^{\ell-1}(\gs\go_{2n+1}) \otimes \cF(2n+1))^{\mathbb{Z}_2} \cong  \cC^{\phi}_{2B}(n,0)$ for $\ell = -2 \phi - 2n + 2$.
\item We have the isomorphism \begin{equation} \label{isoGKO:typeB} \cC^{\psi}_{2B}(n,0) \cong \cC_{2B}^{\phi}(0,n),\qquad \frac{1}{\psi} + \frac{1}{\phi} = 2,\end{equation} of $1$-parameter vertex algebras, which can be extended to the coset realization
$$\cW^{k}(\go\gs\gp_{1|2n}) \cong \text{Com}(V^{\ell}(\gs\go_{2n+1}), V^{\ell-1}(\gs\go_{2n+1}) \otimes \cF(2n+1)),$$ where $k$ and $\ell$ are related as above.
\end{enumerate}
Moreover, when $\ell-1$ is admissible for $\gs\go_{2n+1}$, we have the embedding of simple vertex algebras $L_{\ell}(\gs\go_{2n+1}) \hookrightarrow L_{\ell-1}(\gs\go_{2n+1}) \otimes \cF(2n+1))$, and we have the isomorphism of simple vertex algebras 
$$\cW_{k}(\go\gs\gp_{1|2n}) \cong \text{Com}(L_{\ell}(\gs\go_{2n+1}), L_{\ell-1}(\gs\go_{2n+1}) \otimes \cF(2n+1)).$$
In addition, \cite[Conjecture 7.1]{CL4} says that $\cC^{\psi}_{2B}(n,0)$ should be strongly rational whenever $\ell-1$ is admissible for $\gs\go_{2n+1}$, which would imply the strong rationality of $\cW_{k}(\go\gs\gp_{1|2n})$. By the same argument as the proof of Lemma \ref{rationalfamily}, this would imply

\begin{conjecture} \label{rationalfamilysecond} For $\ell-2$ an admissible level for $\gs\go_{2n+1}$, 
$$\cC_{\ell}(2n+1) = \text{Com}(L_{\ell}(\gs\go_{2n+1}),  L_{\ell-2}(\gs\go_{2n+1}) \otimes \cF(2(2n+1)))^{\mathbb{Z}_2},$$ which is an extension of
$$\text{Com}(L_{\ell}(\gs\go_{2n+1}), L_{\ell-1}(\gs\go_{2n+1}) \otimes \cF(2n+1))^{\mathbb{Z}_2} \otimes \text{Com}(L_{\ell-1}(\gs\go_{2n+1}), L_{\ell-2}(\gs\go_{2n+1}) \otimes \cF(2n+1))^{\mathbb{Z}_2},$$
is strongly rational.
\end{conjecture}
As in the case of $\cC_{\ell}(2n)$, it is possible for the image of $\cW^{\fr{so}_2}_{\infty}$ in $\cC_{\ell}(2n+1)$, which is the subalgebra $\tilde{\cC}_{\ell}(2n+1)$ generated by $\{H, X^2, Y^2, H^2\}$, to be a proper subalgebra and to also be non-simple. For $\ell -2$ admissible for $\gs\go_{2n+1}$, let $N_{\ell} \subseteq \mathbb{C}[c,t]$ be the maximal ideal generated by 
$ c - \frac{2 (-2 + t) n (-2 + t + 4 n)}{(-3 + t + 2 n) (-1 + t+ 2 n)}$ and $t - \ell$, so that the simple quotient $\cW^{\gs\go_2}_{\infty,N_{\ell}}$ of $\cW^{\fr{so}_2}_{\infty} / N_{\ell} \cdot \cW^{\fr{so}_2}_{\infty}$ coincides with the simple quotient of $\tilde{\cC}_{\ell}(2n+1)$. As above, we need a criterion for when $\cW^{\gs\go_2}_{\infty,N_{\ell}}$ is strongly rational. 

Recall that $\cW^{k}(\go\gs\gp_{1|2n})$ is freely generated of type $\cW(2,4,\dots, 2n; n+\frac{1}{2})$, and has an action of $\mathbb{Z}_2$ that fixes the generators in weights $2,4,\dots, 2n$ and acts by $-1$ on the odd field $P$ of weight $n+\frac{1}{2}$. The orbifold $\cW^{k}(\go\gs\gp_{1|2n})^{\mathbb{Z}_2}$ is strongly generated by the generators in weights $2,4,\dots, 2n$, together with the composite fields
$$\omega^{2n+2d+2} =\ : \partial^{2d+1} P P:,\qquad d \geq 0.$$

Recall that $\cW^{k}(\go\gs\gp_{1|2n})^{\mathbb{Z}_2}$ is a $1$-parameter quotient of $\cW^{\text{ev}}_{\infty}$, so that for generic values of $k$, it is weakly generated by the fields in weights $2$ and $4$. As before, we take the weight $2$ field to be the Virasoro field $L$, the weight $4$ field to be primary and normalized as in \cite{KL}, and $\{W^{2i}|\ i= 3,\dots, n-1\}$ to be defined recursively by $W^{2i} = W^4_{(1)} W^{2i-2}$. We need to determine the values of $k$ for which  $\cW^{k}(\go\gs\gp_{1|2n})^{\mathbb{Z}_2}$ is {\it not} weakly generated by $L, W^4$. First, we have $W^4_{(1)} W^{2n} = p(k) \omega^{2n+2}$ for some nontrivial polynomial $p(k)$, which we are unable to compute in general.

\begin{theorem}  \label{weakgenWosp1|2nfirst} $\cW^{k}(\go\gs\gp_{1|2n})^{\mathbb{Z}_2}$ is weakly generated by $L, W^4$ for all noncritical $k$ except for the following:
\begin{enumerate}
\item The roots of $p(k)$,
\item $k = -(n+\frac{1}{2}) + \frac{d-1}{2d-1}$ and $k = -(n+\frac{1}{2}) + \frac{2d-1}{4(d-1)}$ for integers $d \geq 2$,
\item The two roots of $4 k^2 (4 n^2 -7) +2 k (16 n^3- 28 n -7 )+ 16 n^4  - 28 n^2 - 14 n -7$.
\end{enumerate}
Therefore at these values of $k$, the simple quotient $\cW_{k}(\go\gs\gp_{1|2n})^{\mathbb{Z}_2}$ is weakly generated by $L, W^4$ as well.
\end{theorem}

\begin{proof}
The Poisson bracket on Zhu's commutative algebra $R_{\cW^{k}(\go\gs\gp_{1|2n})}$ given by $\{\bar{a},\bar{b}\} = \overline{a_{(0)} b}$ is trivial, and the refined Poisson bracket given by  $\{\bar{a},\bar{b}\} = \overline{a_{(1)} b}$ is nontrivial and s independent of $k$. Moreover, all the fields $W^{2i}$ for $2 < i \leq n$ can be generated by $L, W^4$ for noncritical values of $k$. 

As long as $k$ is not a root of $p(k)$, $\omega^{2n+2}$ is also generated by $L, W^4$. Similarly, for $r > n+1$, we have a relation
 $$W^{2r} = \lambda_{2r} \omega^{2r} + \cdots,$$ where the remaining terms depend only on $L, W^4,\dots, W^{2s}$ for $s < r$, and $\lambda_{2r}$ is a rational function of $k$ which is nontrivial since $\cW^{k}(\go\gs\gp_{1|2n})^{\mathbb{Z}_2}$ is generated by $L, W^4$ for generic values of $k$. Using the OPE $W^4(z) P(w)$ which is completely determined in \cite{CL4}, for $d \geq 2$ we can compute

\begin{equation} \begin{split} & W^4_{(1)} \omega^{2d-2+2n}  = \frac{32 (d + n) (1 + 2 k + 2 n) (-1 - 4 k + 4 d k - 4 n + 4 d n) (1 - 
   2 k + 4 d k - 2 n + 4 d n)}{21 (-1 + d) (-1 + 2 d) (-1 + 2 n) (-1 + 4 k n + 4 n^2)  f(n,k) g(n,k)} \omega^{2d+2n} + \dots,
     \\ & f(n,k) =  1 + 2 k + 2 n + 4 k n + 4 n^2,
     \\ & g(n,k) =  4 k^2 (4 n^2 -7) +2 k (16 n^3- 28 n -7 )+ 16 n^4  - 28 n^2 - 14 n -7.
     \end{split} \end{equation}
In particular, this shows that $\lambda_{2d+2n} \neq 0$ except for the levels when the above rational functions vanish or are undefined.
\end{proof}

\begin{corollary} \label{weakgenWosp1|2n} Suppose that $\ell-2$ is an admissible level for $\gs\go_{2n+1}$ such that
\begin{enumerate}
\item $ 2n-1 + \ell \neq 2 d -1$ for $d \in \mathbb{Z}_{\geq 2}$
\item $p(k) \neq 0$ where $k,\ell$ are related as above.
\end{enumerate}
 Then $\cC_{\psi,2B}(n,0) = \text{Com}(L_{\ell}(\gs\go_{2n+1}), L_{\ell-1}(\gs\go_{2n+1}) \otimes \cF(2n+1))^{\mathbb{Z}_2}$ is weakly generated by $\{L, W^4\}$. \end{corollary} 

\begin{proof} This is apparent because the isomorphism \eqref{isoGKO:typeB} holds generically, and the corresponding isomorphism of simple quotients holds whenever $\cW_{k}(\go\gs\gp_{1|2n})^{\mathbb{Z}_2}$ and $\cC_{\psi,2B}(n,0)$ are both quotients of the specialization of the $1$-parameter algebras. This is always the case for $\cW_{k}(\go\gs\gp_{1|2n})^{\mathbb{Z}_2}$ when $k$ is noncritical, and by \cite[Thm. 8.1]{CL2}, it holds for $\cC_{\psi,2B}(n,0)$ whenever $\ell +2n-1$ is not a negative rational number.
\end{proof}

By the same argument as the proof of Corollary \ref{cor:rationalityatell}, Corollary \ref{weakgenWosp1|2n} and Conjecture \ref{rationalfamilysecond} implies the following
\begin{conjecture} \label{cor:rationalityatellsecond} For all admissible levels $\ell-2$ for $\gs\go_{2n+1}$ such that $-(2n-2) + \ell \notin \mathbb{Z}_{\geq 1}$, $\cW^{\gs\go_2}_{\infty, N_{\ell}}$, which is isomorphic simple quotient of $\tilde{\cC}_{\ell}(2n)$, is strongly rational.
\end{conjecture}

Finally, we observe that the truncation curves for $\cC^{\psi}_{CO}(0,m)$ and $\cC^{\ell}(2n+1)$ intersect at
$$\ell = -(2n-1)+\frac{2m+n}{m},\qquad \psi = \frac{2m+n}{2m}.$$ By the same argument as the proof of Theorem \ref{Rationality:COcase}, Conjecture \ref{cor:rationalityatellsecond} implies the following.

\begin{conjecture} Suppose that $m,n$ are coprime. 
For $m \geq 2n-1$ and $\psi = \frac{2m+n}{2m}$, $\cC_{\psi,CO}(0,m)$ is strongly rational. This would imply the strong rationality of $\cW_{\frac{n+2m}{2m}- (2m+\frac{1}{2})}(\go\gs\gp_{1|4m}, f_{2m,2m})$.
\end{conjecture}

\appendix

\section{Generators of commuting even-spin subVOAs of weights four}\label{app:wt4}
Weight four generators admit the following form,
\begin{equation*}
    \begin{split}
        W^4_{\pm}=(k-3) (k-2) (k-1) (k+2)\omega_{\pm}\sum_{\Omega\in PBW[4]} \xi_{\Omega}\Omega,
    \end{split}
\end{equation*}
where the $\T{PBW}[4]$ are the 28 monomials spanning the subspace $\Wso[4]$.
The structure constants $\{\xi_{\Omega}\}_{\Omega\in \T{PBW}[4]}$ are as follows,
\begin{equation*}
    \begin{split}
        \xi_{W^{4,0}}=&1,\\
        \xi_{W^{4,-2}}=&\frac{(k-2) (k-1) \omega  \omega _3 N_{W^{4,-2}}}{2 (\omega -1) (k \omega -1) (k \omega +1)D_2 }, \quad \xi_{W^{4,2}}=\xi_{W^{4,-2}},\\
        \xi_{:HW^{3,0}:}=&\frac{(k-2) (k-1) (k \omega -3 \omega +1) (k \omega +\omega +3) N_{:HW^{3,0}:}}{(\omega -1) \omega_4 (k \omega -1) (k \omega +1)D_2 }\\
        \xi_{\partial W^{3,0}}=&\frac{(k-8) (k-2) (k-1) (k \omega -3 \omega +1) (k \omega +\omega +3) (k \omega+2 \omega +2) N_{:W^{2,2}W^{2,-2}:}}{4(\omega -1) (k \omega -1) (k \omega +1) \omega_4 D_1 D_2},\\
        \xi_{:W^{2,2}W^{2,-2}:}=&\frac{3 (k-8) k^2 \omega ^2 (k \omega -3 \omega +1) (k \omega +\omega +3) \omega _2 \omega_3^2 N_{:W^{2,-2}W^{2,2}:}}{ (\omega -1)(k \omega -1) (k \omega +1) (k \omega +2 \omega -2)\omega _4 D_1 D_2}\\
        \xi_{:W^{2,0}W^{2,0}:}=&\frac{3 (k-8) k^2 (\omega -1) (\omega +1)^2(k \omega -3 \omega +1) (k \omega +\omega +3)  \omega _2 N_{:W^{2,0}W^{2,0}:}}{4  (k \omega -1)(k \omega +1) (k \omega +2 \omega -2) \omega _4 D_1 D_2}\\
        \xi_{:W^{2,0}W^{2,-2}:}=&\frac{3 (k-8) k^2 \omega  (\omega +1) (k \omega -3 \omega +1) (k \omega +\omega +3) (11 k \omega -16 \omega -6) \omega _2 \omega _3 N_{W^{4,-2}}}{ (k \omega -1) (k \omega +1) (k\omega +2 \omega -2)\omega _4 D_1D_2^2 },\quad \xi_{:W^{2,0}W^{2,2}:}=\xi_{:W^{2,0}W^{2,-2}:},\\
        \xi_{:W^{2,-2}W^{2,-2}:}=&-\frac{3 k^2 \omega ^2 (k \omega +\omega +3) (5 k \omega -12 \omega -7)\omega _2 \omega _3^2  N_{W^{4,-2}}}{2(\omega -1)  (k \omega -1)(k \omega +1) (k \omega +2 \omega -2)\omega _4 D_1 D_2 },\quad \xi_{:W^{2,2}W^{2,2}:}=\xi_{:W^{2,-2}W^{2,-2}:},\\
        \xi_{:HHW^{2,0}:}=&-\frac{3 (k-8) k (\omega +1) (k \omega -3 \omega +1) (k \omega +\omega +3) \omega _2N_{W^{2,-2}W^{2,2}}}{(k \omega -1) (k \omega +1)(k \omega +2 \omega -2)\omega _4 D_1 D_2 },\\
        \xi_{:H\partial W^{2,0}:}=&-\frac{3 (k-8) k (\omega +1) (k \omega -3 \omega +1) (k \omega +\omega+3)\omega _2  N_{:W^{2,2}W^{2,-2}:}}{2 (k \omega -1) (k \omega+1) (k \omega +2 \omega -2)\omega _4 D_1 D_2 },\quad \xi_{:\partial HW^{2,0}:}=\xi_{: H\partial W^{2,0}:},\\
        \xi_{\partial^2W^{2,0}}=&\frac{9 (k-8) k^2 (\omega +1) (k \omega -3 \omega +1) (k \omega +\omega +3) \omega _2 N_{\partial^2W^{2,0}}}{1(k \omega -1) (k \omega +1) (k\omega +2 \omega -2) 6\omega _4 D_1 D_2 },\\
        \xi_{:HHW^{2,-2}:}=&\frac{6 k \omega (k \omega -3 \omega +1) (k \omega +\omega +3)\omega _2 \omega _3  N_{W(4,-1)} N_{:HHW^{2,0}:}}{ (\omega -1)  (k\omega -1) (k \omega +1) (k \omega +2 \omega -2)\omega _4D_1 D_2^2},\quad  \xi_{:HHW^{2,2}:}= \xi_{:HHW^{2,-2}:},\\
        \xi_{:H\partial W^{2,-2}:}=&-\frac{3 (k-8) k^2 \omega  (k \omega -3 \omega +1) (k \omega +\omega +3)  \omega _2 \omega _3 N_{W^{4,-2}} N_{:H\partial W^{2,-2}:}}{2(\omega -1)(k \omega -1) (k \omega +1) (k \omega +2 \omega -2) \omega _4 D_1 D_2^2 },\quad \xi_{:H\partial W^{2,2}:}=-\xi_{:H\partial W^{2,-2}:},\\
        \xi_{:\partial HW^{2,-2}:}=&\frac{3 k \omega (k \omega -3 \omega +1) (k \omega +\omega +3) \omega _2 \omega _3  N_{W^{4,-2}} N_{\partial H W^{2,-2}}}{(\omega -1)(k\omega -1) (k \omega +1) (k \omega +2 \omega-2) \omega _4 D_1 D_2^2 },\quad \xi_{:\partial HW^{2,2}:}=-\xi_{:\partial HW^{2,-2}:},\\
        \xi_{:\partial^2W^{2,-2}:}=&-\frac{3 (k-8) k^2 \omega (k \omega -3 \omega +1) (k \omega +\omega +3)  \omega _2 \omega _3 N_{W^{4,-2}} N_{\partial^2W^{2,-2}}}{4 (\omega -1) (k \omega -1) (k \omega +1) (k \omega +2 \omega-2) \omega_4D_1 D_2^2 },\quad  \xi_{:\partial^2W^{2,2}:}= \xi_{:\partial^2W^{2,-2}:},\\
        \xi_{:HHHH:}=&-\frac{(k \omega -3 \omega +1) (k \omega -\omega +1) (k \omega +\omega +3)\omega _2 N_{:HHHH:}}{ (\omega -1)  (k \omega -1) (k \omega+1) (k \omega +2 \omega -2)\omega _4D_1 D_2},\\
        \xi_{:\partial HHH:}=&\frac{3 (k-8) (k \omega -3 \omega +1) (k \omega -\omega -1) (k \omega-\omega +1) (k \omega +\omega +3) \omega _2 N_{:W^{2,2}W^{2,-2}:}}{(\omega -1)(k \omega-1) (k \omega +1) (k \omega +2 \omega -2) \omega _4 D_1 D_2 },\\
        \xi_{:\partial^2HH:}=&\frac{k(k \omega -3 \omega +1) (k \omega -\omega +1) (k \omega +\omega +3)  \omega _2 N_{:\partial^2HH:}}{4 (\omega -1) (k \omega -1) (k \omega +1) (k\omega +2 \omega -2)\omega _4 D_1 D_2 },\\
        \xi_{:\partial H\partial H:}=&-\frac{3 k (k \omega -3 \omega +1) (k \omega -\omega +1) (k \omega +\omega +3) \omega _2 N_{:\partial H\partial H:}}{8 (\omega-1) (k \omega -1) (k \omega +1) (k \omega +2 \omega -2)\omega _4 D_1 D_2},\\
        \xi_{\partial^3 H}=&\frac{(k-8) k^2 (k \omega -3 \omega +1) (k \omega -\omega -1) (k \omega-\omega +1) (k \omega +\omega+3)\omega _2 N_{:W^{2,2}W^{2,-2}:}}{32 (\omega -1) (k \omega -1) (k \omega +1) (k \omega +2 \omega -2) \omega _4D_1 D_2}.
    \end{split}
\end{equation*}
Denominators and numerators are as follows,
\begin{equation}
    \begin{split}
        D_1=&5 k^3 \omega ^2-15 k^2 \omega ^2+10 k \omega ^2-5 k \omega -5 k-44 \omega -44,\\
        D_2=&5 k^4 \omega ^2-14 k^3 \omega ^2+13 k^2 \omega ^2-5 k^2+40 k \omega ^2-40 k-64 \omega ^2+64,
    \end{split}
\end{equation}
\begin{equation*}
    \begin{split}
   N_{:W^{2,2}W^{2,-2}:}=& 25 k^7 \omega ^4-120 k^6 \omega ^4-60 k^6 \omega ^3+120 k^5 \omega ^4+265 k^5 \omega ^3+10 k^5 \omega ^2+340 k^4 \omega ^4-405 k^4 \omega ^3\\&-135 k^4\omega ^2+60 k^4 \omega -703 k^3 \omega ^4-39 k^3 \omega ^3+238 k^3 \omega ^2+119 k^3 \omega -35 k^3-186 k^2 \omega ^4\\&+379 k^2 \omega ^3+575 k^2\omega ^2-319 k^2 \omega -129 k^2+736 k \omega ^4+102 k \omega ^3-870 k \omega ^2-102 k \omega \\&+134 k+192 \omega ^4-384 \omega ^2+192,\\
   N_{:\partial^2 H H:}=&5 k^{10} \omega ^5+30 k^9 \omega ^5-5 k^9 \omega ^4-795 k^8 \omega ^5-170 k^8 \omega ^4-10 k^8 \omega ^3+3456 k^7 \omega ^5+2355 k^7 \omega ^4\\&+114 k^7\omega ^3+10 k^7 \omega ^2-4099 k^6 \omega ^5-8936 k^6 \omega ^4-1306 k^6 \omega ^3+166 k^6 \omega ^2+5 k^6 \omega -7792 k^5 \omega ^5\\&+12447 k^5\omega ^4+5816 k^5 \omega ^3-782 k^5 \omega ^2-144 k^5 \omega -5 k^5+22817 k^4 \omega ^5+3032 k^4 \omega ^4-6526 k^4 \omega ^3\\&-3636 k^4 \omega ^2+589
   k^4 \omega +4 k^4-5582 k^3 \omega ^5-18389 k^3 \omega ^4-19932 k^3 \omega ^3+8210 k^3 \omega ^2\\
   &+5674 k^3 \omega -61 k^3-24552 k^2 \omega ^5-2150 k^2\omega ^4+41360 k^2 \omega ^3+10380 k^2 \omega ^2-9128 k^2 \omega \\&-2470 k^2+11904 k \omega ^5+14016 k \omega ^4-6912 k \omega ^3-15360 k \omega^2-4992 k \omega +1344 k\\&+4608 \omega ^5+3456 \omega ^4-9216 \omega ^3-6912 \omega ^2+4608 \omega +3456,\\
   N_{:\partial H \partial H:}=&5 k^{10} \omega ^5-45 k^9 \omega ^5-5 k^9 \omega ^4+240 k^8 \omega ^5+85 k^8 \omega ^4-10 k^8 \omega ^3-744 k^7 \omega ^5-1020 k^7 \omega ^4-96 k^7\omega ^3\\&+10 k^7 \omega ^2+1001 k^6 \omega ^5+4114 k^6 \omega ^4+1604 k^6 \omega ^3+16 k^6 \omega ^2+5 k^6 \omega +617 k^5 \omega ^5-6591 k^5 \omega^4\\&-6358 k^5 \omega ^3-164 k^5 \omega ^2+141 k^5 \omega -5 k^5-3766 k^4 \omega ^5+293 k^4 \omega ^4+11018 k^4 \omega ^3+3048 k^4 \omega ^2\\&-1052 k^4
   \omega -101 k^4+4060 k^3 \omega ^5+8788 k^3 \omega ^4-5436 k^3 \omega ^3-6940 k^3 \omega ^2-1184 k^3 \omega +392 k^3\\&-792 k^2 \omega ^5-1820 k^2
   \omega ^4-3568 k^2 \omega ^3-1128 k^2 \omega ^2+4360 k^2 \omega +1028 k^2-576 k \omega ^5-5520 k \omega ^4\\&+1152 k \omega ^3+6816 k \omega ^2-576 k
   \omega -1296 k-1152 \omega ^4+2304 \omega ^2-1152,\\
   N_{\partial^2 W^{2,0}}=&5 k^7 \omega ^4-24 k^6 \omega ^4-12 k^6 \omega ^3+24 k^5 \omega ^4+53 k^5 \omega ^3+2 k^5 \omega ^2+44 k^4 \omega ^4-81 k^4 \omega ^3-3 k^4 \omega ^2+12k^4 \omega\\&-89 k^3 \omega ^4+45 k^3 \omega ^3-4 k^3 \omega ^2-29 k^3 \omega -7 k^3+28 k^2 \omega ^4-47 k^2 \omega ^3+21 k^2 \omega ^2+59 k^2 \omega+3 k^2\\&+12 k \omega ^4+62 k \omega ^3+2 k \omega ^2-62 k \omega -14 k+24 \omega ^3-24 \omega ^2-24 \omega +24,\\
   N_{:HHHH:}=&25 k^9 \omega ^5-90 k^8 \omega ^5-40 k^8 \omega ^4+165 k^7 \omega ^5+170 k^7 \omega ^4-50 k^7 \omega ^3-780 k^6 \omega ^5-1410 k^6 \omega ^4\\&-390 k^6\omega ^3+80 k^6 \omega ^2+2737 k^5 \omega ^5+6812 k^5 \omega ^4+4318 k^5 \omega ^3+428 k^5 \omega ^2+25 k^5 \omega -2384 k^4 \omega ^5\\&-14652 k^4
   \omega ^4-15176 k^4 \omega ^3-2548 k^4 \omega ^2+480 k^4 \omega -40 k^4-6719 k^3 \omega ^5+7822 k^3 \omega ^4\\&+27706 k^3 \omega ^3+11016 k^3 \omega
   ^2-2827 k^3 \omega -598 k^3+13910 k^2 \omega ^5+15950 k^2 \omega ^4-18996 k^2 \omega ^3\\&-25220 k^2 \omega ^2-1954 k^2 \omega +2230 k^2-4560 k \omega^5-15220 k \omega ^4-5408 k \omega ^3\\&+13992 k \omega^2+9968 k \omega +1228 k-2304 \omega ^5-5088 \omega ^4+4608 \omega ^3+10176 \omega ^2-2304,\\
   N_{:HW^{3,0}:}=&5 k^6 \omega ^3-5 k^5 \omega ^2-28 k^4 \omega ^3-27 k^4 \omega ^2-5 k^4 \omega +60 k^3 \omega ^3+63 k^3 \omega ^2+12 k^3 \omega +5 k^3+49 k^2 \omega
   ^3\\&-51 k^2 \omega ^2-93 k^2 \omega +15 k^2-130 k \omega ^3-46 k \omega ^2+130 k \omega +46 k-48 \omega ^3+48 \omega ^2+48 \omega -48\omega -5088,\\
   N_{W^{4,-2}}=&5 k^8 \omega ^4+5 k^7 \omega ^4-100 k^6 \omega ^4-10 k^6 \omega ^2+136 k^5 \omega ^4-76 k^5 \omega ^2+381 k^4 \omega ^4+294 k^4 \omega ^2+5 k^4\\&-825 k^3\omega ^4-526 k^3 \omega ^2+71 k^3-554 k^2 \omega ^4+1364 k^2 \omega ^2-170 k^2+1376 k \omega ^4-1344 k \omega ^2\\&-32 k+384 \omega ^4-768 \omega^2+384,\\
   N_{:HHW^{2,0}:}=& 5 k^7 \omega ^4-45 k^6 \omega ^4+140 k^5 \omega ^4+120 k^5 \omega ^3-10 k^5 \omega ^2-104 k^4 \omega ^4-530 k^4 \omega ^3-96 k^4 \omega ^2-299 k^3\omega ^4\\&+810 k^3 \omega ^3+564 k^3 \omega ^2-120 k^3 \omega +5 k^3+581 k^2 \omega ^4+78 k^2 \omega ^3-1002 k^2 \omega ^2-238 k^2 \omega +141 k^2\\&-182
   k \omega ^4-758 k \omega ^3+214 k \omega ^2+638 k \omega +88 k-96 \omega ^4-204 \omega ^3+396 \omega ^2+204 \omega -300,\\
   N_{:W^{2,0}W^{2,0}:}=&60 k^4 \omega ^2-129 k^3 \omega ^2-132 k^3 \omega -163 k^2 \omega ^2+307 k^2 \omega +72 k^2+338 k \omega ^2-104 k \omega  -102 k\\&+96 \omega ^2-60 \omega
   -36,\\
   N_{:\partial H W^{2,-2}:}=&5 k^6 \omega ^3-40 k^5 \omega ^3-5 k^5 \omega ^2+115 k^4 \omega ^3+20 k^4 \omega ^2-5 k^4 \omega -118 k^3 \omega ^3-37 k^3 \omega ^2-14 k^3 \omega +5
   k^3\\&-192 k^2 \omega ^3-10 k^2 \omega ^2+256 k^2 \omega +34 k^2+672 k \omega ^3+112 k \omega ^2-672 k \omega -112 k-512 \omega ^3\\&-192 \omega ^2+512
   \omega +192,\\
   N_{\partial^2 W^{2,-2}}=&k^4 \omega ^3-4 k^3 \omega ^3-k^3 \omega ^2+5 k^2 \omega ^3+2 k^2 \omega ^2-k^2 \omega +20 k \omega ^3-k \omega ^2-20 k \omega +k-32 \omega ^3-16 \omega
   ^2+32 \omega +16.
    \end{split}
\end{equation*}

\end{document}